\documentclass{article}
\makeatletter
\renewcommand*{\@fnsymbol}[1]{\ensuremath{\ifcase#1\or *\or 1\or 2\or
 3\or 4\or 5\or 6\or 7 \or 8 \else\@ctrerr\fi}}
\makeatother
\usepackage{amsmath,amsthm,amssymb}
\usepackage{upgreek}
\usepackage{cite,yhmath}
\usepackage{enumerate}
\newtheorem{theorem}{Theorem}[section]
\newtheorem{corollary}[theorem]{Corollary}
\newtheorem{lemma}[theorem]{Lemma}
\newtheorem{proposition}[theorem]{Proposition}
\theoremstyle{definition}

\newtheorem{example}[theorem]{Example}

\theoremstyle{remark}
\newtheorem{remark}[theorem]{Remark}

\allowdisplaybreaks[3]
\numberwithin{equation}{section}
\DeclareMathOperator*\esssup{{\mathrm{ess\,sup}}}
\newcommand{\1}{1\!\!1}
\newcommand{\Bbs}{B_{\mathrm{bs}}(\Ga_0)}
\newcommand{\Bb}{\B_{\mathrm{b}}(\X)}
\newcommand{\B}{\mathcal{B}}
\newcommand{\Dom}{\mathrm{Dom}}
\newcommand{\D}{\mathcal{D}}
\newcommand{\Fcyl}{\F_{\mathrm{cyl}}(\Ga)}
\newcommand{\Ffin}{\F_0(\Ga)}
\newcommand{\F}{{\mathcal F}}
\newcommand{\Ga}{\Gamma}
\newcommand{\KK}{\overline{\K_\aC}}
\newcommand{\Ka}{\KK}
\newcommand{\K}{\mathcal{K}}
\newcommand{\La}{\Lambda}
\newcommand{\Mfm}{\M^1_\mathrm{fm}(\Ga)}
\newcommand{\M}{\mathcal{M}}
\newcommand{\N}{{\mathbb N}}
\newcommand{\R}{{\mathbb R}}
\newcommand{\Sect}{\mathrm{Sect}}
\newcommand{\X}{{\R^d}}
\newcommand{\Z}{\mathbb{Z}}
\newcommand{\aC}{{\a C}}

\newcommand{\eps}{\varepsilon}

\newcommand{\ga}{\gamma}
\newcommand{\goto}{\rightarrow}

\newcommand{\hL}{{\widehat{L}}}
\newcommand{\hP}{{\widehat{P}}}

\newcommand{\hT}{{\widehat{T}}}
\newcommand{\inv}{\mathrm{inv}}

\newcommand{\la}{\lambda}
\newcommand{\lluu}{\lu\!\!\lu}

\newcommand{\lu}{\left\langle}

\newcommand{\n}{{(n)}}
\newcommand{\en}{{|\eta|}}

\newcommand{\rruu}{\ru\!\!\ru}

\newcommand{\ru}{\right\rangle}

\newcommand{\xn}{x_1,\ldots,x_n}
\renewcommand\L{{\mathcal L}}
\renewcommand\a{{\alpha}}
\renewcommand{\Im}{\mathrm{Im\,}}
\renewcommand{\L}{\mathcal{L}}
\renewcommand{\Re}{\mathrm{Re\,}}

\begin{document}

\title{Statistical dynamics of continuous systems: perturbative and approximative approaches\thanks{This work was financially supported by the DFG through SFB 701: ``Spektrale
Strukturen und Topologische Methoden in der Mathematik"}}
\author{Dmitri Finkelshtein\thanks{Department of Mathematics,
Swansea University, Singleton Park, Swansea SA2 8PP, U.K. ({\tt d.l.finkelshtein@swansea.ac.uk}).} \and Yuri Kondratiev\thanks{Fakult\"{a}t
f\"{u}r Mathematik, Universit\"{a}t Bielefeld, 33615 Bielefeld,
Germany ({\tt kondrat@math.uni-bielefeld.de}).} \and Oleksandr
Kutoviy\thanks{Department of Mathematics, Massachusetts Institute of Technology,
77 Massachusetts Avenue E18-420, Cambridge, MA, USA ({\tt
kutovyi@mit.edu}); Fakult\"{a}t f\"{u}r Mathematik, Universit\"{a}t
Bielefeld, 33615 Bielefeld, Germany ({\tt
kutoviy@math.uni-bielefeld.de}).}}
\maketitle

\begin{abstract}
We discuss general concept of Markov statistical dynamics in the continuum.  For a class of spatial birth-and-death models, we develope a perturbative technique for the construction of statistical dynamics. Particular examples of such systems are
considered. For the case of Glauber type dynamics in the continuum we describe a Markov chain approximation approach that gives more detailed information about statistical evolution in this model.


\end{abstract}

\section{Introduction}
Dynamics of interacting particle systems appear  in several areas of the complex systems theory. In particular, we observe a growing activity in the study of Markov dynamics for continuous systems. The latter fact is motivated, in particular, by modern problems of mathematical physics, ecology, mathematical biology, and genetics, see e.g. \cite{FKO2009,FKK2010b,FKKL2011a,FKKL2011b,FKK2010,FKKZ2010,FKO2011a,FKO2011b,FKK2010c,FKO2011c,OFKKBC2013,KKM2008,KKP2008,KKZ2006}
and literature cited therein. Moreover, Markov dynamics are used for the construction of social, economic and demographic models.  Note that Markov processes for continuous systems are considering in the stochastic analysis as dynamical point processes
\cite{HS1978,GK2008,GK2006} and they appear even
in the representation theory of big groups \cite{BO2005,BO2005b,BO2000,BO1998,BO2006}.

A mathematical formalization  of the problem may be described as the following.  As a phase space of the system we use the space $\Gamma(\R^d)$ of locally finite configurations in the Euclidean space $\R^d$. An heuristic Markov generator which describes considered model is given by its expression on a proper set of functions (observables) over $\Gamma(\R^d)$. With this operator we can relate two evolution equations. Namely, Kolmogorov backward
equation for observables and Kolmogorov forward equation on probability measures on the phase space $\Gamma(\R^d)$ (macroscopic states of the system). The latter equation is a.k.a. Fokker--Planck equation in the mathematical physics terminology. Comparing with the usual situation in the stochastic analysis, there is an essential technical difficulty: corresponding Markov process in the configuration space may be constructed only
 in very special particular cases. As a result, a description of Markov dynamics in terms of random trajectories is absent for most of models under considerations.

 As an alternative approach we use a concept of the statistical dynamics that substitutes
 the notion of a Markov stochastic process.  A central object now is an evolution of states of the system that shall be defined by mean of the Fokker--Planck equation. This evolution equation w.r.t. probability measures on  $\Gamma(\R^d)$ may be reformulated as a hierarchical chain of  equations for correlation functions of considered measures.
 Such kind of evolution equations are well known in the study of Hamiltonian dynamics for
 classical gases as BBGKY chains but now they appear as a  tool for construction and analysis of  Markov dynamics. As an essential technical step, we consider related pre-dual evolution  chains of equations on the so-called quasi-observables. As it will be shown in the paper, such hierarchical equations may be analyzed in the framework
 of semigroup theory with the use of powerful techniques of perturbation theory for the semigroup generators etc.
 Considering the dual evolution for the constructed semigroup on quasi-observables we introduce then the dynamics on
 correlation functions.
 Described scheme of  the dynamics construction looks quite surprising because any perturbation  techniques for initial Kolmogorov evolution equations one can not expect. The point is that states of infinite interacting particle systems are given by measures which are, in general, mutually orthogonal. As a result,
 we can not compare their evolutions or apply a perturbative approach. But under quite general assumptions they have correlation functions and corresponding dynamics may be considered in a common  Banach space of correlation functions.
Proper choice of this Banach space means, in fact, that we find an admissible class of initial states for which the statistical dynamics may be constructed. There  we see again a crucial difference comparing with Markov stochastic processes framework
where the initial distribution evolution is defined for any initial date.

 The structure of the paper is following.  In Section 2 we discuss general concept of statistical dynamics for Markov evolutions in the continuum and introduce necessary mathematical structures. Then, in Section 3,  this concept is applied to an important class of Markov dynamics of continuous systems, namely, to birth-and-death models. Here
 general conditions for the existence of a semigroup evolution in a space of quasi-observables are obtained.
 Then we construct evolutions of correlation functions as dual objects. It is shown how to apply general results to the study of particular models of statistical dynamics coming from mathematical physics and ecology.

 Finally, in Section 4 we  describe an alternative techniques for the construction of solutions to hierarchical chains evolution
 equations by means of an approximative approach.  For concreteness, this approach is discussed in the
 case of the
 so-called Glauber type dynamics in the continuum. We construct a family of Markov chains on configuration space in finite volumes with concrete transition kernels adopted to the Glauber dynamics. Then the solution to the hierarchical equation for correlation functions may be obtained as the limit of the corresponding object for the Markov chain dynamics. This limiting evolution generates the state dynamics. Moreover,  in the uniqueness regime for the corresponding equilibrium measure of Glauber dynamics which is, in fact, Gibbs, dynamics of correlation functions is exponentially ergodic.

 This paper is based on a series of our previous works \cite{FKK2009,FKK2010,FKK2011a,FKKZ2010,FKK2011c, KKZ2006}  but certain results and constructions are detailed  and generalized, in particular, in more complete analysis of the dual dynamics on correlation functions.

\section{Statistical description for stochastic dynamics of~complex~systems~in~the~continuum}

\subsection{Complex systems in the~continuum}
In recent decades, different brunches of natural and life sciences have been addressing to a unifying point of view on a number of phenomena occurring in systems composed of interacting subunits. This leads to formation of an~interdisciplinary science which is
referred to as the theory of complex systems. It provides reciprocation of concepts and tools involving wide spectrum of applications as well as various mathematical theories such that statistical mechanics, probability, nonlinear dynamics, chaos theory, numerical simulation and many others.

Nowadays complex systems theory is a quickly growing interdisciplinary area with a very broad spectrum of motivations and applications. For instance, having in mind biological applications, S.~Levin \cite{Lev2002} characterized complex adaptive systems by such properties as diversity and individuality of components, localized interactions among components, and the outcomes of interactions used for replication or enhancement of components.
We will use a more general informal description of a complex system as a specific collection of interacting elements which has so-called collective behavior that means appearance of system properties which are not peculiar to inner nature of each element itself.
The significant physical example of such properties is thermodynamical effects which were a~background for creation by L.\,Boltzmann of statistical physics as a~mathematical language for studying complex systems of molecules.

We assume that all elements of a complex system are identical by properties and possibilities. Thus, one can model these elements as points in a proper space whereas the complex system will be modeled as a discrete set in this space. Mathematically this means that for the study of complex systems the proper language and techniques are delivered by the interacting particle models which form a rich and powerful direction in modern stochastic and infinite dimensional analysis. Interacting particle systems have a wide use as models in condensed matter physics, chemical kinetics, population biology, ecology (individual based models), sociology and economics (agent based models). For instance a population in biology or ecology may be represented by a configuration of organisms located in a proper habitat.

In spite of completely different orders of numbers of elements in real physical, biological, social, and other systems (typical numbers start from $10^{23}$ for molecules and, say, $10^5$~for plants) their complexities have analogous phenomena and need similar mathematical methods. One of them consists in mathematical approximation of a huge but finite real-world system by an infinite system realized in an infinite space. This approach was successfully approved to the thermodynamic limit for models of statistical physics and appeared quite useful for the ecological modeling in the infinite habitat to avoid boundary effects in a population evolution.

Therefore, our phase space for the mathematical description should consist of countable sets from an underlying space. This space itself may have discrete or continuous nature that leads to segregation of the world of complex systems on two big classes. Discrete models correspond to systems whose elements can occupy some prescribing countable set of positions, for example, vertices of the lattice $\Z^d$ or, more generally, of some graph embedded to $\R^d$. These models are widely studied and the corresponding theories were realized in numerous publications, see e.g. \cite{Lig1985,Lig1999} and the~references therein. Continuous models, or models in the~continuum, were studied not so intensively and broadly. We concentrate our attention exactly on continuous models of systems whose elements may occupy any points in Eucledian space $\R^d$. (Note that the most part of our results may be easily transferred to much more general underlying spaces.) Having in mind that real elements have physical sizes we will consider only the so-called locally finite subsets of the underlying space $\R^d$, that means that in any bounded region we assume to have a finite number of the elements. Another our restriction will be prohibition of multiple elements in the same position of the space.

We will consider systems of elements of the same type only. The mathematical realization of considered approaches may be successfully extended to multi-type systems, meanwhile such systems will have more rich qualitative properties and will be an object of interest for applications. Some particular results can be found e.g. in \cite{Fin2009,FFK2008,FKO2011c}.

\subsection{Mathematical description for a complex systems}

We proceed to the mathematical realization of complex systems.

Let ${\B}({\X})$ be the family of all Borel sets in ${\X}$, $d\geq
1$; ${\B}_{\mathrm{b}} ({\X})$ denotes the system of all bounded
sets from ${\B}({\X})$.

The configuration space over space $\X$ consists of all locally
finite subsets (configurations) of $\X$. Namely,
\begin{equation*} \label{confspace}
\Ga =\Ga\bigl(\X\bigr) :=\Bigl\{ \ga \subset \X \Bigm| |\ga _\La
|<\infty, \ \mathrm{for \ all } \ \La \in {\B}_{\mathrm{b}}
(\X)\Bigr\}.
\end{equation*}
Here $|\cdot|$ means the cardinality of a~set, and
$\ga_\La:=\ga\cap\La$. We may identify each
$\ga\in\Ga$ with the non-negative Radon measure $\sum_{x\in
\gamma }\delta_x\in \M(\X)$, where $\delta_x$ is
the Dirac measure with unit mass at $x$, $\sum_{x\in\emptyset}\delta_x$ is, by definition, the zero measure, and $\M(\X)$ denotes the space of all
non-negative Radon measures on $\B(\X)$. This identification allows to endow
$\Ga$ with the topology induced by the vague topology on
$\M(\X)$, i.e. the weakest topology on $\Ga$
with respect to which all mappings
\begin{equation}\label{gentop}
    \Ga\ni\ga\mapsto \sum_{x\in\ga} f(x)\in{\R}
\end{equation}
are continuous for any $f\in C_0(\X)$ that is the set of all continuous functions on $\X$ with compact supports.  It is worth noting
the vague topology may be metrizable in such a way that $\Ga$ becomes a~Polish space (see e.g. \cite{KK2006} and references therein).

Corresponding to the vague topology the Borel $\sigma $-algebra $\B(\Ga )$ appears the smallest
$\sigma $-algebra for which all mappings
\begin{equation}\label{gensalg}
    \Ga \ni \ga \mapsto N_\La(\ga):=|\ga_
\La |\in{ \N}_0:={\N}\cup\{0\}
\end{equation}
are measurable for any $\La\in{
\B}_{\mathrm{b}}(\X)$, see e.g. \cite{AKR1998a}. This $\sigma$-algebra may be generated by the~sets
\begin{equation}\label{QLa}
    Q(\La,n):=\bigl\{\ga\in\Ga \bigm| N_\La(\ga)=|\ga_\La|=n\bigr\}, \qquad \La\in\Bb, n\in\N_0.
\end{equation}
Clearly, for any $\La\in\Bb$,
\begin{equation*}\label{unionCLa}
    \Ga=\bigsqcup_{n\in\N_0} Q(\La,n).
\end{equation*}

Among all measurable functions $F:\Ga\to\bar{\R}:=\R\cup\{\infty\}$ we mark out the set $\Ffin$ consisting of such of them for which $|F(\ga)|<\infty$ at least for all $|\ga|<\infty$. The important subset of $\Ffin$ formed by cylindric functions on $\Ga$. Any such a function is characterized by a set $\La\in\Bb$ such that $F(\ga)=F(\ga_\La)$ for all $\ga\in\Ga$. The class of cylindric functions we denote by $\Fcyl\subset\Ffin$.

Functions on $\Ga$ are usually called {\em observables}. This notion is borrowed from statistical physics and means that typically in course of empirical investigation we may estimate, check, see only some quantities of a whole system rather then look on the system itself.

\begin{example}
  Let $\varphi:\X\to\R$ and consider the so-called {\em linear function} on $\Ga$, cf. \eqref{gentop},
  \begin{equation*}\label{linfunc}
      \langle \varphi,\ga \rangle :=
        \begin{cases}
            \displaystyle\sum_{x\in\ga}\varphi(x), & \displaystyle \text{if} \ \sum_{x\in\ga}|\varphi(x)|<\infty, \quad \ga\in\Ga,\\
            \displaystyle+\infty, &\text{otherwise}.
        \end{cases}
  \end{equation*}
  Then, evidently, $\langle \varphi,\cdot \rangle\in\Ffin$. If, additionally, $\varphi\in C_0(\X)$ then $\langle \varphi,\cdot \rangle\in\Fcyl$. Not that for e.g. $\varphi(x)=\|x\|_{\X}$ (the Euclidean norm in $\X$) we have that $\langle \varphi,\ga \rangle=\infty$ for any infinite $\ga\in\Ga$.
\end{example}

\begin{example}\label{ex:energy}
  Let $\phi:\X\setminus\{0\}\to \R$ be an even function, namely, $\phi(-x)=\phi(x)$, $x\in\X$. Then one can consider the so-called {\em energy function}
  \begin{equation}\label{energy}
    E^\phi(\ga):=
        \begin{cases}
            \displaystyle\sum_{\{x,y\}\subset\ga} \phi(x-y), & \displaystyle\text{if} \ \sum_{\{x,y\}\subset\ga} |\phi(x-y)|<\infty,\quad \ga\in\Ga,\\
            \displaystyle+\infty, &\text{otherwise}.
        \end{cases}
  \end{equation}
  Clearly, $E^\phi\in\Ffin$. However, even for $\phi$ with a compact support, $E^\phi$ will not be a~cylindric function.
\end{example}

As we discussed before, any configuration $\ga$ represents some system of elements in a real-world application. Typically, investigators are not able to take into account exact positions of all elements due to huge number of them. For quantitative and qualitative analysis of a system researchers mostly need some its statistical characteristics such as density, correlations, spatial structures and so on. This leads to the so-called statistical description of complex systems when people study distributions of countable sets in an underlying space instead of sets themselves. Moreover, the main idea in Boltzmann's approach to thermodynamics based on giving up the description in terms of evolution for groups of molecules and using statistical interpretation of molecules motion laws. Therefore, the crucial role for studying of complex systems plays distributions (probability measures) on the space of configurations. In statistical physics these measures usually called {\em states} that accentuates their role for description of systems under consideration.

We denote the class of all probability measures on $\bigl(\Ga,\B(\Ga)\bigr)$ by $\M^1(\Ga)$. Given a distribution $\mu\in\M^1(\Ga)$ one can consider a collection of random variables $N_\La(\cdot)$, $\La\in\Bb$ defined in \eqref{gensalg}. They describe random numbers of elements inside bounded regions. The natural assumption is that these random variables should have finite moments. Thus, we consider the class $\Mfm$ of all measures from $\M^1(\Ga)$ such that
\begin{equation}\label{Mfm}
    \int_\Ga |\ga_\La|^n \,d\mu(\ga)<\infty, \qquad \La\in\Bb, n\in\N.
\end{equation}

\begin{example}\label{ex:Poisson}
  Let $\sigma$ be a non-atomic Radon measure on $\bigl(\X,\B(\X)\bigr)$. Then the {\em Poisson measure} $\pi_\sigma$ with intensity measure $\sigma$ is defined on $\B(\Ga)$ by
  \begin{equation}\label{Poisson}
    \pi_\sigma \bigl(Q(\La,n)\bigr)=\frac{\bigl(\sigma(\La)\bigr)^n}{n!}\exp\bigl\{ -\sigma(\La)\bigr\}, \qquad \La\in\Bb, n\in\N_0.
  \end{equation}
This formula is nothing but the statement that the random variables $N_\La$ have Poissonian distribution with mean value $\sigma(\La)$, $\La\in\Bb$.  Note that by the R\'{e}nyi theorem \cite{Ren1970,Ken1974} a measure $\pi_\sigma$ will be Poissonian if \eqref{Poisson} holds for $n=0$ only.
  In the case then $d\sigma(x)=\rho(x)\,dx$ one can say about nonhomogeneous Poisson measure $\pi_\rho$ with density (or intensity) $\rho$. This notion goes back to the famous Campbell formula \cite{Cam1910,Cam1909} which states that
  \begin{equation}\label{Campbell}
      \int_\Ga \langle \varphi,\ga\rangle \,d\pi_\rho(\ga) = \int_\X \varphi(x) \rho(x) \,dx,
  \end{equation}
  if only the right hand side of \eqref{Campbell} is well-defined. The generalization of \eqref{Campbell} is the Mecke identity \cite{Mec1968}
  \begin{equation}\label{Mecke}
    \int_\Ga\sum_{x\in\ga} h(x,\ga)\,d\pi_\sigma(\ga) = \int_\Ga\int_\X
    h(x,\ga\cup x) \,d\sigma(x)\,d\pi_\sigma(\ga),
  \end{equation}
  which holds for all measurable nonnegative functions $h:\X\times\Ga\to\R$. Here and in the sequel we will omit brackets for the one-point set $\{x\}$. In \cite{Mec1968}, it was shown that the Mecke identity is a characterization identity for the Poisson measure. In the case $\rho(x)= z>0$, $x\in\X$ one can say about the homogeneous Poisson distribution (measure) $\pi_z$ with constant intensity $z$. We will omit sub-index for the case $z=1$, namely, $\pi:=\pi_1=\pi_{dx}$. Note that the property \eqref{Mfm} is followed from \eqref{Mecke} easily.
\end{example}

\begin{example}\label{ex:Gibbs}
  Let $\phi$ be as in Example \ref{ex:energy} and suppose that the energy given by \eqref{energy} is {\em stable}: there exists $B\geq0$ such that, for any $|\ga|<\infty$,  $E^\phi(\ga)\geq - B |\ga|$. An example of such $\phi$ my be given by the expansion
  \begin{equation}\label{stpot}
    \phi(x)=\phi^+(x)+\phi^p(x), \quad x\in\X,
  \end{equation}
  where $\phi^+\geq0$ whereas $\phi^p$ is a positive defined function on $\X$ (the Fourier transform of a measure on $\X$), see e.g. \cite{Rue1969, FR1966}. Fix any $z>0$ and define the {\em Gibbs measure} $\mu\in\M^1(\Ga)$ with potential $\phi$ and activity parameter $z$ as a measure which satisfies the following generalization of the Mecke identity:
  \begin{equation}\label{GNZ}
    \int_\Ga\sum_{x\in\ga} h(x,\ga)\,d\mu(\ga) = \int_\Ga\int_\X
    h(x,\ga\cup x) \exp\{-E^\phi(x,\ga)\} \, z dx\,d\mu(\ga),
  \end{equation}
  where
  \begin{equation}\label{localenergy}
    E^\phi(x,\ga):=\langle\phi(x-\cdot),\ga\rangle=\sum_{y\in\ga}\phi(x-y), \quad \ga\in\Ga, x\in\X\setminus\ga.
  \end{equation}
  The identity \eqref{GNZ} is called the Georgii--Nguyen--Zessin identity, see  \cite{Geo1976,NZ1979}. If potential $\phi$ is additionally satisfied the so-called integrability condition
  \begin{equation}\label{integrability}
    \beta:=\int_\X \bigl| e^{-\phi(x)}-1\bigr|\, dx <  \infty,
  \end{equation}
  then it can checked that the condition \eqref{Mfm} for the Gibbs measure holds. Note that under conditions $z\beta\leq (2e)^{-1}$ there exists a unique measure on $\bigl(\Ga, \B(\Ga)\bigr)$ which satisfies \eqref{GNZ}. Heuristically, the measure $\mu$ may be given by the formula
  \begin{equation}\label{Gibbsdensity}
    d\mu(\ga)=\frac{1}{Z} e^{-E^\phi(\ga)}\,d\pi_z(\ga),
  \end{equation}
  where $Z$ is a normalizing factor. To give rigorous meaning for \eqref{Gibbsdensity} it is possible to use the so-called DLR-approach (named after R.\,L.\,Dobrushin, O.\,Lanford, D.\,Ruelle), see e.g. \cite{AKR1998} and references therein. As was shown in \cite{NZ1979}, this approach gives the equivalent definition of the Gibbs measures which satisfies \eqref{GNZ}.
\end{example}

Note that \eqref{Gibbsdensity} could have a rigorous sense if we restrict our attention on the space of configuration which belong to a bounded domain $\La\in\Bb$. The space of such (finite) configurations will be denoted by $\Ga(\La)$. The $\sigma$-algebra $\B(\Ga(\La))$ may be generated by family of mappings $\Ga(\La)\ni\ga\mapsto N_{\La'}(\ga)\in\N_0$, $\La'\in\Bb$, $\La'\subset\La$. A measure $\mu\in\Mfm$ is called
locally absolutely continuous with respect to the Poisson measure
$\pi$ if for any $\La\in\Bb$ the projection of
$\mu$ onto $\Ga(\La)$ is absolutely continuous with respect to (w.r.t.) the
projection of $ \pi$ onto $\Ga(\La)$. More precisely, if we consider the projection mapping $p_\La:\Ga\to\Ga(\La)$, $p_\La(\ga):=\ga_\La$ then $\mu^\La:=\mu\circ p_\La^{-1}$ is absolutely continuous w.r.t. $\pi_\La:=\pi\circ p_\La^{-1}$.

\begin{remark}
Having in mind \eqref{Gibbsdensity}, it is possible to derive from \eqref{GNZ} that the Gibbs measure from Example~\ref{ex:Gibbs} is locally absolutely continuous w.r.t. the Poisson measure, see e.g. \cite{FK2005} for the more general case.
\end{remark}

By e.g. \cite{KK2002}, for any $\mu\in\Mfm$ which is locally absolutely continuous w.r.t the Poisson measure there exists the family of (symmetric) {\em correlation functions} $k_\mu^\n:(\X)^n\to\R_+:=[0,\infty)$ which defined as follows. For any symmetric function $f^\n:(\X)^n\to\R$ with a finite support the following equality holds
\begin{multline}\label{cfdef}
    \int_\Ga \sum_{\{\xn\}\subset\ga} f^\n(\xn) \,d\mu(\ga)
    \\=\frac{1}{n!}\int_{(\X)^n} f^\n(\xn)k_\mu^\n(\xn)\, dx_1\ldots dx_n
\end{multline}
for $n\in\N$, and $k_\mu^{(0)}:=1$.

The meaning of the notion of correlation functions is the following: the correlation function $k_\mu^\n(\xn)$ describes the non-normalized density of probability to have points of our systems in the positions $\xn$.
%

\begin{remark}
Iterating the Mecke identity \eqref{Mecke}, it can be easily shown that
\begin{equation}\label{cfofpirho}
  k_{\pi_\rho}^\n(\xn)=\prod\limits_{i=1}^n\rho(x_i),
\end{equation}
in particular,
\begin{equation}\label{cfofpiz}
  k_{\pi_z}^\n(\xn)\equiv z^n.
\end{equation}
\end{remark}

\begin{remark}
  Note that if potential $\phi$ from Example~\ref{ex:Gibbs} satisfies to \eqref{stpot}, \eqref{integrability} then, by \cite{Rue1970}, there exists $C=C(z,\phi)>0$ such that for $\mu$ defined by~\eqref{GNZ}
  \begin{equation}\label{RB}
    k_\mu^\n(\xn)\leq C^n, \qquad \xn\in\X.
  \end{equation}
  The inequality \eqref{RB} is referred to as the Ruelle bound.
\end{remark}

We dealt with symmetric function of $n$ variables from $\X$, hence, they can be considered as functions on $n$-point subsets from $\X$. We proceed now to the exact constructions.

The space of $n$-point configurations in $Y\in\B(\X)$ is defined by
\begin{equation*}
\Ga^{(n)}(Y):=\bigl\{ \eta \subset Y \bigm| |\eta |=n\bigr\}, \qquad
n\in { \N}.
\end{equation*}
We put $\Ga^{(0)}(Y):=\{\emptyset\}$. As a~set, $\Ga^{(n)}(Y)$ may
be identified with the symmetrization of
\[
\widetilde{Y^n} = \bigl\{
(x_1,\ldots ,x_n)\in Y^n \bigm| x_k\neq x_l \ \mathrm{if} \ k\neq
l\bigr\}.
\]
Hence, one can introduce the corresponding Borel $\sigma$-algebra, which we denote by $\B\bigl(\Ga^{(n)}(Y)\bigr)$. The
space of finite configurations in $Y\in\B(\X)$ is defined as
\begin{equation}\label{Ga0}
\Ga_0(Y):=\bigsqcup_{n\in {\N}_0}\Ga^{(n)}(Y).
\end{equation}
This space is equipped with the topology of the disjoint union. Let
$\B \bigl(\Ga_0(Y)\bigr)$ denote the corresponding Borel $\sigma
$-algebra. In the case of $Y=\X$ we will omit the index $Y$ in the
previously defined notations. Namely,
\begin{equation}\label{speccases}
    \Ga_0:=\Ga_{0}(\X), \qquad
\Ga^{(n)}:=\Ga^{(n)}(\X), \quad n\in\N_0.
\end{equation}

The restriction of the Lebesgue product measure $(dx)^n$ to
$\bigl(\Ga^{(n)}, \B(\Ga^{(n)})\bigr)$ we denote by $m^{(n)}$. We
set $m^{(0)}:=\delta_{\{\emptyset\}}$. The Lebesgue--Poisson measure
$\la $ on $\Ga_0$ is defined by
\begin{equation} \label{LP-meas-def}
\la :=\sum_{n=0}^\infty \frac {1}{n!}m^{(n)}.
\end{equation}
For any $\La\in\B_{\mathrm{b}}(\X)$ the restriction of $\la$ to $\Ga_{0}(\La)=\Ga
(\La)$ will be also denoted by $\la $.

\begin{remark}
The space
$\bigl( \Ga, \B(\Ga)\bigr)$ is the projective limit of the family of measurable
spaces $\bigl\{\bigl( \Ga(\La), \B(\Ga(\La))\bigr)\bigr\}_{\La \in
\Bb}$. The Poisson measure $\pi$ on $\bigl(\Ga
,\B(\Ga )\bigr)$ from Example~\ref{ex:Poisson} may be defined as the projective limit of the family of
measures $\{\pi^\La \}_{\La \in \Bb}$, where $
\pi^\La:=e^{-m(\La)}\la $ is the probability measure on $\bigl(
\Ga(\La), \B(\Ga(\La))\bigr)$ and $m(\La)$ is the Lebesgue measure
of $\La\in \Bb$ (see e.g. \cite{AKR1998a} for
details).
\end{remark}

Functions on $\Ga_0$ will be called {\em quasi-observables}.
Any $\B(\Ga_0)$-measurable function $G$ on $ \Ga_0$, in
fact, is defined by a~sequence of functions
$\bigl\{G^{(n)}\bigr\}_{n\in{ \N}_0}$ where $G^{(n)}$ is a
$\B(\Ga^{(n)})$-measurable function on $\Ga^{(n)}$. We preserve the same notation for the function $G^{(n)}$ considered as a symmetric function on $(\X)^n$. Note that $G^{(0)}\in\R$.

A set $M\in \B (\Ga_0)$ is called bounded if there exists $ \La \in
\Bb$ and $N\in { \N}$ such that
\begin{equation*}\label{bddinGa0}
    M\subset \bigsqcup_{n=0}^N\Ga^{(n)}(\La).
\end{equation*}
The set of bounded measurable functions on $\Ga_0$ with bounded support we denote by
$\Bbs$, i.e., $G\in \Bbs$ iff $G\upharpoonright_{\Ga_0\setminus M}=0$ for some bounded $M\in {\B}(\Ga_0)$. For any $G\in\Bbs$ the functions $G^\n$ have finite supports in $(\X)^n$ and may be substituted into \eqref{cfdef}. But, additionally, the sequence of $G^\n$ vanishes for big $n$. Therefore, one can summarize equalities \eqref{cfdef} by $n\in\N_0$. This leads to the following definition.

Let $G\in\Bbs$, then we define the function $KG:\Ga\to\R$ such that:
\begin{align}
(KG)(\ga )&:=\sum_{\eta \Subset \ga }G(\eta )\label{K-transform}\\&\hphantom{:}=G^{(0)}+\sum_{n=1}^\infty\sum_{\{\xn\}\subset\ga}G^\n(\xn), \quad \ga \in \Ga,\notag
\end{align}
see e.g. \cite{KK2002,Len1975,Len1975a}. The summation in \eqref{K-transform} is
taken over all finite subconfigurations $\eta\in\Ga_0$ of the
(infinite) configuration $\ga\in\Ga$; we denote this by the symbol,
$\eta\Subset\ga $. The mapping $K$ is linear, positivity preserving,
and invertible, with
\begin{equation}
(K^{-1}F)(\eta ):=\sum_{\xi \subset \eta }(-1)^{|\eta \setminus \xi
|}F(\xi ),\quad \eta \in \Ga_0. \label{k-1transform}
\end{equation}
By \cite{KK2002}, for any $G\in\Bbs$, $KG\in\Fcyl$, moreover, there exists $C=C(G)>0$,
$\La=\La(G)\in\Bb$, and $N=N(G)\in\N$ such that
\begin{equation}\label{estBbs}
    |KG(\ga)|\leq C \bigl( 1+ |\ga_\La|\bigr)^N, \quad \ga\in\Ga.
\end{equation}

The expression \eqref{K-transform} can be extended to the class of all nonnegative measurable $G:\Ga_0\to\R_+$, in this case, evidently, $KG\in\Ffin$. Stress that the left hand side (l.h.s.) of \eqref{k-1transform} has a meaning for any $F\in\Ffin$, moreover, in this case $(KK^{-1}F)(\ga)=F(\ga)$ for any $\ga\in\Ga_0$.

For $G$ as above we may summarize \eqref{cfdef} by $n$ and rewrite the result in a compact form:
\begin{equation}\label{eqmeans}
\int_\Ga (KG)(\ga) d\mu(\ga)=\int_{\Ga_0}G(\eta)
k_\mu(\eta)d\la(\eta).
\end{equation}
As was shown in \cite{KK2002}, the equality \eqref{K-transform} may be extended on all functions $G$ such that the l.h.s. of \eqref{eqmeans} is finite. In this case \eqref{K-transform} holds for $\mu$-a.a. $\ga\in\Ga$ and \eqref{eqmeans} holds too.

\begin{remark}
The equality \eqref{eqmeans} may be considered as definition of the correlation function $k_\mu$. In fact, the definition of correlation functions in statistical physics, given by N.\,N.\,Bogolyubov in \cite{Bog1962}, based on a~similar relation. More precisely, consider for a~$\B(\X)$-measurable
function $f$ the so-called coherent state, given as a function on $\Ga_0$ by
\begin{equation}\label{Leb-Pois-exp}
e_\la (f,\eta ):=\prod_{x\in \eta }f(x) ,\ \eta \in \Ga
_0\!\setminus\!\{\emptyset\},\qquad e_\la (f,\emptyset ):=1.
\end{equation}
Then for any $f\in C_0(\X)$ we have the point-wise equality
\begin{equation}\label{Kexp}
\bigl(Ke_\la (f)\bigr)(\ga)=\prod_{x\in\ga}\bigl(1+f(x)\bigr), \quad \eta\in\Ga_0.
\end{equation}
As a result, the correlation functions of different orders may be considered as
kernels of a Taylor-type expansion
\begin{align}\notag
    \int_\Ga \prod_{x\in\ga}\bigl(1+f(x)\bigr) \, d\mu(\ga) &=
    1 + \sum_{n=1}^\infty \frac{1}{n!} \int_{(\X)^n} \prod_{i=1}^n f(x_i) k_\mu^\n(\xn) \,dx_1\ldots dx_n\\&=\int_{\Ga_0} e_\la(f,\eta) k_\mu(\eta)\, d\la(\eta).\label{funcBogol}
\end{align}
\end{remark}

\begin{remark}
By \eqref{Ga0}--\eqref{LP-meas-def}, we have that for any $f\in L^1(\X,dx)$
\begin{equation}\label{intexp}
\int_{\Ga_0}e_\la (f,\eta)d\la(\eta)=\exp\Bigl\{\int_\X
f(x)dx\Bigr\}.
\end{equation}
As a result, taking into account \eqref{cfofpirho}, we obtain from \eqref{funcBogol}
the expression for the~Laplace transform of the~Poisson measure
\begin{equation*}\label{LaplPois}
    \begin{split}
    \int_\Ga e^{-\langle \varphi, \ga\rangle} \, d\pi_\rho(\ga)
    &=\int_{\Ga_0} e_\la \bigl( e^{-\varphi(x)}-1, \eta \bigr) e_\la(\rho, \eta)
    \,d\la(\eta)\\&=\exp \Bigl\{ - \int_\X \bigl( 1- e^{-\varphi(x)}\bigr) \rho(x) dx\Bigr\}, \qquad \varphi\in C_0(\X).
    \end{split}
\end{equation*}
\end{remark}

\begin{remark}
  Of course, to obtain convergence of the expansion \eqref{funcBogol} for, say,
  $f\in L^1(\X,dx)$ we need some bounds for the correlation functions $k_\mu^\n$. For example, if the generalized Ruelle bound holds, that is, cf. \eqref{RB},
  \begin{equation}\label{genRB}
    k_\mu^\n(\xn) \leq A C^n (n!)^{1-\delta}, \quad \xn\in\X
  \end{equation}
  for some $A,C>0$, $\delta \in(0,1]$ independent on $n$, then the l.h.s. of \eqref{funcBogol} may be estimated by the expression
  \begin{equation*}
    1+A\sum_{n=1}^\infty \frac{\bigl( C\Vert f\Vert_{L^1(\X)}\bigr)^n}{(n!)^\delta} <\infty.
  \end{equation*}
\end{remark}

For a given system of functions $k^\n$ on $(\X)^n$ the question
about existence and uniqueness of a probability measure $\mu$ on $\Ga$
which has correlation functions $k_\mu^\n=k^\n$ is an analog
of the moment problem in classical analysis. Significant results in this area were obtained by A.\,Lenard.
\begin{proposition}[\!\!{\cite{Len1973}, \cite{Len1975a}}]\label{exuniqmeas}
Let $k:\Ga_0\to\R$.

 $1.$ Suppose that  $k$ is a positive definite function, that means that for any $G\in\Bbs$ such that $(KG)(\ga)\geq0$ for all $\ga\in\Ga$ the following inequality holds
 \begin{equation}\label{Lenpos}
    \int_{\Ga_0} G(\eta) k(\eta) \,d\la(\eta)\geq0.
 \end{equation}
 Suppose also that $k(\emptyset)=1$. Then there exists at least one measure $\mu\in\Mfm$ such that $k=k_\mu$.

 $2.$ For any $n\in\N$, $\La\in\Bb$, we set
 \[
 s_n^\La:= \frac{1}{n!}\int_{\La^n}k^\n(\xn) \,dx_1\ldots dx_n.
 \]
 Suppose that for all $m\in\N$, $\La\in\Bb$
\begin{equation}\label{bd}
    \sum_{n\in\N}\bigl(s_{n+m}^\La\bigr)^{-\frac{1}{n}}=\infty.
\end{equation}
 Then there exists at most one measure $\mu\in\Mfm$ such that $k=k_\mu$.
\end{proposition}

\begin{remark}\label{remLen}
  1. In \cite{Len1975a,Len1973}, the wider space of multiple configurations was considered. The adaptation for the space $\Ga$ was realized in~\cite{Kun1999}.

  2. It is worth noting also that the growth of correlation functions $k^\n$ up to $(n!)^2$ is admissible to have \eqref{bd}.

  3. Another conditions for existence and uniqueness for the moment problem on $\Ga$ were srudied in \cite{KK2002,BKKL1999}.
\end{remark}

\subsection{Statistical descriptions of Markov evolutions}
Spatial Markov processes in $\X$ may be described as stochastic
evolutions of configurations $\ga\subset\X$. In course of such evolutions points of configurations may disappear (die), move (continuously or with jumps from one position to another), or new particles may appear in a configuration (that is birth). The rates of these random events may depend on whole configuration that reflect an interaction between elements of the our system.

The construction of a spatial Markov process in the~continuum is highly difficult question which is not solved in a full generality at present, see e.g. a review \cite{Pen2008} and more detail references about birth-and-death processes in Section 3. Meanwhile, for the discrete systems the corresponding processes are constructed under quite general assumptions, see e.g. \cite{Lig1985}. One of the main difficulties for continuous systems includes the necessity to control number of elements in a bounded region. Note that the construction of spatial processes on bounded sets from $\X$ are typically well solved, see e.g. \cite{FM2004}.

The existing Markov process $\Ga\ni\ga\mapsto X_t^\ga\in\Ga$, $t>0$ provides solution the backward Kolmogorov equation for bounded continuous functions:
\begin{equation}\label{BKE}
  \frac{\partial}{\partial t} F_t = L F_t,
\end{equation}
where $L$ is the Markov generator of the process $X_t$. The question about existence and properties of solutions to \eqref{BKE} in proper spaces itself is also highly nontrivial problem of infinite-dimensional analysis. The Markov generator $L$ should satisfies the following two (informal) properties: 1)~to be conservative, that is $L1=0$, 2)~maximum principle, namely, if there exists $\ga_0\in\Ga$ such that $F(\ga)\leq F(\ga_0)$ for all $\ga\in\Ga$, then $(LF)(\ga_0)\leq0$. These properties might yield that the semigroup, related to \eqref{BKE} (provided it exists), will preserves constants and positive functions, correspondingly.

To consider an example of such $L$ let us consider a general Markov evolution with appearing and disappearing of groups of points (giving up the case of continuous moving of particles). Namely, let $F\in\Fcyl$ and set
\begin{equation}\label{genGen}
    (LF)(\ga)=\sum_{\eta\Subset\ga}\int_{\Ga_0} c (\eta, \xi, \ga\setminus\eta) \bigl[
    F((\ga\setminus\eta)\cup\xi)-F(\ga)\bigr]\,d\la(\xi).
\end{equation}
Heuristically, it means that any finite group $\eta$  of points from the existing configuration~$\ga$ may disappear and simultaneously a new group $\xi$ of points may appear somewhere in the space $\X$. The rate of this random event is equal to $c(\eta,\xi,\ga\setminus\eta)\geq0$. We need some minimal conditions on the rate $c$ to guarantee that at least
\begin{equation}\label{lfin}
    LF\in\Ffin   \qquad \text{ for all } F\in\Fcyl
\end{equation}
(see Section 3 for a particular case). The term in the sum in \eqref{genGen} with $\eta=\emptyset$ corresponds to a pure birth of a finite group $\xi$ of points whereas the part of integral corresponding to $\xi=\emptyset$ (recall that $\la(\{\emptyset\})=1$) is related to pure death of a finite sub-configuration $\eta\subset\ga$. The parts with $|\eta|=|\xi|\neq0$ corresponds to jumps of one group of points into another positions in $\X$. The rest parts present splitting and merging effects. In the present paper the technical realization of the ideas below is given for one-point birth-and-death parts only, i.e. for the cases $|\eta|=0$, $|\xi|=1$ and $|\eta|=1$, $|\xi|=0$, correspondingly.

As we noted before, for most cases appearing in applications, the existence problem for a corresponding Markov process with a generator $L$ is still open. On the other hand, the evolution of a state in the course of a stochastic dynamics is an important question in its own right. A~mathematical formulation of this question may be realized through
the forward Kolmogorov equation for probability measures (states) on
the configuration space $\Gamma$. Namely, we consider the pairing between
functions and measures on $\Ga$ given by
\begin{equation}\label{pairing}
  \langle F,\mu \rangle:=\int_\Ga F(\ga)\,d\mu(\ga).
\end{equation}
Then we consider the initial value problem
\begin{equation}\label{FPE-init}
  \frac{d}{d t} \langle F, \mu_t\rangle = \langle LF, \mu_t\rangle,
  \quad t>0, \quad \mu_t\bigr|_{t=0}=\mu_0,
\end{equation}
where $F$ is an arbitrary function from a proper set, e.g. $F\in K\bigl(\Bbs\bigr)\subset\Fcyl$.
In fact, the solution to \eqref{FPE-init} describes
the time evolution of distributions instead of the evolution of
initial points in the Markov process. We rewrite \eqref{FPE-init} in the following heuristic form
\begin{equation}\label{FPE}
  \frac{\partial}{\partial t} \mu_t = L^* \mu_t,
\end{equation}
where $L^*$ is the (informally) adjoint operator of $L$ with respect to the
pairing \eqref{pairing}.

In the physical literature, \eqref{FPE} is referred to the {\em Fokker--Planck
equation}. The Markovian property of $L$ yields that \eqref{FPE} might have a solution in the class of probability measures. However, the mere existence of the corresponding Markov process will not give us much information about properties
of the solution to \eqref{FPE}, in particular, about its moments or correlation functions. To do this, we suppose now that a solution
$\mu_t\in {\mathcal{M}}_{\mathrm{fm} }^1(\Ga )$ to \eqref{FPE-init}
exists and remains locally absolutely continuous with respect to
the Poisson measure $\pi$ for all $t>0$ provided $\mu_0$ has such a property.
Then one can consider the correlation function
$k_t:=k_{\mu_t}$, $t\geq0$.

Recall that we suppose \eqref{lfin}. Then, one can calculate
$K^{-1}LF$ using \eqref{k-1transform}, and, by \eqref{eqmeans}, we may rewrite
\eqref{FPE-init} in the following way
\begin{equation}\label{ssd0}
  \frac{d}{d t} \langle\!\langle K^{-1}F, k_t\rangle\!\rangle
  = \langle\!\langle K^{-1}LF, k_t\rangle\!\rangle,\quad t>0, \quad
  k_t\bigr|_{t=0}=k_0,
\end{equation}
for all $F\in K\bigl(\Bbs\bigr)\subset\Fcyl$. Here the pairing
between functions on $\Ga_0$ is given by
\begin{equation}
\left\langle \!\left\langle G,\,k\right\rangle \!\right\rangle
:=\int_{\Ga _{0}}G(\eta) k(\eta) \,d\la(\eta). \label{duality}
\end{equation}
Let us recall that then, by \eqref{LP-meas-def},
\begin{equation*}\label{duality-intro}
  \langle\!\langle G,k \rangle\!\rangle=
  \sum_{n=0}^\infty \frac{1}{n!} \int_{(\X)^n}
  G^{(n)}(x_1,\ldots,x_n)
  k^{(n)}(x_1,\ldots,x_n)\,dx_1\ldots dx_n,
\end{equation*}

Next, if we substitute $F=KG$, $G\in
\Bbs$ in \eqref{ssd0}, we derive
\begin{equation}\label{ssd}
  \frac{d}{d t} \langle\!\langle G, k_t\rangle\!\rangle
  = \langle\!\langle \widehat{L}G, k_t\rangle\!\rangle, \quad t>0, \quad
  k_t\bigr|_{t=0}=k_0,
\end{equation}
for all $G\in \Bbs$. Here the operator
\begin{equation*}\label{Lhat}
    (\widehat{L}G)(\eta ) :=( K^{-1}LKG)(\eta),\quad \eta\in\Ga_0
\end{equation*}
is defined point-wise for all $G\in\Bbs$ under conditions \eqref{lfin}.
As a result, we
are interested in a weak solution to the equation
\begin{equation}\label{QE}
  \frac{\partial}{\partial t} k_t
  = \widehat{L}^* k_t, \quad t>0, \quad
  k_t\bigr|_{t=0}=k_0,
\end{equation}
where $\widehat{L}^*$ is dual operator to $\widehat{L}$ with respect to the
duality \eqref{duality}, namely,
\begin{equation}
\int_{\Ga _{0}}(\widehat{L}G)(\eta) k(\eta) \,d\la(\eta)
=\int_{\Ga _{0}}G(\eta) (\widehat{L}^*k)(\eta) \,d\la(\eta). \label{dualoper}
\end{equation}

The procedure of deriving the operator $\widehat{L}$ for a given $L$ is fully combinatorial meanwhile to obtain the expression for the operator $\widehat{L}^*$ we need
an analog of integration by parts formula. For a difference operator $L$ considered in \eqref{genGen} this discrete integration by parts rule is presented in Lemma~\ref{Minlos} below.

We recall that any function on $\Ga_0$ may be identified with an infinite vector of symmetric functions of the growing number of variables. In this approach, the operator $\widehat{L}^*$ in \eqref{QE} will be realized as an infinite matrix $\bigl(\widehat{L}^*_{n,m}\bigr)_{n,m\in\N_0}$, where $\widehat{L}^*_{n,m}$ is a mapping from the space of symmetric functions of $n$ variables into the space of symmetric functions of $m$ variables. As a result, instead of equation \eqref{FPE-init} for infinite-dimensional objects we obtain an infinite system of equations for functions $k_t^\n$ each of them is a function of a finite number of variables, namely
\begin{equation}\label{QE-n}
    \begin{split}
    \frac{\partial}{\partial t} k_t^\n(\xn)&=\bigl(\widehat{L}^*_{n,m}k_t^\n\bigr)(\xn), \quad t>0, \quad n\in\N_0,\\ k_t^\n(\xn)\bigr|_{t=0}&=k_0^\n(\xn).
    \end{split}
\end{equation}
Of course, in general, for a fixed $n$, any equation from \eqref{QE-n} itself is not closed and includes functions $k_t^{(m)}$ of other orders $m\neq n$, nevertheless, the system \eqref{QE-n} is a closed linear system. The chain evolution equations for $k_t^{(n)}$ consists the so-called {\em hierarchy} which is an analog of the BBGKY hierarchy for Hamiltonian systems, see e.g. \cite{DSS1989}.

One of the main aims of the present paper is to study the classical solution to
\eqref{QE} in a proper functional space. The choice of such a space
might be based on estimates \eqref{RB}, or more generally, \eqref{genRB}.
However, even the correlation functions \eqref{cfofpiz} of the Poisson
measures shows that it is rather natural to study the solutions to
the equation \eqref{QE} in weighted $L^\infty$-type space of
functions with the Ruelle-type bounds. Integrable correlation functions are not natural for the dynamics on the spaces of locally finite configurations.
For example, it is well-known that the Poisson measure $\pi_\rho$ with integrable density $\rho(x)$ is concentrated on the space $\Ga_0$ of finite configurations (since in this case on can consider $\X$ instead of $\La$ in \eqref{Poisson}).
Therefore, typically, the case of integrable correlation functions yields that effectively our stochastic dynamics evolves through finite configurations only. Note that the case of an integrable first order correlation function is referred to {\em zero density} case in statistical physics.

In the present paper the restrict our attention to the so-called
{\em sub-Poissonian} correlation functions. Namely, for a given $C>0$ we consider the following Banach space
\begin{equation}\label{KC}
{\K}_{C}:=\bigl\{ k:\Ga_0\to\R \bigm| k\cdot
C^{-|\cdot |}\in L^\infty(\Ga_0,d\la )\bigr\}
\end{equation}
with the norm
\begin{equation*}\label{normKC}
    \Vert k\Vert _{{\K}_{C}}:=\Vert C^{-|\cdot |}k(\cdot )\Vert
_{L^{\infty }(\Ga _{0},\la )}.
\end{equation*}
It is clear that $k\in {\K}_{C}$ implies, cf. \eqref{RB},
\begin{equation}\label{ineqKC}
    \bigl|k(\eta )\bigr|\leq \Vert k\Vert _{{\K}_{C}}\,C^{|\eta|} \qquad
\mathrm{for} \ \la\text{-a.a.} \ \eta \in \Ga _{0}.
\end{equation}

In the following we distinguish two possibilities for a study of the initial value problem \eqref{QE}. We may try to solve this equation in one space $\K_C$. The well-posedness of the initial value problem  in this case is equivalent with an existence of the strongly continuous semigroup ($C_0$-semigroup in the sequel) in the space $\K_C$ with a generator $\widehat{L}^*$. However, the space $\K_C$ is isometrically isomorphic to the space $L^{\infty }(\Ga_0,C^{|\cdot|}d\la)$ whereas, by the H.~Lotz theorem \cite{Lot1985}, \cite{AGGG1986}, in a $L^\infty$ space any $C_0$-semigroup is uniformly continuous, that is it has a bounded generator. Typically, for the difference operator $L$ given in \eqref{genGen}, any operator $\widehat{L}^*_{n,m}$, cf. \eqref{QE-n}, might be bounded as an operator between two spaces of bounded symmetric functions of $n$ and $m$ variables whereas the whole operator $\widehat{L}^*$ is unbounded in $\K_C$.

To avoid this difficulties we use a trick which goes back to R.~Phillips \cite{Phi1955}. The main idea is to consider the semigroup in $L^\infty$ space not itself but as a dual semigroup $T^*(t)$ to a $C_0$-semigroup $T(t)$ with a generator $A$ in the pre-dual $L^1$ space. In this case $T^*(t)$ appears strongly continuous semigroup not on the whole $L^\infty$ but on the closure of the domain of $A^*$ only.

In our case this leads to the following scheme. We consider the pre-dual Banach space to $\K_C$, namely, for $C>0$,
\begin{equation}
\L_{C}:=L^{1}\bigl(\Ga _{0},C^{|\cdot|}d\la\bigr).
\label{space1}
\end{equation}
The norm in $\L_C$ is given by
\begin{equation*}\label{normLC}
    \|G\|_C:=\int_{\Ga_0} \bigl| G(\eta)\bigr| C^{|\eta|}\,d\la(\eta)=
    \sum_{n=0}^\infty \frac{C^n}{n!} \int_{(\X)^n}
  \bigl|G^{(n)}(x_1,\ldots,x_n)\bigr|
  \,dx_1\ldots dx_n.
\end{equation*}
Consider the initial value problem, cf. \eqref{ssd}, \eqref{QE},
\begin{equation}\label{ssd2}
  \frac{\partial}{\partial t} G_t
  = \widehat{L} G_t, \quad t>0, \quad
  G_t\bigr|_{t=0}=G_0\in\L_C.
\end{equation}
Whereas \eqref{ssd2} is well-posed in $\L_C$ there exists a $C_0$-semigroup $\widehat{T}(t)$ in $\L_C$. Then using Philips' result we obtain that the restriction of the dual semigroup $\widehat{T}^*(t)$ onto $\overline{\mathrm{Dom}(\widehat{L}^*)}$ will be $C_0$-semigroup with generator which is a part of $\widehat{L}^*$ (the details see in Section~3 below). This provides a solution to \eqref{QE} which continuously depends on an initial data from $\overline{\mathrm{Dom}(\widehat{L}^*)}$. And after we would like to find a more useful universal subspace of $\K_C$ which is not depend on the operator $\widehat{L}^*$. The realization of this scheme for a birth-and-death operator $L$ is presented in Section 3 below. As a result, we obtain the classical solution to \eqref{QE} for $t>0$ in a class of sub-Poissonian functions which satisfy the Ruelle-type bound \eqref{ineqKC}. Of course, after this we need to verify existence and uniqueness of measures whose correlation functions are solutions to \eqref{QE}, cf. Proposition~\ref{exuniqmeas} above. This usually can be done using proper approximation schemes, see e.g. Section 4.

There is another possibility for a study of the initial value problem \eqref{QE} which we will not touch below. Namely, one can consider this evolutional equation in a proper scale of spaces $\{\K_C\}_{C_*\leq C\leq C^*}$. In this case we will have typically that the solution is local in time only. More precisely, there exists $T>0$ such that for any $t\in[0,T)$ there exists a unique solution to \eqref{QE} and $k_t\in\K_{C_t}$ for some $C_t\in[C_*,C^*]$. We realized this approach in series of papers \cite{FKO2011a,FKO2011b,FKKoz2011,BKKK2011} using the so-called Ovsyannikov method \cite{Ovs1965,Tre1968,Saf1995}. This method provides less restrictions on systems parameters, however, the price for this is a finite time interval. And, of course, the question about possibility to recover measures via solutions to \eqref{QE} should be also solved separately in this case.

\section{Birth-and-death evolutions in the~continuum}

\subsection{Microscopic description} One of the most important classes of Markov evolution in the~continuum is given by the birth-and-death Markov processes
in the space $\Gamma$ of all configurations from $\X$. These are
processes in which an infinite number of individuals exist at each
instant, and the rates at which new individuals appear and some old ones
disappear depend on the instantaneous configuration of existing
individuals \cite{HS1978}. The corresponding Markov generators have
a~natural heuristic representation in terms of birth and death
intensities. The birth intensity $b(x,\ga)\geq0$ characterizes the
appearance of a~new point at $x\in\X$ in~the presence of a~given
configuration $\ga\in\Ga$. The death intensity $d(x,\ga)\geq0$
characterizes the probability of the event that the point $x$ of the configuration
$\ga$ disappears, depending on the location of the remaining
points of the configuration, $\ga\setminus x$.
Heuristically, the corresponding Markov generator is described by
the following expression, cf. \eqref{genGen},
\begin{multline}
(LF)(\ga ):=\sum_{x\in \ga }d(x,\ga \setminus x)\left[F(\ga
\setminus x)-F(\ga )\right]\\+\int_{\R^{d}}b(x,\ga
)\left[F(\ga \cup x)-F(\ga )\right] dx, \label{BaDGen}
\end{multline}
for proper functions $F:\Ga\rightarrow\R$.

The study of spatial birth-and-death processes was initiated by
C. Preston \cite{Pre1975}. This paper dealt with a solution of the
backward Kolmogorov equation \eqref{BKE}
under the restriction that only a~finite number of individuals are
alive at each moment of time. Under certain conditions,
corresponding processes exist and are temporally ergodic, that is,
there exists a~unique stationary distribution. Note that a more general
setting for birth-and-death processes only requires that the number of points in any compact
set remains finite at all times. A further progress in the study of these
processes was achieved by R. Holley and D. Stroock in \cite{HS1978}. They
described in detail an analytic framework for birth-and-death
dynamics. In particular, they analyzed the case of a birth-and-death process in
a bounded region.

Stochastic equations for spatial birth-and-death processes were
formulated in \cite{Gar1995}, through a~spatial version of the
time-change approach. Further, in \cite{GK2006}, these processes were
represented as solutions to a~system of stochastic equations, and
conditions for the existence and uniqueness of solutions to these
equations, as well as for the corresponding martingale problems, were
given. Unfortunately, quite restrictive assumptions on the birth and
death rates in \cite{GK2006} do not allow an application of
these results to several particular models that are interesting for
applications (see e.g. some of examples below).

A growing interest to the study of spatial birth-and-death processes, which we have recently observed, is stimulated
by (among others) an
important role which these processes play in several applications. For example,
in spatial plant ecology, a~general approach to the so-called
individual based models was developed in a~series of works, see e.g.\
\cite{BP1997,BP1999,DL2000,MDL2004} and the references therein. These
models are described as birth-and-death Markov processes in the
configuration space $\Gamma$ with specific rates $b$ and $d$
which reflect biological notions such as competition, establishment,
fecundity etc. Other examples of birth-and-death processes may be
found in mathematical physics. In particular, the Glauber-type
stochastic dynamics in $\Gamma$ is properly associated with the grand
canonical Gibbs measures for classical gases. This gives a~possibility to
study these Gibbs measures as equilibrium states for specific
birth-and-death Markov evolutions \cite{BCC2002}. Starting with a~Dirichlet form for a~given Gibbs measure, one
can consider an equilibrium stochastic dynamics \cite{KL2005}.
However, these dynamics give the time evolution of initial distributions from a quite narrow class. Namely, the class of admissible initial
distributions is essentially reduced to the states which are absolutely
continuous with respect to the invariant measure. Below we
construct non-equilibrium stochastic dynamics which may have a much
wider class of initial states.

This approach was successfully applied to the construction and
analysis of state evolutions for different versions of the
Glauber dynamics \cite{KKZ2006,FKKZ2010,FKK2010} and for some
spatial ecology models \cite{FKK2009}. Each of the considered models
required its own specific version of the construction of a semigroup, which
takes into account particular properties of corresponding birth and
death rates.

In this Section, we realize a~general approach considered in Section 2
to the construction of the state evolution corresponding to the
birth-and-death Markov generators. We present conditions on the birth
and death intensities which are sufficient for the existence of
corresponding evolutions as strongly continuous semigroups in proper
Banach spaces of correlation functions satisfying the Ruelle-type bounds.
Also we consider weaker assumptions on these intensities which provide
the corresponding evolutions for finite time intervals
in scales of Banach spaces as above.

\subsection{Expressions for $\widehat{L}$ and $\widehat{L}^*$. Examples of rates $b$ and $d$}

We always suppose that rates $d,b:\X\times\Ga\to[0;+\infty]$ from \eqref{BaDGen} satisfy the following assumptions
    \begin{align}
        &d(x,\eta), b(x,\eta)>0, &&  \text{$\eta\in\Ga_0\setminus\{\emptyset\}$, $x\in\X\setminus\eta$},\label{condbad1}\\
        &d(x,\eta), b(x,\eta)<\infty, &&  \text{$\eta\in\Ga_0$, $x\in\X\setminus\eta$},\label{condbad2}\\
        &\int_M \bigl( d(x,\eta)+ b(x,\eta)\bigr) d\la(\eta)<\infty,  &&
        \text{$M\in {\B}(\Ga_0)$ bounded,  a.a. $x\in\X$},\label{condbad3}\\
        &\int_\La \bigl( d(x,\eta)+ b(x,\eta)\bigr) dx<\infty, &&
        \text{$\eta\in\Ga_0$, $\La\in\Bb$}.\label{condbad4}
      \end{align}

\begin{proposition}\label{LF-Ffin}
  Let conditions \eqref{condbad1}--\eqref{condbad4} hold. The for any $G\in\Bbs$ and $F=KG$ one has
  $LF\in\Ffin$.
\end{proposition}
\begin{proof}
  By \eqref{estBbs}, there exist $\La\in\Bb$, $N\in\N$, $C>0$ (dependent on $G$) such that
  \begin{align*}
    \bigl| F(\ga\setminus x)- F(\ga)\bigr| &\leq C\1_\La(x) \bigl(1+|\ga_\La|\bigr)^N, \qquad x\in\ga, \ga \in\Ga,\\
    \bigl| F(\ga\cup x)- F(\ga)\bigr| &\leq C\1_\La(x) \bigl(2+|\ga_\La|\bigr)^N, \qquad \ga\in\Ga, x\in\X\setminus\ga.
  \end{align*}
  Then, by \eqref{condbad2}, \eqref{condbad4}, for any $\eta\in\Ga_0$,
 \begin{equation*}
    \bigl| (LF)(\eta)\bigr| \leq C \bigl(2+|\eta_\La|\bigr)^N \Bigl(\sum_{x\in\eta_\La}d(x,\eta\setminus x)
    + \int_\La b(x,\eta)dx\Bigr)<\infty.
 \end{equation*}
 The statement is proved.
\end{proof}

We start from the deriving of the expression for $\widehat{L}=K^{-1}LK$.
\begin{proposition}
For any $G\in\Bbs$ the following formula
holds
\begin{equation}
\begin{split}
(\widehat{L}G)(\eta ) =&-\sum_{\xi \subset \eta
}G(\xi )\sum_{x\in \xi }\bigl(K^{-1}d(x,\cdot\cup\xi\setminus
x)\bigr)(\eta\setminus
\xi)\\
&+\sum_{\xi \subset \eta }\int_{\R^{d}}\,G(\xi \cup
x)\bigl(K^{-1}b(x,\cdot\cup\xi)\bigr)(\eta\setminus \xi) dx,
\qquad \eta\in\Ga_0.
\end{split}\label{newexpr}
\end{equation}
\end{proposition}

\begin{proof}
First of all, note that, by \eqref{condbad2} and \eqref{k-1transform}, the expressions $\bigl(K^{-1}b(x,\cdot\cup\xi)\bigr)(\eta)$ and $\bigl(K^{-1}d(x,\cdot\cup\xi)\bigr)(\eta)$ have sense.
Recall that $G\in\Bbs$ implies $F\in\Fcyl\subset\Ffin$, then, by \eqref{K-transform},
\begin{equation}\label{expr-d}
    \begin{split}
F(\ga \setminus x)- F(\ga )\ &=\sum_{\eta\Subset\ga\setminus x}G(\eta)-\sum_{\eta\Subset\ga}G(\eta)=\\&=-\sum_{\eta\Subset\ga\setminus x}G(\eta\cup x)= -(K(G(\cdot\cup x)))(\ga \setminus x).
\end{split}
\end{equation}
In the same way, for $x\notin\ga$, we derive
\begin{equation}\label{expr-b}
    F(\ga \cup x)- F(\ga ) = (K(G(\cdot\cup x)))(\ga ).
\end{equation}

By Proposition~\ref{LF-Ffin}, the values of $(\widehat{L}G)(\eta)$ are finite, and, by  \eqref{k-1transform}, one can interchange order of summations and integration in the following computations, that takes into account \eqref{expr-d}, \eqref{expr-b}:
\begin{align*}
  (\widehat{L}G)(\eta)=&-\sum_{\zeta\subset\eta }(-1)^{|\eta\setminus\zeta|}\sum_{x\in\zeta}
  d(x,\zeta\setminus x)\sum_{\xi\subset\zeta\setminus x}G(\xi\cup x)
  \\&+ \int_\X \sum_{\zeta\subset\eta }(-1)^{|\eta\setminus\zeta|}b(x,\zeta)\sum_{\xi\subset\zeta} G(\xi\cup x) \,dx,\\
  \intertext{and making substitution $\xi'=\xi\cup x\subset\zeta$, one may continue}
  =&-\sum_{\zeta\subset\eta }(-1)^{|\eta\setminus\zeta|}\sum_{\xi'\subset\zeta}
  \sum_{x\in\xi'}d(x,\zeta\setminus x)G(\xi')\\&+ \int_\X \sum_{\zeta\subset\eta }(-1)^{|\eta\setminus\zeta|}b(x,\zeta)\sum_{\xi\subset\zeta} G(\xi\cup x) \,dx.
\end{align*}
Next, for any measurable $H:\Ga_0\times\Ga_0\to\R$, one has
\begin{equation*}
    \sum_{\zeta\subset\eta}\sum_{\xi\subset\zeta}H(\xi,\zeta)
    =\sum_{\xi\subset\eta}\sum_{\substack{\zeta\subset\eta\\\xi\subset\zeta}}H(\xi,\zeta)
    =\sum_{\xi\subset\eta}\sum_{\zeta'\subset\eta\setminus\xi}H(\xi,\zeta'\cup\xi).
  \end{equation*}
Using this changing of variables rule, we continue:
\begin{align*}
  (\widehat{L}G)(\eta)
  =&-\sum_{\xi\subset\eta}\sum_{\zeta'\subset\eta\setminus\xi }(-1)^{|\eta\setminus(\xi\cup\zeta')|}
  \sum_{x\in\xi}d(x,\zeta'\cup\xi\setminus x)G(\xi)\\&+
  \int_\X \sum_{\xi\subset\eta} \sum_{\zeta'\subset\eta \setminus \xi }(-1)^{|\eta\setminus(\zeta'\cup\xi)|}b(x,\zeta'\cup\xi)G(\xi\cup x) \,dx,
\end{align*}
that yields \eqref{newexpr}, using the equality $\bigl|\eta\setminus(\xi\cup\zeta')\bigr|=\bigl|(\eta\setminus \xi)\setminus \zeta'\bigr|$ and \eqref{k-1transform}.
\end{proof}

\begin{remark}\label{hatL-matrix}
The initial value problem \eqref{ssd2} can be considered in the following matrix form, cf.~\eqref{QE-n},
  \begin{equation*}\label{ssd-n}
    \begin{split}
    \frac{\partial}{\partial t} G_t^\n(\xn)&=\bigl(\widehat{L}_{n,m}G_t^\n\bigr)(\xn), \quad t>0, \quad n\in\N_0,\\ G_t^\n(\xn)\bigr|_{t=0}&=G_0^\n(\xn).
    \end{split}
  \end{equation*}
  The expression \eqref{newexpr} shows that the matrix above has on the main diagonal the collection of operators $\widehat{L}_{n,n}$, $n\in\N_0$ which forms the following operator on functions on $\Ga_0$:
  \begin{equation}\label{diag-oper}
    (\widehat{L}_\mathrm{diag}G)(\eta) = -D(\eta) G(\eta) + \sum_{y\in\eta}\int_\X
    G\bigl((\eta\setminus y)\cup x\bigr) \bigl[b(x,\eta)-b(x,\eta\setminus y)\bigr]\,dx,
  \end{equation}
  where the term in the square brackets is equal, by \eqref{k-1transform}, to $\bigl(K^{-1}b(x,\cdot\cup(\eta\setminus y))\bigr)(\{y\})$. Next, by \eqref{newexpr}, there exist only one non-zero upper diagonal in the matrix. The corresponding operator is
  \begin{equation}\label{upper-oper}
    (\widehat{L}_\mathrm{upper}G)(\eta) =\int_\X G(\eta\cup x) b(x,\eta) \,dx,
  \end{equation}
  since $\bigl(K^{-1}b(x,\cdot\cup \eta)\bigr)(\emptyset)=b(x,\eta)$. The rest part of the~expression~\eqref{newexpr} corresponds to the low diagonals.
\end{remark}

As we mentioned above, to derive the expression for $\widehat{L}^*$ we need some discrete analog of the integration by parts formula. As such, we will use the partial case of the
well-known lemma (see e.g. \cite{KMZ2004}):
\begin{lemma}
\label{Minlos} For any measurable function $H:\Ga_0\times\Ga_0\times
\Ga_0\rightarrow{\R}$
\begin{equation} \label{minlosid}
\int_{\Ga _{0}}\sum_{\xi \subset \eta }H\left( \xi ,\eta \setminus
\xi ,\eta \right) d\la \left( \eta \right) =\int_{\Ga _{0}}\int_{\Ga
_{0}}H\left( \xi ,\eta ,\eta \cup \xi \right) d\la \left( \xi
\right) d\la \left( \eta \right)
\end{equation}
if at least one side of the equality is finite for $|H|$.
\end{lemma}
In particular, if $H(\xi,\cdot,\cdot)\equiv0$ if only $|\xi|\neq1$ we obtain an analog of \eqref{Mecke}, namely,
 \begin{equation}\label{laMecke}
 \int_{\Ga_0} \sum_{x\in\eta} h(x,\eta\setminus x,\eta)d\la(\eta)
 =\int_{\Ga_0} \int_\X h(x,\eta,\eta\cup x) dx d\la(\eta),
\end{equation}
for any measurable function $h:\X\times\Ga_0\times
\Ga_0\to\R$ such that both sides make sense.

Using this, one can derive the explicit form of $\widehat{L}^*$.
\begin{proposition}\label{propLtriangle}
  For any $k\in\Bbs$ the following formula holds
  \begin{equation}\begin{split}
(\widehat{L}^{\ast}k)(\eta )=&-\sum_{x\in \eta}\int_{\Ga_0} k(\zeta
\cup \eta )\bigl( K^{-1}d(x,\cdot\cup \eta\setminus x)\bigr)
(\zeta)d\la (\zeta ) \\
& +\sum_{x\in \eta}\int_{\Ga_0}k(\zeta \cup (\eta \setminus
x))\bigl( K^{-1}b(x,\cdot\cup \eta\setminus x)\bigr) (\zeta)d\la
(\zeta),
\end{split}\label{defLtriangle}
  \end{equation}
where $\widehat{L}^{\ast}k$ is defined by \eqref{dualoper}.
\end{proposition}

\begin{proof}
  Using Lemma~\ref{Minlos}, \eqref{dualoper}, \eqref{newexpr}, we obtain for any $G\in\Bbs$
  \begin{align*}
    &\int_{\Ga_0} G(\eta) \bigl(\widehat{L}^{\ast}k\bigr)(\eta)\,d\la(\eta)\\
    =&-\int_{\Ga_0} \sum_{\xi \subset \eta
}G(\xi )\sum_{x\in \xi }\bigl(K^{-1}d(x,\cdot\cup\xi\setminus
x)\bigr)(\eta\setminus
\xi)k(\eta)\,d\la(\eta)\\
&+\int_{\Ga_0} \sum_{\xi \subset \eta }\int_{\R^{d}}\,G(\xi \cup
x)\bigl(K^{-1}b(x,\cdot\cup\xi)\bigr)(\eta\setminus \xi) \,dxk(\eta)\,d\la(\eta)\\
=& -\int_{\Ga_0} \int_{\Ga_0} G(\xi )\sum_{x\in \xi }\bigl(K^{-1}d(x,\cdot\cup\xi\setminus
x)\bigr)(\eta)k(\eta\cup\xi)\,d\la(\eta)\,d\la(\xi)\\
&+\int_{\Ga_0} \int_{\Ga_0}\int_{\R^{d}}\,G(\xi \cup
x)\bigl(K^{-1}b(x,\cdot\cup\xi)\bigr)(\eta) \,dx k(\eta\cup\xi)\,d\la(\eta)\,d\la(\xi).
  \end{align*}
Applying \eqref{laMecke} for the second term, we easily obtain the statement. The correctness of using \eqref{Mecke} and \eqref{laMecke} follow from the assumptions that $G,k\in\Bbs$, therefore, all integrals over $\Ga_0$ will be taken, in fact, over some bounded $M\in\B(\Ga_0)$. Then, using \eqref{condbad3}, \eqref{condbad4}, we obtain that the all integrals are finite.
\end{proof}

\begin{remark}
  Accordingly to Remark~\ref{hatL-matrix} (or just directly from \eqref{defLtriangle}), we have that the matrix corresponding to \eqref{QE-n} has the main diagonal given by
  \begin{multline}\label{diag-oper-ast}
    (\widehat{L}^*_\mathrm{diag}k)(\eta) = -D(\eta) k(\eta) \\+ \sum_{x\in\eta}\int_\X
    k\bigl((\eta\setminus x)\cup y\bigr) \bigl[b\bigl(x,(\eta\setminus x)\cup y\bigr)-b\bigl(x,\eta\setminus x\bigr)\bigr]\,dy,
  \end{multline}
  where we have used \eqref{laMecke}. Next, this matrix has only one non zero low diagonal, given by the expression
    \begin{equation}\label{low-oper-ast}
    (\widehat{L}^*_\mathrm{low}k)(\eta) =\sum_{x\in\eta} k(\eta\setminus x) b(x,\eta\setminus x).
  \end{equation}
  The rest part of expression \eqref{defLtriangle} corresponds to the upper diagonals.
\end{remark}

Let us consider now several examples of rates $b$ and $d$ which will appear in the following considerations (concrete examples of birth-and-death dynamics, with such rates, important for applications will be presented later). As we see from \eqref{newexpr}, \eqref{defLtriangle}, we always need to calculate expressions like $\bigl(K^{-1} a(x,\cdot\cup\xi)\bigr)(\eta)$, $\eta\cap\xi=\emptyset$, where $a$ equal to $b$ or $d$. We consider the following kinds of function $a:\X\times\Ga\to\R$:
\begin{itemize}
  \item {\em Constant rate}:
        \begin{equation}\label{constrate}
            a(x,\ga)\equiv m >0.
        \end{equation}
        If we substitute $f\equiv0$ into \eqref{Kexp}, we obtain that
      \begin{equation}\label{K-1m}
        (K^{-1}m)(\eta)=m0^{|\eta|}, \quad \eta\in\Ga_0,
      \end{equation}
      where as usual $0^0:=1$, and, of course, in this case $K^{-1} a(x,\cdot\cup\xi)(\eta)$ also equal to $m0^{|\eta|}$ for any $\xi\in\Ga_0$;
  \item {\em Linear rate}:
    \begin{equation}\label{linearrate}
        a(x,\ga)=\langle c(x-\cdot),\ga\rangle=\sum_{y\in\ga}c(x-y),
    \end{equation}
    where $c$ is a potential like in Example~\ref{ex:energy}. Any such $c$ for a given $x\in\X$ defines a function $C_x:\Ga_0\to\R$ such that $C_x(\eta)=0$ for all $\eta\notin\Ga^{(1)}$ and, for any $\eta\in\Ga^{(1)}$, $y\in\X$ with $\eta=\{y\}$, we have $C_x(\eta)=c(x-y)$. Then, in this case, taking into account \eqref{K-1m} and the obvious equality
    \begin{equation}\label{linearproperty}
        \langle c(x-\cdot),\eta\cup\xi\rangle=\langle c(x-\cdot),\eta\rangle+\langle c(x-\cdot),\xi\rangle,
    \end{equation}
    we obtain
    \begin{equation}\label{K-1-of-lin}
        \bigl(K^{-1} a(x,\cdot\cup\xi)\bigr)(\eta)=
        a(x,\xi)0^{|\eta|}+C_x(\eta), \quad \eta\in\Ga_0.
    \end{equation}
  \item {\em Exponential rate}:
  \begin{equation}\label{exprate}
        a(x,\ga)=e^{\langle c(x-\cdot),\ga\rangle}=\exp\Bigl\{\sum_{y\in\ga}c(x-y)\Bigr\},
    \end{equation}
    where $c$ as above. Taking into account \eqref{linearproperty} and \eqref{Kexp}, we obtain that in this case
    \begin{equation}\label{K-1-of-exp}
        \bigl(K^{-1} a(x,\cdot\cup\xi)\bigr)(\eta)=
        a(x,\xi)e_\la\bigl(e^{c(x-\cdot)}-1,\eta\bigr), \quad \eta\in\Ga_0.
    \end{equation}
    \item {\em Product of linear and exponential rates}:
      \begin{equation}\label{exprate1}
        a(x,\ga)=\langle c_1(x-\cdot),\ga\rangle e^{\langle c_2(x-\cdot),\ga\rangle},
    \end{equation}
    where $c_1$ and $c_2$ are potentials as before. Then we have
    \begin{equation}\label{asd}
      a(x,\eta\cup\xi)=a(x,\eta)e^{\langle c_2(x-\cdot),\xi\rangle}
      +a(x,\xi)e^{\langle c_2(x-\cdot),\eta\rangle}.
    \end{equation}
    Next, by \eqref{k-1transform},
    \begin{align}
      \bigl(K^{-1}a(x,\cdot)\bigr)(\eta)=&\sum_{\zeta\subset\eta}(-1)^{|\eta\setminus\zeta|}
      \sum_{y\in\zeta}c_1(x-y) e^{\langle c_2(x-\cdot),\zeta\rangle}\notag\\
      =&\sum_{y\in\eta}c_1(x-y)\sum_{\zeta\subset\eta\setminus y}(-1)^{|(\eta\setminus y)\setminus\zeta|} e^{\langle c_2(x-\cdot),\zeta\cup y\rangle},\notag\\\intertext{and taking into account \eqref{Kexp},}
      =&\sum_{y\in\eta}c_1(x-y)e^{c_2(x-y)}e_\la\bigl( e^{ c_2(x-\cdot)}-1,\eta\setminus y\bigr).\label{ase}
    \end{align}
    By \eqref{asd} and \eqref{ase}, we finally obtain that in this case
    \begin{equation}\label{K-1-of-lin-exp}
        \begin{split}
        \bigl(K^{-1} a(x,\cdot\cup\xi)\bigr)(\eta)=&\,e^{\langle c_2(x-\cdot),\zeta\rangle}\sum_{y\in\eta}c_1(x-y)e^{c_2(x-y)}e_\la\bigl( e^{ c_2(x-\cdot)}-1,\eta\setminus y\bigr)
        \\&+a(x,\xi)e_\la\bigl( e^{ c_2(x-\cdot)}-1,\eta\bigr), \quad \eta\in\Ga_0.
        \end{split}
    \end{equation}
    \item {\em Mixing of linear and exponential rates}:
    \begin{equation}\label{mixrate}
        a(x,\ga)=\sum_{y\in\ga}c_1(x-y) e^{\langle c_2(y-\cdot),\ga\setminus y\rangle}.
    \end{equation}
    We have
    \begin{align*}
      a(x,\eta\cup\xi)=&\sum_{y\in\eta}c_1(x-y) e^{\langle c_2(y-\cdot),\eta\setminus y\rangle}e^{\langle c_2(y-\cdot),\xi\rangle}
      \\&+\sum_{y\in\xi}c_1(x-y) e^{\langle c_2(y-\cdot),\eta\rangle}e^{\langle c_2(y-\cdot),\xi\setminus y\rangle}.
    \end{align*}
    Then, similarly to \eqref{K-1-of-lin-exp}, we easily derive
    \begin{align*}
      \bigl(K^{-1}a(x,\cdot\cup\xi)\bigr)(\eta)=&\sum_{y\in\eta}c_1(x-y)
      e_\la\bigl(e^{c_2(y-\cdot)}-1,\eta\setminus y\bigr)e^{\langle c_2(y-\cdot),\xi\rangle}
      \\&+\sum_{y\in\xi}c_1(x-y) e_\la\bigl(e^{c_2(y-\cdot)}-1,\eta\bigr)e^{\langle c_2(y-\cdot),\xi\setminus y\rangle}.
    \end{align*}
\end{itemize}

Using the similar arguments one can consider polynomials rates and their compositions with exponents as well.

\subsection{Semigroup evolutions in the space of quasi-observables}\label{subsect-evol-qo}

We proceed now to the construction of a semigroup in the space $\L_C$, $C>0$, see \eqref{space1}, which has a generator, given by $\widehat{L}$, with a proper domain. To define such domain, let us set
\begin{align}
D\left( \eta \right) :=&\sum_{x\in \eta }d\left( x,\eta
\setminus x\right) \geq 0,\quad\eta \in \Ga _{0}; \label{mult}\\
\D :=&\left\{ G\in \L_{C}~|~D\left( \cdot \right) G\in
\L_{C}\right\}.\label{dom}
\end{align}
Note that $\Bbs \subset \D$
and $\Bbs$ is a~dense set in $\L_C$. Therefore, $\D$ is
also a dense set in $\L_C$. We will show now that $(\widehat{L},\D)$ given by \eqref{newexpr}, \eqref{dom} generates $C_0$-semigroup on $\L_C$ if only `the {\em full energy of death}', given by \eqref{mult}, is big enough.

\begin{theorem}\label{th1}
Suppose that there exists $a_{1}\geq1$, $a_2>0$ such that for all
$\xi \in \Ga _{0}$ and a.a. $x\in\X$
\begin{align}
\sum_{x\in\xi}\int_{\Ga _{0}}\left\vert K^{-1}d\left( x,\cdot
\cup \xi \setminus x\right) \right\vert \left( \eta \right)
C^{\left\vert \eta \right\vert }d\la \left( \eta \right) &\leq
a_{1} D(\xi) ,
\label{est-d} \\
\sum_{x\in\xi}\int_{\Ga _{0}}\left\vert K^{-1}b\left( x,\cdot
\cup \xi \setminus x\right) \right\vert \left( \eta \right)
C^{\left\vert \eta \right\vert }d\la \left( \eta \right) &\leq
a_{2}D(\xi) . \label{est-b}
\end{align}
and, moreover,
\begin{equation}\label{asmall}
a_1+\frac{a_2}{C}<\frac{3}{2}.
\end{equation}
Then $(\widehat{L},\D)$ is the generator of a~holomorphic semigroup
$\widehat{T}\left( t\right) $ on $\L_{C}$.
\end{theorem}

\begin{remark}
  Having in mind Remark \ref{hatL-matrix} one can say that the idea of the proof is to show that the multiplication part of the diagonal operator \eqref{diag-oper} will dominates on the rest part of the operator matrix $\bigl(\widehat{L}_{n,m}\bigr)$ provided the conditions \eqref{est-d}, \eqref{est-b} hold. Note also that, by~\eqref{LP-meas-def}, \eqref{Ga0}, \eqref{speccases}, \eqref{mult}, the l.h.s of \eqref{est-d} is equal to
  \[
  D(\xi) + \sum_{x\in\xi}\int_{\Ga _{0}\setminus \{\emptyset\}}\left\vert K^{-1}d\left( x,\cdot
\cup \xi \setminus x\right) \right\vert \left( \eta \right)
C^{\left\vert \eta \right\vert }d\la \left( \eta \right).
  \]
  This is the reason to demand that $a_1$ should be not less than $1$.
\end{remark}
\begin{proof}[Proof of Theorem \ref{th1}]
Let us consider the multiplication operator $\left( L_{0},\D\right)
$ on $ \L_{C}$ given by
\begin{equation} (L_{0}G)(\eta )=-D\left( \eta \right) G(\eta
),\quad G\in \D,\ \ \eta \in \Ga _{0}. \label{L0oper}
\end{equation}
We recall that a~densely defined closed operators $A$ on
$\L_{C}$ is called sectorial of angle $\omega\in
(0,\,\frac{ \pi }{2})$ if its resolvent set $\rho (A)$ contains the
sector
\begin{equation*}
\mathrm{Sect}\left( \frac{\pi }{2}+\omega \right) :=\left\{ z \in
\mathbb{C}\Bigm||\arg z |<\frac{\pi }{2}+\omega \right\} \setminus\{0\}
\end{equation*} and for each $\eps\in(0;\omega)$ there exists
$M_\eps\geq 1$ such that
\begin{equation}\label{resbound}
||R(z,A)||\leq \frac{M_{\eps }}{|z|}
\end{equation}
for all $z\neq 0$ with $|\arg z |\leq \dfrac{\pi }{2}+\omega -\eps.$
Here and below we will use notation
\[
R(z,A):=(z\1 -A)^{-1}, \quad z\in\rho(A).
\]
The set of all sectorial operators of angle $\omega\in (0,\,\frac{
\pi }{2})$ in $\L_C$ we denote by $\mathcal{H} _{C}(\omega
)$. Any $A\in\mathcal{H} _{C}(\omega )$ is a~generator of a~bounded
semigroup $T(t)$ which is holomorphic in the sector $|\arg
\,t|<\omega $ (see e.g. \cite[Theorem II.4.6]{EN2000}). One can
prove the following lemma.
\begin{lemma}\label{lem1}
The operator $\left( L_{0},\D\right) $ given by \eqref{L0oper} is a
generator of a~contraction semigroup on $\L_{C}.$ Moreover,
$L_{0}\in \mathcal{H}_{C}(\omega)$ for all $\omega \in (0,\,\frac{
\pi }{2})$ and \eqref{resbound} holds with
$M_\eps=\frac{1}{\cos\omega}$ for all $\eps\in(0;\omega)$.
\end{lemma}
\begin{proof}[Proof of Lemma~\ref{lem1}]
It is not difficult to show that the densely defined operator
$L_{0}$ is closed in~$\L_{C}$. Let $0<\omega <\frac{\pi
}{2}$ be arbitrary and fixed. Clear, that for all $z \in
\mathrm{Sect}\left( \frac{\pi }{2} +\omega \right) $
\begin{equation*}
\bigl|D\left( \eta \right) +z \bigr|>0,\quad \eta \in \Ga _{0}.
\end{equation*} Therefore, for any $z \in \mathrm{Sect}\left( \frac{\pi }{2}+\omega
\right) $ the inverse operator $R(z,L_{0})=(z\1
- L_{0})^{-1}$, the action of
which is given by
\begin{equation}
\bigl( R(z,L_{0})G\bigr)(\eta )=\frac{1}{D\left( \eta \right) +z
}\,G(\eta ), \label{necf}
\end{equation} is well defined on the whole space $\L_{C}$. Moreover,
\[
|D(\eta)+z|=\sqrt{(D(\eta)+\Re z)^2+(\Im
z)^2}\geq
\begin{cases}
|z|, &\mathrm{if} \ \Re z\geq 0\\
|\Im\,z|, &\mathrm{if} \ \Re z<0
\end{cases},
\]
and for any $z\in\Sect\left(\frac{\pi}{2}+\omega\right)$
\[
|\Im\,z|=|z| |\sin \arg z|\geq|z|\left|\sin\left(\frac{\pi}{2}+\omega\right)\right|=|z|\cos\omega.
\]
As a~result,
 for any $z\in\Sect\left(\frac{\pi}{2}+\omega\right)$\begin{equation}\label{resbound-ex}
||R(z,L_{0})||\leq \frac{1}{|z |\cos\omega},
\end{equation}
that implies the second assertion. Note also that
$|D(\eta)+z|\geq\Re z$ for $\Re z>0$, hence,
\begin{equation}\label{HY}
||R(z,L_{0})||\leq
\frac{1}{\Re z},
\end{equation}
that proves the first statement by the classical
Hille--Yosida theorem.
\end{proof}

For any $G\in \Bbs $ we define
\begin{equation}
    \begin{split}
\left( L_{1}G\right) \left( \eta \right)
:=&\,(\widehat{L}G)(\eta)-(L_0G)(\eta)\\=&-\sum_{\xi\subsetneq \eta
}G(\xi )\sum_{x\in \xi }\bigl(K^{-1}d(x,\cdot\cup\xi\setminus
x)\bigr)(\eta\setminus
\xi)\\
&+\sum_{\xi \subset \eta }\int_{\R^{d}}\,G(\xi \cup
x)\bigl(K^{-1}b(x,\cdot\cup\xi)\bigr)(\eta\setminus \xi)
dx.
\end{split}\label{operL1def}
\end{equation}

Next Lemma shows that, under conditions \eqref{est-d}, \eqref{est-b}
above, the operator $L_1$ is relatively bounded by the operator
$L_0$.
\begin{lemma}\label{lem2}
Let \eqref{est-d}, \eqref{est-b} hold. Then $(L_1,\D)$ is a
well-defined operator in $\L_C$ such that
\begin{equation}
\left\Vert L_{1}R(z,L_{0})\right\Vert \leq a_1-1+\frac{a_2}{C}, \quad \Re z>0\label{relbound}
\end{equation}
and
\begin{equation}\label{deep}
\|L_1 G\|\leq \Bigl(a_1-1+\frac{a_2}{C}\Bigr)\|L_0G\|,
\quad G\in\D.
\end{equation}
\end{lemma}
\begin{proof}[Proof of Lemma~\ref{lem2}]
By Lemma~\ref{Minlos}, we have for any $G\in \L_{C}$, $\Re
z>0$
\begin{align*}
&\int_{\Ga _{0}}\biggl\vert -\sum_{\xi\subsetneq \eta
}\frac{1}{z+D(\xi)}G(\xi )\sum_{x\in \xi
}\bigl(K^{-1}d(x,\cdot\cup\xi\setminus x)\bigr)(\eta\setminus
\xi)\biggr\vert C^{\left\vert \eta \right\vert }d\la
\left( \eta \right) \\
\leq &\int_{\Ga _{0}}\sum_{\xi \subsetneq \eta
}\frac{1}{|z+D(\xi)|}\left\vert G(\xi )\right\vert \sum_{x\in \xi
}\bigl|K^{-1}d(x,\cdot\cup\xi\setminus x)\bigr|(\eta\setminus \xi)
C^{\left\vert \eta \right\vert
}d\la \left( \eta \right) \\
=&\int_{\Ga _{0}}\frac{1}{|z+D(\xi)|}\left\vert G(\xi )\right\vert
\sum_{x\in \xi }\int_{\Ga
_{0}}\bigl|K^{-1}d(x,\cdot\cup\xi\setminus
x)\bigr|(\eta)C^{\left\vert \eta \right\vert }d\la \left( \eta
\right) C^{\left\vert \xi \right\vert
}d\la \left( \xi \right) \\
&-\int_{\Ga _{0}}\frac{1}{|z+D(\eta)|}D\left( \eta \right)
\left\vert G(\eta )\right\vert C^{\left\vert \eta \right\vert }d\la
\left( \eta
\right)\\
\leq&(a_{1}-1)\int_{\Ga _{0}} \frac{1}{\Re z+D(\eta)}
D(\eta)|G(\eta)|C^{|\eta|} d\la(\eta)\leq(a_1-1)\|G\|_C,
\end{align*} and \begin{align*}
&\int_{\Ga _{0}}\biggl\vert \sum_{\xi \subset \eta }\int_{\R
^{d}}\,\frac{1}{z+D(\xi\cup x)}G(\xi \cup
x)\bigl(K^{-1}b(x,\cdot\cup\xi)\bigr) (\eta\setminus
\xi)dx\biggr\vert C^{\left\vert \eta \right\vert }d\la
\left( \eta \right) \\
\leq &\int_{\Ga _{0}}\int_{\Ga
_{0}}\int_{\R^{d}}\,\frac{1}{|z+D(\xi\cup x)|}\left\vert G(\xi \cup
x)\right\vert \left\vert K^{-1}b(x,\cdot\cup\xi) \right\vert
(\eta)dxC^{\left\vert \eta \right\vert }C^{\left\vert \xi
\right\vert }d\la \left( \xi \right) d\la \left(
\eta \right) \\
=&\frac{1}{C}\int_{\Ga _{0}}\,\frac{1}{|z+D(\xi)|}\left\vert G(\xi
)\right\vert \sum_{x\in \xi }\int_{\Ga _{0}} \left\vert
K^{-1}b(x,\cdot\cup\xi\setminus x) \right\vert (\eta)
C^{\left\vert \eta \right\vert }d\la \left( \eta \right)
C^{\left\vert \xi \right\vert }d\la \left( \xi \right) \\
\leq&\frac{a_{2}}{C}\int_{\Ga _{0}}\,\frac{1}{\Re
z+D(\xi)}\left\vert G(\xi )\right| D(\xi) C^{\left\vert \xi
\right\vert }d\la \left( \xi \right)\leq\frac{a_2}{C}\|G\|_C.
\end{align*}
Combining these inequalities we obtain \eqref{relbound}.
The same considerations yield
\begin{align*}
&\int_{\Ga _{0}}\biggl\vert -\sum_{\xi\subsetneq \eta }G(\xi
)\sum_{x\in \xi }\bigl(K^{-1}d(x,\cdot\cup\xi\setminus
x)\bigr)(\eta\setminus \xi)\biggr\vert C^{\left\vert \eta
\right\vert }d\la
\left( \eta \right) \\
&\quad+\int_{\Ga _{0}}\biggl\vert \sum_{\xi \subset \eta }\int_{\R
^{d}}\,G(\xi \cup x)\bigl(K^{-1}b(x,\cdot\cup\xi)\bigr)
(\eta\setminus \xi)dx\biggr\vert C^{\left\vert \eta \right\vert
}d\la
\left( \eta \right) \\
\leq&\left((a_{1}-1)+\frac{a_{2}}{C}\right)\int_{\Ga
_{0}}\,\left\vert G(\eta )\right| D(\eta) C^{\left\vert \eta
\right\vert }d\la \left( \eta \right),
\end{align*}
that proves \eqref{deep} as well.
\end{proof}

And now we proceed to finish the proof of the
Theorem~\ref{th1}. Let us set
\[
    \theta:=a_1+\frac{a_2}{C}-1\in\bigl(0;\frac{1}{2}\bigr).
\]
Then
$\frac{\theta}{1-\theta}\in(0;1)$. Let $\omega\in
\bigl(0;\frac{\pi}{2}\bigr)$ be such that
$\cos\omega<\frac{\theta}{1-\theta}$. Then, by the proof of
Lemma~\ref{lem1}, $L_{0}\in \mathcal{H}_{C}(\omega)$ and
$||R(z,L_0)||\leq \frac{M}{|z|}$ for all $z\neq 0$ with $|\arg z
|\leq \dfrac{\pi }{2}+\omega $, where $M:=\frac{1}{\cos \omega}$.
Then
$$\theta=\frac{1}{1+\frac{1-\theta}{\theta}}<\frac{1}{1+\frac{1}{\cos\omega}}=\frac{1}{1+M}.$$
Hence, by \eqref{deep} and the proof of \cite[Theorem
III.2.10]{EN2000}, we have that $(\widehat{L}=L_0+L_1,\D)$ is a
generator of holomorphic semigroup on $\L_C$.
\end{proof}

\begin{remark}
By \eqref{mult}, the estimates \eqref{est-d}, \eqref{est-b} are
satisfied if
\begin{align}
\int_{\Ga _{0}}\left\vert K^{-1}d\left( x,\cdot \cup \xi \right)
\right\vert \left( \eta \right) C^{\left\vert \eta \right\vert }d\la
\left( \eta \right) \leq & a_{1} d\left( x,\xi \right) , \label{est-d1} \\
\int_{\Ga _{0}}\left\vert K^{-1}b\left( x,\cdot \cup \xi \right)
\right\vert \left( \eta \right) C^{\left\vert \eta \right\vert }d\la
\left( \eta \right) \leq &a_{2}d\left( x,\xi \right) .
\label{est-b1}
\end{align}
\end{remark}

\subsection{Evolutions in the space of correlation
functions}\label{subsect-evol-cf}

In this Subsection we will use the semigroup $\widehat{T}(t)$ acting oh
the space of quasi-observables for a construction of solution to the
evolution equation \eqref{QE} on the space of correlation functions.

We denote $d\la _{C}:=C^{|\cdot |}d\la $; and the dual space $(
\L_{C})^{\prime }=\bigl(L^{1}(\Ga _{0},d\la _{C})\bigr)
^{\prime }=L^{\infty }(\Ga _{0},d\la _{C})$. As was mentioned before the space $(\L
_{C})^{\prime }$ is isometrically isomorphic to the Banach space $\K_C$ considered in \eqref{KC}--\eqref{ineqKC}. The isomorphism is given by the isometry $R_{C}$
\begin{equation}
(\L_{C})^{\prime }\ni k\longmapsto R_{C}k:=k\cdot C^{|\cdot
|}\in { \K}_{C}. \label{isometry}
\end{equation}

Recall, one may consider the duality between the Banach spaces
$\L _{C}$ and ${\K}_{C}$ given by \eqref{duality} with
\[
\left\vert \left\langle \!\left\langle G,k\right\rangle
\!\right\rangle \right\vert \leq \Vert G\Vert _{C}\cdot \Vert k\Vert
_{{ \K}_{C}}.
\]

Let $\bigl({\widehat{L}}^{\prime },\Dom({\widehat{L}}^{\prime })\bigr)$ be
an operator in $(\L_{C})^{\prime }$ which is dual to the
closed operator $\bigl( {\widehat{L}},\D\bigr)$. We consider also its
image on ${\K} _{C}$ under the isometry $R_{C}$. Namely, let
${\widehat{L}}^{\ast }=R_{C}{\widehat{L }}^{\prime }R_{C^{-1}}$ with the
domain $\Dom({\widehat{L}}^{\ast })=R_{C}\Dom({\widehat{L}}^{\prime })$.
Similarly, one can consider the adjoint semigroup $\widehat{T}^{\prime
}(t)$ in $(\L_{C})^{\prime }$ and its image $\widehat{T}^{\ast
}(t)$ in ${\K}_{C}$.

The space $\L_C$ is not reflexive,
hence, $\widehat{T}^{\ast }(t)$ is not $C_0$-semigroup in whole ${\K}_{C}$.
By e.g. \cite[Subsection~II.2.5]{EN2000}, the last
semigroup will be weak*-continuous, weak*-differentiable at $0$ and
${\widehat{L}}^\ast$ will be weak*-generator of ${\widehat{T}}^\ast(t)$.
Therefore, one has an evolution in the space of correlation
functions. In fact, we have a~solution to the evolution equation
\eqref{QE}, in a~weak*-sense. This subsection is devoted to the
study of a~classical solution to this equation. By~e.g.~\cite[Subsection~II.2.6]{EN2000}, the restriction
$\widehat{T}^{\odot }(t)$ of the semigroup $\widehat{T}^{\ast }(t)$ onto its
invariant Banach subspace $\overline{ \Dom({\widehat{L}}^{\ast })}$
(here and below all closures are in the norm of the space
${\K}_{C}$) is a~strongly continuous semigroup. Moreover, its
generator ${\widehat{L}} ^\odot$ will be a~part of ${\widehat{L}}^\ast$,
namely,
\begin{equation}\label{domLsundual}
    \Dom({\widehat{L}} ^\odot)=\Bigl\{k\in \Dom({\widehat{L}}^\ast) \Bigm|
{\widehat{L}}^\ast k\in \overline{\Dom({\widehat{L}}^\ast)}\Bigr\}
\end{equation}
and ${\widehat{L} }^\odot k ={\widehat{L}}^\ast k$ for any $k\in
\Dom({\widehat{L}}^\odot)$.

\begin{proposition}\label{pr3}
 Let \eqref{est-d}, \eqref{est-b} be satisfied. Suppose that there exists
$A>0$, $N\in \N_{0}$, $\nu\geq 1$ such that for $\xi \in \Ga _{0}$
and $x\notin \xi $
\begin{equation}
d\left( x,\xi \right) \leq
A(1+\left\vert \xi \right\vert) ^{N}\nu^{\left\vert \xi \right\vert
}, \qquad \label{D-bdd}
\end{equation} Then for any $\alpha \in \left( 0;\frac{1}{\nu}\right) $
\begin{equation}\label{incldom}
\K_{\alpha C}\subset \Dom({\widehat{L}}^{\ast }).
\end{equation}
\end{proposition}

\begin{proof}
In order to show \eqref{incldom} it is enough to verify that for any
$k\in\K_{\alpha C}$ there exists $k^*\in\K_C$ such that for any
$G\in\Dom(\widehat{L})$
\begin{equation}\label{pairdualneed}
\bigl\langle \!\bigl\langle \widehat{L} G,\,k\bigr\rangle \!\bigr\rangle
=\left\langle \!\left\langle G,\,k^*\right\rangle \!\right\rangle.
\end{equation}
By the same calculations as in the proof of Proposition~\ref{propLtriangle}, it is easy to see that \eqref{pairdualneed} is valid for any
$k\in\K_{\alpha C}$ with $k^*=\widehat{L}^{\ast }k$, where $\widehat{L}^{\ast }$ is given by \eqref{defLtriangle},
provided $k^*\in\K_C$.

To prove the last inclusion, one can estimate, by \eqref{est-d}, \eqref{est-b}, \eqref{D-bdd}, that
\begin{align*}
&C^{-\left\vert \eta \right\vert }\left\vert (\widehat{L}^{\ast }k)(\eta
)\right\vert \\
\leq &C^{-\left\vert \eta \right\vert }\sum_{x\in \eta }\int_{\Ga
_{0}}\!\left\vert k(\zeta \cup \eta )\right\vert \bigl|
K^{-1}d(x,\cdot\cup \eta\setminus x)\bigr|
(\zeta) d\la (\zeta ) \\
&+C^{-\left\vert \eta \right\vert }\sum_{x\in \eta }\int_{\Ga
_{0}}\!\,\left\vert k(\zeta \cup (\eta \setminus x))\right\vert
\bigl| K^{-1}b(x,\cdot\cup \eta\setminus x)\bigr| (\zeta)d\la
(\zeta
)\, \\
\leq &\left\Vert k\right\Vert _{\K_{\alpha C}}\alpha ^{\left\vert
\eta \right\vert }\sum_{x\in \eta }\int_{\Ga _{0}}\!\left( \alpha
C\right) ^{\left\vert \zeta \right\vert} \bigl|
K^{-1}d(x,\cdot\cup \eta\setminus x)\bigr|
(\zeta)d\la (\zeta ) \\
&+\frac{1}{\alpha C}\left\Vert k\right\Vert _{\K_{\alpha C}}\alpha
^{\left\vert \eta \right\vert }\sum_{x\in \eta }\int_{\Ga
_{0}}\!\,\left( \alpha C\right) ^{\left\vert \zeta \right\vert }
\bigl| K^{-1}b(x,\cdot\cup \eta\setminus x)\bigr|
(\zeta) d\la (\zeta ) \\
\leq &\,\left\Vert k\right\Vert _{\K_{\alpha C}}\left( a_1
+\frac{a_2}{\alpha C}\right)\alpha ^{\left\vert \eta \right\vert
}\sum_{x\in \eta
} d\left( x,\eta \setminus x\right) \\
\leq &\,A\left\Vert k\right\Vert _{\K_{\alpha C}}\left( a_1
+\frac{a_2}{\alpha C}\right)\alpha ^{\left\vert \eta \right\vert
}(1+\left\vert \eta \right\vert) ^{N+1}\nu^{\left\vert \eta
\right\vert -1}.
\end{align*} Using elementary inequality \begin{equation}
(1+t)^{b}a^{t}\leq \frac{1}{a}\left( \frac{b}{-e\ln a}\right)
^{b},~~b\geq 1,~a\in \left( 0;1\right) ,~t\geq 0, \label{bdd}
\end{equation} we have for $\alpha \nu <1$ \begin{equation*}
\esssup_{\eta \in \Ga _{0}}C^{-\left\vert \eta \right\vert
}\left\vert (\widehat{L}^{\ast }k)(\eta )\right\vert \leq \left\Vert
k\right\Vert _{\K_{\alpha C}}\left( a_1 +\frac{a_2}{\alpha
C}\right)\frac{A }{ \alpha \nu^2}\left( \frac{N+1}{-e\ln \left(
\alpha \nu\right) }\right) ^{N+1}<\infty.
\end{equation*}
The statement is proved.
\end{proof}

\begin{lemma}
Let \eqref{D-bdd} holds. We define
for any $\alpha\in (0;1)$
\begin{equation*}
\D_{\alpha }:\mathcal{=}\left\{ G\in \L_{\alpha
C}~|~D\left( \cdot \right) G\in \L_{\alpha C}\right\} .
\end{equation*} Then for any $\alpha\in (0;\frac{1}{\nu})$ \begin{equation}
\D\subset \L_{C}\subset \D_{\alpha }\subset
\L_{\alpha C} \label{subs}
\end{equation}
\end{lemma}
\begin{proof}
The first and last inclusions are obvious. To prove the second one,
we use \eqref{D-bdd}, \eqref{bdd} and obtain for any $G\in
\L_{C}$
\begin{align*}
\int_{\Ga _{0}}D\left( \eta \right)
\left\vert G\left( \eta \right) \right\vert \left( \alpha C\right)
^{\left\vert \eta \right\vert }d\la \left( \eta \right) \leq
&\int_{\Ga _{0}}\alpha ^{\left\vert \eta \right\vert }\sum_{x\in
\eta }A(1+\left\vert \eta \right\vert) ^{N}\nu^{\left\vert \eta
\right\vert -1}\left\vert G\left( \eta \right) \right\vert
C^{\left\vert \eta \right\vert }d\la \left( \eta \right)
\\\leq &\,\mathrm{const}\int_{\Ga _{0}}\left\vert G\left(
\eta \right) \right\vert C^{\left\vert \eta \right\vert }d\la \left(
\eta \right) <\infty.
\end{align*}
The statement is proved.
\end{proof}

\begin{proposition}
Let \eqref{est-d}, \eqref{est-b}, and \eqref{D-bdd} hold with
\begin{equation}\label{asmallalpha}
a_1+\frac{a_2}{\alpha C}<\frac{3}{2}
\end{equation}
for some $\alpha\in (0;1)$. Then $(\widehat{L},\D_\alpha)$ is
a~generator of a~holomorphic semigroup $\widehat{T}_\alpha\left(
t\right) $ on $\L_{\alpha C}$.
\end{proposition}

\begin{proof}
The proof is similar to the proof of Theorem~\ref{th1}, taking into
account that bounds \eqref{est-b}, \eqref{est-d} imply the same
bounds for $\alpha C$ instead of $C$. Note also that
\eqref{asmallalpha} is stronger than \eqref{asmall}.
\end{proof}

\begin{proposition}\label{hint}
Let \eqref{est-d}, \eqref{est-b}, and \eqref{D-bdd} hold with
\begin{equation}\label{nusmall}
1\leq\nu<\frac{C}{a_2}\biggl(\frac{3}{2}-a_1\biggr).
\end{equation}
Then, for any $\alpha$ with
\begin{equation}\label{intervalalpha}
    \dfrac{a_{2}}{C\bigl(\frac{3}{2}-a_{1}\bigr)}<\alpha<\dfrac{1}{\nu},
\end{equation}
the set $\overline{ \K_{\alpha C}}$ is a~$\widehat{T}^{\odot
}(t)$-invariant Banach subspace of $\K_{C}$. Moreover, the set $\K_{\alpha C}$ is $\widehat{T}^{\odot}(t)$-invariant too.
\end{proposition}

\begin{proof} First of all note that the
condition on $\alpha$ implies \eqref{asmallalpha}. Next, we prove
that $\widehat{T}_{\alpha }\left( t\right) G=\widehat{T} \left( t\right) G$
for any $G\in \L_{C}\subset \L_{\alpha C}$. Let
$\widehat{L}_\alpha=(\widehat{L},\D_{\alpha })$ is the operator in
$\L_{\alpha C}$. There exists $\omega>0$ such that
$(\omega;+\infty)\subset\rho(\widehat{L})\cap\rho(\widehat{L}_\alpha)$, see
e.g. \cite[Section III.2]{EN2000}. For some fixed $z
\in(\omega;+\infty)$ we denote by $R(z,\widehat{L})=\bigl( z\1-
\widehat{L}\bigr) ^{-1}$ the resolvent of $(\widehat{L},\D)$ in
$\L_{C}$ and by $R(z,\widehat{L}_\alpha)=\bigl(
z\1-\widehat{L}_\alpha \bigr) ^{-1} $ the resolvent of $\widehat{L}_\alpha$
in $\L_{\alpha C}$. Then for any $G\in \L_{C}$ we
have $R(z,\widehat{L})G\in \D\subset \D_{\alpha }$ and \begin{equation*}
R(z,\widehat{L})G-R(z,\widehat{L}_\alpha)G=R(z,\widehat{L}_\alpha)\bigl( (z
\1-\widehat{L}_\alpha) -\bigl( z\1-\widehat{L}) \bigr)R(z,\widehat{L})G=0,
\end{equation*}
since $\widehat{L}_\alpha=\widehat{L}$ on $\D$. As a~result,
$\widehat{T}_{\alpha }\left( t\right) G=\widehat{T}\left( t\right) G$ on
$\L_C$.

Note that for any $G\in \L_{C}\subset \L_{\alpha
C}$ and for any $k\in {\K}_{\alpha C}\subset {\K}_{C}$ we have $
\widehat{T}_{\alpha }(t)G\in \L_{\alpha C}$ and
\begin{equation*}
\left\langle \!\!\left\langle \widehat{T}_{\alpha }(t)G,k\right\rangle
\!\!\right\rangle =\left\langle \!\!\left\langle G,\widehat{T}_{\alpha
}^{\ast }(t)k\right\rangle \!\!\right\rangle ,
\end{equation*} where, by the same construction as before, $\widehat{T}_{\alpha }^{\ast
}(t)k\in {\K}_{\alpha C}$. But $G\in \L_{C}$, $k\in {\K}
_{C}$ implies
\begin{equation*}
\left\langle \!\!\left\langle \widehat{T}_{\alpha }(t)G,k\right\rangle
\!\!\right\rangle =\left\langle \!\!\left\langle
\widehat{T}(t)G,k\right\rangle \!\!\right\rangle =\left\langle
\!\!\left\langle G,\widehat{T}^{\ast }(t)k\right\rangle
\!\!\right\rangle .
\end{equation*} Hence, $\widehat{T}^{\ast }(t)k=\widehat{T}_{\alpha }^{\ast }(t)k\in {\K} _{\alpha C}$ that proves the statement due to continuity of the family $\widehat{ T}^{\ast }(t)$.
\end{proof}

By e.g. \cite[Subsection II.2.3]{EN2000}, one can consider the restriction $\widehat{T}^{\odot \alpha}(t)$ of the semigroup $\widehat{T}^{\odot}(t)$ onto $\overline{\K_{\alpha C}}$. It will be strongly continuous semigroup with the generator
$\widehat{L}^{\odot \alpha}$ which is a~restriction of ${\widehat{L}}^\odot$
onto $\overline{\K_{\alpha C}}$. Namely, cf.~\ref{domLsundual},
\begin{equation}\label{domLsundualalpha}
    \Dom({\widehat{L}} ^{\odot\alpha})=\Bigl\{k\in \overline{\K_{\alpha C}} \Bigm|
{\widehat{L}}^\ast k\in \overline{\K_{\alpha C}}\Bigr\},
\end{equation}
and ${\widehat{L} }^{\odot \alpha}k ={\widehat{L}}^\odot k={\widehat{L}}^\ast k$ for any $k\in\overline{\K_{\alpha C}}$. In the other words, ${\widehat{L}} ^{\odot\alpha}$ is a part of ${\widehat{L}}^\ast$.

And now we may proceed to the main statement of this Subsection.
\begin{theorem}\label{ThClasCol}
  Let \eqref{est-d}, \eqref{est-b}, \eqref{D-bdd}, and \eqref{nusmall} hold, and let $\alpha$ be chosen as in \eqref{intervalalpha}. Then for any $k_0\in\overline{\K_{\alpha C}}$ there exists a unique classical solution to~\eqref{QE} in the space $\overline{\K_{\alpha C}}$, and this solution is given by $k_t=\widehat{T}^{\odot \alpha}(t)k_0$. Moreover, $k_0\in\K_{\alpha C}$ implies $k_t\in\K_{\alpha C}$.
\end{theorem}
\begin{proof}
  We recall that $(\widehat{L}, \D)$ is a densely defined closed operator on $\L_C$ (as a generator of a $C_0$-semigroup $\widehat{T}(t)$). Then, by e.g. \cite[Lemma 1.4.1]{vNee1992}, for the dual operator $\bigl( \widehat{L}^*, \Dom(\widehat{L}^*\bigr)$ we have that $\rho(\widehat{L}^*)=\rho(\widehat{L})$ and, for any $z\in\rho(\widehat{L})$, $R(z,\widehat{L}^*)=R(z,\widehat{L})^*$. In particular,
  \begin{equation}\label{eqnorm}
    \bigl\Vert R(z,\widehat{L}^*)\bigr\Vert=\bigl\Vert R(z,\widehat{L})^*\bigr\Vert
    =\bigl\Vert R(z,\widehat{L})\bigr\Vert.
  \end{equation}
  Next, if we denote by $R(z,\widehat{L})^\odot$ the restriction of $R(z,\widehat{L})^\ast$ onto $R(z,\widehat{L})^\ast$-invariant space $\overline{\Dom(\widehat{L}^*\bigr)}$ then, by e.g. \cite[Theorem 1.4.2]{vNee1992},
  $\rho(\widehat{L}^\odot)=\rho(\widehat{L}^*)$ and, for any $z\in\rho(\widehat{L}^*)=\rho(\widehat{L})$, $R(z,\widehat{L}^\odot)=R(z,\widehat{L})^\odot$. Therefore, by \eqref{eqnorm},
  \[
  \bigl\Vert R(z,\widehat{L}^\odot)\bigr\Vert\leq\bigl\Vert R(z,\widehat{L})\bigr\Vert.
  \]
  Then, taking into account that by Theorem \ref{th1} the operator $( \widehat{L}, \D)$ is a generator of the holomorphic semigroup $\widehat{T}(t)$, we immediately conclude that the same property has the semigroup $\widehat{T}^\odot(t)$ with the generator $\bigl( \widehat{L}^\odot, \Dom(\widehat{L}^\odot)\bigr)$ in the space $\overline{\Dom(\widehat{L}^*\bigr)}$.

  As a result, by e.g. \cite[Corollary 4.1.5]{Paz1983}, the initial value problem \eqref{QE} in the Banach space $\overline{\Dom(\widehat{L}^*\bigr)}$ has a~unique classical solution for {\em any} $k_0\in\overline{\Dom(\widehat{L}^*\bigr)}$. In particular, it means that the solution $k_t=\widehat{T}^\odot(t)k_0$ is continuously differentiable in $t$ w.r.t. the norm of $\overline{\Dom(\widehat{L}^*\bigr)}$ that is the norm $\|\cdot\|_{\K_C}$, and also $k_t\in\Dom(\widehat{L}^\odot)$. But by Proposition~\ref{hint}, the space $\overline{\K_{\alpha C}}$ is $\widehat{T}^\odot(t)$-invariant. Hence, if we consider now the initial value $k_0\in\overline{\K_{\alpha C}}\subset\overline{\Dom(\widehat{L}^*\bigr)}$ we obtain with a necessity that $k_t=\widehat{T}^\odot(t)k_0=\widehat{T}^{\odot\alpha}(t)k_0\in\overline{\K_{\alpha C}}$. Therefore, $k_t\in\overline{\K_{\alpha C}}\bigcap \Dom(\widehat{L}^\odot)=\Dom(\widehat{L}^{\odot\alpha})$ (see again \cite[Subsection II.2.3]{EN2000}) and, recall, $k_t$ is continuously differentiable~in~$t$ w.r.t.~the~norm $\|\cdot\|_{\K_C}$ that~is~the~norm in $\overline{\K_{\alpha C}}$. This completes the proof of the first statement. The second one follows directly now from Proposition~\ref{hint}.
\end{proof}

\subsection{Examples of dynamics}

We proceed now to describing the concrete birth-and-death dynamics which are important for different application. We will consider the explicit conditions on parameters of systems which imply the general conditions on rates $b$ and $d$ above.
For simplicity of notations we denote the l.h.s. of \eqref{est-d} and \eqref{est-b} by $I_d(\xi)$ and $I_b(\xi)$, $\xi\in\Ga_0$, correspondingly.

\begin{example} (Surgailis dynamics)
Let the rates $d$ and $b$ are independent on configuration variable, namely,
\begin{equation}\label{SurgRates}
    d(x,\ga)=m(x), \qquad b(x,\ga)=z(x), \quad x\in\X,\ga\in\Ga,
\end{equation}
where $0<m,z\in L^\infty(\X)$. Then, by \eqref{K-1m} we obtain that
\[
I_d(\xi) = \langle m,\xi\rangle= D(\xi), \qquad I_b(\xi)=\langle z,\xi\rangle, \quad\xi\in\Ga_0.
\]
Therefore, \eqref{est-d}, \eqref{est-b}, \eqref{asmall} hold if only
\begin{equation}\label{SurgCond}
    z(x)\leq a m(x), \quad x\in\X
\end{equation}
with any
\begin{equation}\label{Surga2}
    0<a<\frac{C}{2}.
\end{equation}
Clearly, in this case \eqref{D-bdd} holds with $N=0$, $\nu=1$, therefore, the condition \eqref{intervalalpha} is just
\begin{equation}\label{Surgalpha}
    \frac{2a}{C}<\alpha<1.
\end{equation}
The case of constant (in space) $m$ and $\sigma$ was considered in \cite{Fin2010Surg}. Similarly to that results, one can derive the explicit expression for the solution to the initial value problem \eqref{QE} considered point-wise in $\Ga_0$, namely,
\begin{equation}\label{SurgSol}
    k_t(\eta)=e_\la(e^{-tm},\eta)\sum_{\xi\subset\eta }e_\la\Bigl(\frac{z}{m}\bigl(e^{tm}-1\bigr),\xi\Bigr)k_0(\eta\setminus\xi), \quad \eta\in\Ga_0.
\end{equation}
Note that, using \eqref{SurgSol}, it can be possible to show directly that the statement of Theorem~\ref{ThClasCol} still holds if we drop $2$ in \eqref{Surga2} and \eqref{Surgalpha}.
\end{example}

\begin{example} (Glauber-type dynamics). \label{ex-Gl}
Let $L$ be given by \eqref{BaDGen} with
\begin{align}
d(x,\ga\setminus x)&=m(x)\exp\Bigl\{s\sum_{y\in\ga\setminus x}\phi(x-y)\Bigr\}, &&x\in\ga,\ \ga\in\Ga,\label{Gl-d}\\
b(x,\ga)&=z(x)\exp\Bigl\{(s-1)\sum_{y\in\ga}\phi(x-y)\Bigr\}, &&x\in\X\setminus\ga,\ \ga\in\Ga,\label{Gl-b}
\end{align}
where $\phi:\X\setminus\{0\}\rightarrow\R_+$ is a~pair potential,
$\phi(-x)=\phi(x)$, $0<z,m\in L^\infty(\X)$, and $s\in[0;1]$.
Note that in the case $m(x)\equiv 1$, $z(x)\equiv z>0$ and for
any $s\in[0;1]$ the operator $L$ is well defined and, moreover,
symmetric in the space $L^2(\Ga,\mu)$, where $\mu$ is a~Gibbs
measure, given by the pair potential $\phi$ and activity parameter
$z$ (see e.g. \cite{KLR2007} and references therein). This gives
possibility to study the corresponding semigroup in $L^2(\Ga,\mu)$.
If, additionally, $s=0$, the corresponding dynamics was also studied in
another Banach spaces, see e.g. \cite{KKZ2006, FKKZ2010, FKK2010}.
Below we show that one of the main result of the paper stated in
Theorem~\ref{ThClasCol} can be applied to the case of arbitrary
$s\in[0;1]$ and non-constant $m$ and $z$. Set
\begin{equation}\label{integr_cond}
\beta_{\tau}:=\int_\X\bigl\vert e^{\tau\phi(x)}-1\bigr\vert
dx\in[0;\infty], \quad \tau\in[-1;1].
\end{equation}
Let $s$ be arbitrary and fixed. Suppose that $\beta_s<\infty$,
$\beta_{s-1}<\infty$. Then,
 by \eqref{Gl-d}, \eqref{Gl-b}, \eqref{K-1-of-exp}, and \eqref{intexp},
\begin{align*}
I_d(\xi)&=\sum_{x\in\xi}d(x,\xi\setminus x)e^{C\beta_{s}}=D(\xi)e^{C\beta_{s}},\\
\intertext{and, analogously, taking into account that $\phi\geq0$,}
I_b(\xi)&=\sum_{x\in\xi}b(x,\xi\setminus x)e^{C\beta_{s-1}}\leq
\sum_{x\in\xi}\frac{z(x)}{m(x)}d(x,\xi\setminus x)e^{C\beta_{s-1}}
\end{align*}
Therefore, to apply Theorem~\ref{th1} we should
assume that there exists $\sigma>0$ such that
\begin{equation}\label{nonhomo}
    z(x)\leq \sigma m(x), \quad x\in\X,
\end{equation}
and
\begin{equation}\label{sn0smallz}
e^{C\beta_{s}} + \frac{\sigma}{C}e^{C\beta_{s-1}}<\frac{3}{2}.
\end{equation}
In particular, for $s=0$ we need
\begin{equation}\label{sn0smallz1}
\frac{\sigma}{C}e^{C\beta_{-1}}<\frac{1}{2}.
\end{equation}
Next, to have \eqref{D-bdd} and \eqref{nusmall}, we will distinguish two cases. For $s=0$ we obtain \eqref{D-bdd} since $m\in L^\infty(\X)$. In this case, $\nu=1$ that fulfilles \eqref{nusmall} as well. For $s\in(0,1]$, we should assume that
\begin{equation}\label{bdd-pot}
    0\leq \phi\in L^\infty(\X).
\end{equation}
Then, by \eqref{Gl-d}, $\nu= e^{s\bar{\phi}}\geq 1$, where $\bar{\phi}:=\Vert\phi\Vert_{L^\infty(\X)}$. Therefore, to have \eqref{nusmall}, we need the following improvement of \eqref{sn0smallz}:
\begin{equation}\label{sn0smallzimprove}
e^{C\beta_{s}} + \frac{\sigma}{C}e^{s\bar{\phi}+C\beta_{s-1}}<\frac{3}{2}.
\end{equation}
\end{example}

\begin{example}\label{ex-BDLP}
(Bolker--Dieckman--Law--Pacala (BDLP) model) This example describes
a generalization of the model of plant ecology (see \cite{FKK2009} and references
therein). Let $L$ be given by \eqref{BaDGen} with
\begin{align}
d(x,\ga\setminus x)&=m(x) + \varkappa^-(x)\sum_{y\in\ga\setminus x}a^-(x-y), &&x\in\ga,\ \ga\in\Ga,\label{BDLP-d}\\
b(x,\ga)&=\varkappa^+(x)\sum_{y\in\ga}a^+(x-y), &&x\in\X\setminus\ga,\
\ga\in\Ga,\label{BDLP-b}
\end{align}
where $0<m\in L^\infty(\X)$, $0\leq\varkappa^\pm\in L^\infty(\X)$, $0\leq a^\pm\in L^1(\X,dx)\cap
L^\infty(\X,dx)$, $\int_\X a^\pm(x)dx=1$. Then, by \eqref{K-1m}, \eqref{K-1-of-lin}, and \eqref{Ga0}--\eqref{speccases},
\[
I_d(\xi)= \sum_{x\in\xi} d(x,\xi\setminus x) + \sum_{x\in\xi} C\varkappa^-(x),\qquad
I_b(\xi)= \sum_{x\in\xi} b(x,\xi\setminus x) + \sum_{x\in\xi} C\varkappa^+(x).
\]
Let us suppose, cf. \cite{FKK2009}, that there exists $\delta>0$ such that
\begin{align}
(4+\delta)C\varkappa^-(x)&\leq m(x), \quad x\in\X,\label{smallparBDLP-1}\\
(4+\delta)\varkappa^+(x)&\leq m(x), \quad x\in\X,\label{smallparBDLP-3}\\
4\varkappa^+(x)a^+(x)&\leq C\varkappa^-(x)a^-(x).
\quad x\in\X,\label{smallparBDLP-2}
\end{align}
Then
\begin{align*}
  d(x,\xi)+C\varkappa^-(x)&\leq d(x,\xi)+\frac{m(x)}{4+\delta}\leq\Bigl(1+\frac{1}{4+\delta}\Bigr)d(x,\xi),\\
  b(x,\xi)+C\varkappa^+(x)&\leq\frac{C}{4}\varkappa^-(x)\sum_{y\in\xi}a^-(x-y)+\frac{Cm(x)}{4+\delta}<\frac{C}{4}d(x,\xi),
\end{align*}
Hence, \eqref{est-d}, \eqref{est-b} hold with
\[
a_1=1+\frac{1}{4+\delta}, \qquad a_2=\frac{C}{4},
\]
that fulfills \eqref{asmall}.
Next, under conditions \eqref{smallparBDLP-1}, \eqref{smallparBDLP-2}, we have
\[
d(x,\xi)\leq \Vert m\Vert_{L^\infty(\X)}+\Vert \varkappa^-\Vert_{L^\infty(\X)}
\Vert a^-\Vert_{L^\infty(\X)} |\xi|, \quad \xi\in\Ga_0,
\]
and hence \eqref{D-bdd} holds with $\nu=1$, which makes \eqref{nusmall} obvious.
\end{example}
\begin{remark}
  It was shown in \cite{FKK2009} that, for the case of constant $m,\varkappa^\pm$, the condition  like \eqref{smallparBDLP-1} is essential. Namely,
 if $m>0$ is arbitrary small the operator
 $\widehat{L}$ will not be even accretive in $\L_C$.
\end{remark}

\begin{example}
(Contact model with establishment). \label{CMest}
Let $L$ be given by \eqref{BaDGen} with $d(x,\ga)=m(x)$ for all $\ga\in\Ga$ and
\begin{equation}\label{CMestrates}
    b(x,\ga)=\varkappa(x)\exp\Bigl\{\sum_{y\in\ga}\phi(x-y)\Bigr\}
    \sum_{y\in\ga}a(x-y), \quad \ga\in\Ga, x\in\X\setminus \ga.
\end{equation}
Here $0<m\in L^\infty(\X)$, $0\leq\varkappa\in L^\infty(\X)$, $0\leq a\in L^1(\X)\cap L^\infty(\X)$, $\int_\X a(x) \,dx=1$.
\end{example}

\subsection{Stationary equation}\label{subsect-evol-se}
In this subsection we study the question about stationary solutions
to \eqref{QE}. For any $s\ge0$, we consider the following subset
of $\K_C$
\[
\K_{\alpha C}^{(s)}:=\bigl\{ k\in \K_{\alpha
C}\bigm| k(\emptyset)=s\bigr\}.
\]
We define $\widetilde{\K}$ to be the closure of $\K_{\alpha
C}^{(0)}$ in the norm of $\K_C$. It is clear that $\widetilde{\K}$
with the norm of $\K_C$ is a~Banach space.

\begin{proposition}
Let \eqref{est-d}, \eqref{est-b}, and \eqref{D-bdd} be satisfied
with
\begin{equation}\label{statior-est}
a_1+\frac{a_2}{C}<2.
\end{equation}
Assume, additionally, that
\begin{equation}\label{nonezerodeath}
  d(x,\emptyset)>0, \quad x\in\X.
\end{equation}
Then for any $\alpha\in(0;\frac{1}{\nu})$ the stationary equation
\begin{equation}\label{stationaryeqn}
\widehat{L}^\ast k=0
\end{equation}
has a~unique solution $k_\inv$ from $\K_{\alpha C}^{(1)}$ which is
given by the expression
\begin{equation}\label{invsol}
k_\inv=1^*+(\1-S)^{-1}E.
\end{equation}
Here $1^\ast$ denotes the function defined by
$1^\ast(\eta)=0^{|\eta|}$, $\eta\in\Ga_0$, the function
$E\in\K_{\alpha C}^{(0)}$ is such that
\[
E(\eta)=\1_{\Ga^{(1)}}(\eta)\sum_{x\in\eta}\frac{b(x,\emptyset)}{d(x,\emptyset)},
\quad \eta\in\Ga_0,
\]
and $S$ is a~generalized Kirkwood--Salzburg operator on
$\tilde{\K}$, given by
\begin{align}
\left( Sk\right) \left( \eta \right) =&-\frac{1}{D\left( \eta
\right) } \sum_{x\in \eta }\int_{\Ga _{0}\setminus \left\{ \emptyset
\right\} }k(\zeta \cup \eta )(K^{-1}d(x,\cdot \cup \eta
\setminus x))(\zeta
)d\la (\zeta ) \label{KSoper}\\
&+\frac{1}{D\left( \eta \right) }\sum_{x\in \eta }\int_{\Ga
_{0}}k(\zeta \cup (\eta \setminus x))(K^{-1}b(x,\cdot \cup \eta
\setminus x))(\zeta )d\la (\zeta ),\notag
\end{align}
for $\eta \neq \emptyset$ and $\left( Sk\right) \left( \emptyset
\right) =0$. In particular, if $b(x,\emptyset)=0$ for a.a. $x\in\X$
then this solution is such that
\begin{equation}\label{deltasol}
k_\inv^{(n)}=0, \qquad n\geq 1.
\end{equation}
\end{proposition}
\begin{remark}
It is worth noting that \eqref{est-d1}, \eqref{est-b1} imply
\eqref{nonezerodeath}.
\end{remark}
\begin{proof}
Suppose that \eqref{stationaryeqn} holds
for some $k\in\K_{\alpha C}^{(1)}$. Then
\begin{align}
D\left( \eta \right) k(\eta )=&-\sum_{x\in \eta }\int_{\Ga
_{0}\setminus \left\{ \emptyset \right\} }k(\zeta \cup \eta
)\bigl(K^{-1}d(x,\cdot \cup
\eta \setminus x)\bigr)(\zeta )d\la (\zeta ) \notag\\
&+\sum_{x\in \eta }\int_{\Ga _{0}}k(\zeta \cup (\eta \setminus
x))\bigl(K^{-1}b(x,\cdot \cup \eta \setminus x)\bigr)(\zeta
)d\la (\zeta ). \label{rel}
\end{align}
The equality \eqref{rel} is satisfied for any $k\in\K_{\alpha
C}^{(1)}$ at the point $\eta=\emptyset$. Using the fact that
$D(\emptyset)=0$ one may rewrite \eqref{rel} in terms of the
function $\tilde{k}=k-1^*\in\K_{\alpha C}^{(0)}$. Namely,
\begin{align}
D\left( \eta \right) \tilde{k}(\eta )=&-\sum_{x\in \eta }\int_{\Ga
_{0}\setminus \left\{ \emptyset \right\} }\tilde{k}(\zeta \cup \eta
)\bigl(K^{-1}d(x,\cdot \cup
\eta \setminus x)\bigr)(\zeta )d\la (\zeta ) \notag\\
&+\sum_{x\in \eta }\int_{\Ga _{0}}\tilde{k}(\zeta \cup (\eta
\setminus x))\bigl(K^{-1}b(x,\cdot \cup \eta \setminus
x)\bigr)(\zeta )d\la (\zeta ).\notag\\ &+\sum_{x\in \eta
}0^{|\eta\setminus x|}b(x,\eta\setminus x).\label{rel1}
\end{align}
As a~result,
\[
\tilde{k}(\eta)=(S\tilde{k})(\eta)+E(\eta),\quad
\eta\in\Ga_0.
\]
Next, for $\eta\neq \emptyset$
\begin{align*}
&C^{-\left\vert \eta \right\vert }\left\vert \left( Sk\right) \left( \eta
\right) \right\vert \\
\leq &\frac{C^{-\left\vert \eta \right\vert }}{D\left( \eta \right)
}\sum_{x\in \eta }\int_{\Ga _{0}\setminus \left\{ \emptyset \right\}
}\left\vert k(\zeta \cup \eta )\right\vert \left\vert
(K^{-1}d(x,\cdot \cup \eta \setminus x))(\zeta
)\right\vert d\la (\zeta ) \\
&+\frac{C^{-\left\vert \eta \right\vert }}{D\left( \eta \right)
}\sum_{x\in \eta }\int_{\Ga _{0}}k(\zeta \cup (\eta \setminus
x))\left\vert (K^{-1}b(x,\cdot \cup \eta \setminus
x))(\zeta )\right\vert d\la (\zeta ) \\
\leq &\frac{\left\Vert k\right\Vert _{\K_C}}{D\left( \eta \right)
}\sum_{x\in \eta }\int_{\Ga _{0}\setminus \left\{ \emptyset \right\}
}C^{\left\vert \zeta \right\vert }\left\vert (K^{-1}d(x,\cdot
\cup \eta \setminus
x))(\zeta )\right\vert d\la (\zeta ) \\
&+\frac{\left\Vert k\right\Vert _{\K_C}}{D\left( \eta \right)
}\frac{1}{C} \sum_{x\in \eta }\int_{\Ga _{0}}C^{\left\vert \zeta
\right\vert }\left\vert (K^{-1}b(x,\cdot \cup \eta \setminus
x))(\zeta )\right\vert
d\la (\zeta ) \\
\leq &\frac{\left\Vert k\right\Vert _{\K_C}}{D\left( \eta \right) }D\left(
\eta \right) \left( a_{1}-1+\frac{a_{2}}{C}\right)=\left( a_{1}-1+\frac{a_{2}}{C}\right)\|k\| _{\K_C}.
\end{align*}
Hence,
\[
\|S\|=a_{1}+\frac{a_{2}}{C}-1<1
\]
in $\widetilde{K}$. This finishes the proof.
\end{proof}

\begin{remark}
The name of the operator \eqref{KSoper} is motivated by
Example~\ref{ex-Gl}. Namely, if $s=0$ then the operator
\eqref{KSoper} has form
\begin{align*}
\left( Sk\right) \left( \eta \right) =\frac{1}{m| \eta| }\sum_{x\in
\eta }e_{\la}(e^{-\phi(x-\cdot)},\eta \setminus x)\int_{\Ga
_{0}}k(\zeta \cup (\eta \setminus
x))e_{\la}(e^{-\phi(x-\cdot)}-1,\zeta )d\la (\zeta ),\notag
\end{align*}
that is quite similar of the so-called Kirkwood--Salsburg operator
known in mathematical physics (see e.g. \cite{Rue1969, KK2003a}).
For $s=0$ condition \eqref{statior-est} has form
$\frac{z}{C}e^{C\beta_{-1}}<1$. Under this
condition, the stationary solution to \eqref{stationaryeqn} is a
unique and coincides with the correlation function of the Gibbs
measure, corresponding to potential $\phi$ and activity $z$.
\end{remark}

\begin{remark} It is worth pointing out that $b(x,\emptyset)=0$
in the case of Example~\ref{ex-BDLP}. Therefore, if we suppose (cf.
\eqref{smallparBDLP-1}, \eqref{smallparBDLP-2}) that
$2\varkappa^-C<m$ and $2\varkappa^+a^+(x)\leq C\varkappa^-a^-(x)$,
for $x\in\X$, condition \eqref{statior-est} will be satisfied.
However, the unique solution to \eqref{stationaryeqn} will be given
by \eqref{deltasol}. In the next example we improve this statement.
\end{remark}

\begin{example}
Let us consider the following natural modification of BDLP-model
coming from Example~\ref{ex-BDLP}: let $d$ be given by
\eqref{BDLP-d} and
\begin{equation}
b(x,\ga)=\kappa+\varkappa^+\sum_{y\in\ga}a^+(x-y), \quad x\in\X\setminus\ga,\
\ga\in\Ga,\label{BDLP-b-mod}
\end{equation}
where $\varkappa^+,a^+$ are as before and $\kappa>0$. Then, under
assumptions
\begin{equation}\label{aaa1}
 2\max\Bigl\{\varkappa^-C;\frac{2\kappa }{C}\Bigr\}<m
\end{equation}
and
\begin{equation}\label{aaa2}
2\varkappa^+a^+(x)\leq C\varkappa^-a^-(x), \quad x\in\X,
\end{equation}
we obtain for some $\delta>0$
\begin{align*}
\int_{\Ga _{0}}\left\vert K^{-1}d\left( x,\cdot \cup \xi \right)
\right\vert \left( \eta \right) & C^{\left\vert \eta \right\vert
}d\la \left( \eta \right)=d(x,\xi)+C\varkappa^-\leq \Bigl(1+\frac{1}{2+\delta}\Bigr)d(x,\xi)\\
\int_{\Ga _{0}}\left\vert K^{-1}b\left( x,\cdot \cup \xi \right)
\right\vert \left( \eta \right) & C^{\left\vert \eta \right\vert }d\la
\left( \eta \right)=b(x,\xi)+C\varkappa^+
\\&\leq \kappa
+\frac{1}{2}C\varkappa^-\sum_{y\in\xi}a^-(x-y)+\frac{m}{4}C<\frac{C}{2}d(x,\xi).
\end{align*}
The latter inequalities imply \eqref{statior-est}. In this case,
$E(\eta)=\1_{\Ga^{(1)}}(\eta)\frac{\kappa}{m}$.
\end{example}

\begin{remark}
If $a^+(x)=a^-(x)$, $x\in\X$ and $\varkappa^+=z\varkappa^-$,
$\kappa=zm$ for some $z>0$ then $b(x,\ga)=zd(x,\ga)$ and the Poisson
measure $\pi_z$ with the intensity $z$ will be symmetrizing measure
for the operator $L$. In particular, it will be invariant measure.
This fact means that its correlation function $k_z(\eta)=z^{|\eta|}$
is a~solution to \eqref{stationaryeqn}. Conditions \eqref{aaa1} and
\eqref{aaa2} in this case are equivalent to $4z< C$ and
$2\varkappa^-C<m$. As a~result, due to uniqueness of such solution,
\[
1^*(\eta)+z(\1-S)^{-1}\1_{\Ga^{(1)}}(\eta)=z^{|\eta|}, \quad
\eta\in\Ga_0.
\]
\end{remark}

\section{Approximative approach for the Glauber dynamics}
In this section we consider an approximative approach for the construction of the Glauber-type dynamics described in Example~\ref{ex-Gl} for
\[
s=0, \quad m(x)\equiv 1, \quad z(x)\equiv z>0.
\]
Therefore, in such a case, \eqref{BaDGen} has the form
\begin{align}
(LF)(\gamma):=&\sum_{x\in\gamma} \bigl[F(\gamma\setminus x) -F(\gamma)\bigr]
\label{genGa} \\
&+ z \int_{{{\mathbb{R}}^d}} \bigl[F(\gamma\cup x) -F(\gamma)\bigr]\exp%
\bigl\{-E^\phi(x,\gamma)\bigr\} dx, \qquad \gamma\in\Gamma,  \notag
\end{align}
with $E^\phi$ given by \eqref{localenergy}.

Let $G\in {B_{\mathrm{bs}}}(\Gamma_0)$ then $F=KG\in{{\mathcal
F}_{\mathrm{cyl}}}(\Gamma )$. By \eqref{newexpr}, \eqref{K-1m}, \eqref{K-1-of-exp}, one has the following explicit form for the mapping $\widehat{L}:=K^{-1}LK$ on
${B_{\mathrm{bs}}}(\Gamma_0)$
\begin{multline}\label{Lhat1}
(\widehat{L}G)(\eta) =- |\eta| G(\eta) \\+ z
\sum_{\xi\subset\eta}\int_{{\mathbb R}^d} e^{-E^\phi(x,\xi)}
G(\xi\cup x)e_\lambda(e^{-\phi (x - \cdot)}-1,\eta\setminus\xi) dx ,
\end{multline}
where $e_\la$ is given by \eqref{Leb-Pois-exp}.

Let us denote, for any $\eta\in\Gamma_0$,
\begin{align}
(L_0 G)(\eta) &:= - |\eta| G(\eta);\label{L0}\\
(L_1 G)(\eta) &:= z \sum_{\xi\subset\eta}\int_{{\mathbb R}^d}
e^{-E^\phi(x,\xi)} G(\xi\cup x)e_\lambda(e^{-\phi (x -
\cdot)}-1,\eta\setminus\xi) dx.\label{L1}
\end{align}

To simplify notation we continue to write $C_{\phi}$ for $\beta_{-1}$. In contrast to \eqref{dom}, we will not work the maximal domain of the operator $L_0$. Namely, the following statement will be used
\begin{proposition}\label{prop-oper}
The expression \eqref{Lhat1} defines a linear operator $\widehat{L}$ in
${\mathcal L}_C$ with the dense domain ${\mathcal L}_{2C}\subset
{\mathcal L}_C$.
\end{proposition}
\begin{proof}
For any $G\in {\mathcal L}_{2C} $
\[
\left\Vert L_{0}G\right\Vert_{C} =\int_{\Gamma_0} |G(\eta)||\eta|
C^{|\eta|} d\lambda ( \eta)<\int_{\Gamma_0} |G(\eta)|2^{|\eta|}
C^{|\eta|} d\lambda (\eta)<\infty
\]
and, by Lemma~\ref{Minlos},
\begin{align*}
&\quad \left\Vert L_{1}G\right\Vert_{C} \\&\leq z\int_{\Gamma
_{0}}\sum_{\xi \subset \eta }\int_{{{\mathbb
R}^d}}\,e^{-E^\phi(x,\xi )}\left\vert G(\xi \cup x)\right\vert
e_{\lambda }\left(\left\vert e^{-\phi (x-\cdot )}-1\right\vert ,\eta
\setminus \xi\right)dxC^{\left\vert \eta \right\vert }d \lambda \left( \eta
\right)  \\
&=z\int_{\Gamma_{0}}\int_{\Gamma_{0}}\int_{{{\mathbb
R}^d}}\,e^{-E^\phi(x,\xi )}\left\vert G(\xi \cup x)\right\vert
e_{\lambda }\left(\left\vert e^{-\phi (x-\cdot )}-1\right\vert ,\eta
\right)dxC^{\left\vert \eta \right\vert }C^{\left\vert \xi
\right\vert }d \lambda \left( \xi \right) d\lambda \left(
\eta \right)  \\
&\leq \frac{z}{C}\exp \left\{ CC_{\phi }\right\} \int_{\Gamma
_{0}}\left\vert G\left( \xi \right) \right\vert \left\vert \xi
\right\vert C^{\left\vert \xi \right\vert }d \lambda \left(  \xi
\right) < \frac{z}{C}\exp \left\{ CC_{\phi }\right\} \int_{\Gamma
_{0}}\left\vert G\left( \xi \right) \right\vert 2^{\left\vert \xi
\right\vert} C^{\left\vert \xi \right\vert }d \lambda \left(  \xi
\right)  \\&<\infty .
\end{align*}
Embedding ${\mathcal L}_{2C}\subset{\mathcal L}_C$ is dense since
${B_{\mathrm{bs}}}(\Gamma_0)\subset{\mathcal L}_{2C}$.
\end{proof}

\subsection{Description of approximation}

In this section we will use the symbol $K_0$ to denote the restriction of $K$ onto functions on $\Gamma_0$.

Let $\delta\in(0;1)$ be arbitrary and fixed. Consider for any
$\Lambda \in{\mathcal B}_{\mathrm{b}}({{\mathbb R}^d})$ the
following linear mapping on functions $F\in K_0 \bigl(B_{\mathrm{bs}}(\Gamma_0)\bigr)\subset \Fcyl$
\begin{align}\label{apprGa0}
\left( P_{\delta }^\Lambda  F\right) \left( \gamma \right)
=&\sum_{\eta \subset \gamma }\delta ^{\left\vert \eta \right\vert
}\left( 1-\delta \right) ^{\left\vert \gamma \setminus \eta
\right\vert }\bigl(\Xi^\Lambda_\delta\left( \gamma
\right)\bigr)^{-1}
\\ &\times \int_{\Gamma_{\Lambda }}\left( z\delta \right)
^{\left\vert \omega \right\vert }\prod\limits_{y\in \omega
}e^{-E^\phi\left( y,\gamma \right) }F\left( \left( \gamma \setminus
\eta \right) \cup \omega \right) d \lambda \left(  \omega
\right),\quad \gamma \in\Gamma_0,\notag
\end{align}
where
\begin{equation}\label{normfactor}
\Xi_\delta^\Lambda \left( \gamma \right) = \int_{\Gamma_{\Lambda
}}\left( z\delta \right) ^{\left\vert \omega \right\vert
}\prod\limits_{y\in \omega }e^{-E^\phi\left( y,\gamma \right)
}d \lambda \left(  \omega \right).
\end{equation}
Clearly, $P_{\delta }^\Lambda $ is a positive preserving mapping and
\[
\left( P_{\delta }^\Lambda  1\right) \left( \gamma
\right)=\sum_{\eta \subset \gamma }\delta ^{\left\vert \eta
\right\vert }\left( 1-\delta \right) ^{\left\vert \gamma \setminus
\eta \right\vert }=1, \quad \gamma \in\Gamma_0.
\]

Operator \eqref{apprGa0} is constructed as a transition operator of
a Markov chain, which is a time discretization of a continuous time
process with the generator \eqref{genGa} and discretization
parameter $\delta\in(0;1)$. Roughly speaking, according to the
representation \eqref{apprGa0}, the probability of transition
$\gamma \rightarrow (\gamma \setminus\eta)\cup\omega$ (which
describes removing of subconfiguration $\eta\subset\gamma $ and
birth of a new subconfiguration $\omega\in\Gamma(\Lambda) $) after
small time $\delta$ is equal to
\[
\bigl(\Xi_\delta^\Lambda (\gamma )\bigr)^{-1} \delta^{|\eta|}
(1-\delta)^{|\gamma \setminus\eta|}(z\delta)^{|\omega|}
\prod_{y\in\omega}e^{-E^\phi(y,\gamma )}.
\]

We may rewrite \eqref{apprGa0} in another manner.
\begin{proposition}
For any $F\in{{\mathcal F}_{\mathrm{cyl}}}(\Gamma_0)$ the following
equality holds
\begin{align}\label{anotherform}
\left( P_{\delta }^\Lambda  F\right) \left( \gamma \right) =&
\sum_{\xi \subset \gamma }\left( 1-\delta \right) ^{\left\vert \xi
\right\vert } \int_{\Gamma_{\Lambda }}\left( z\delta \right)
^{\left\vert \omega \right\vert }\prod\limits_{y\in \omega
}e^{-E^\phi\left( y,\gamma \right) }\\&\times (K_0^{-1}
F)\left( \xi \cup \omega \right) d \lambda \left(  \omega
\right).\notag
\end{align}
\end{proposition}
\begin{proof} Let $G:=K_0^{-1}F\in {B_{\mathrm{bs}}}(\Gamma_0)$.
Since $\Xi_\delta^\Lambda $ doesn't depend on $\eta$, for $\gamma
\in\Gamma_0$ we have
\begin{align}\label{ex1}
\left( P_{\delta }^\Lambda  F\right) \left( \gamma \right)
=&\bigl(\Xi^\Lambda_\delta\left( \gamma \right)\bigr)^{-1}
\int_{\Gamma_{\Lambda }}\left( z\delta \right) ^{\left\vert \omega
\right\vert
}\prod\limits_{y\in \omega }e^{-E^\phi\left( y,\gamma \right) }\\
&\times \sum_{\eta \subset \gamma }\delta ^{\left\vert \gamma
\setminus \eta \right\vert } \left( 1-\delta \right) ^{\left\vert
\eta \right\vert } F\left( \eta \cup \omega \right) d \lambda \left(
\omega \right).\notag
\end{align}
To rewrite \eqref{apprGa0}, we have used also that any
$\eta\subset\gamma$ corresponds to a unique
$\gamma\setminus\eta\subset\gamma$. Applying the definition of $K_0$
to $F=K_0G$ we obtain
\begin{align}\label{ex2}
\sum_{\eta \subset \gamma }\delta ^{\left\vert \gamma \setminus
\eta \right\vert } \left( 1-\delta \right) ^{\left\vert \eta
\right\vert } F\left( \eta \cup \omega \right)&=\sum_{\eta \subset
\gamma }\delta ^{\left\vert \gamma \setminus \eta \right\vert }
\left(
1-\delta \right) ^{\left\vert\eta \right\vert } \sum_{\zeta\subset\eta}\sum_{\beta\subset\omega}G\left( \zeta \cup \beta \right)\\
&=\sum_{\zeta\subset\gamma}\sum_{\beta\subset\omega}G\left( \zeta \cup \beta \right)\sum_{\eta '\subset
\gamma\setminus\zeta }\delta ^{\left\vert \gamma \setminus( \eta'\cup\zeta) \right\vert }
\left(
1-\delta \right) ^{\left\vert\eta' \cup\zeta\right\vert },\notag
\end{align}
where after changing summation over $\eta\subset\gamma$ and
$\zeta\subset\eta$ we have used the fact that for any configuration
$\eta\subset\gamma$ which contains fixed $\zeta\subset\gamma$ there
exists a unique $\eta'\subset\gamma\setminus\zeta$ such that
$\eta=\eta'\cup\zeta$. But by the binomial formula
\begin{align}\label{ex3}
\sum_{\eta '\subset \gamma\setminus\zeta }\delta ^{\left\vert \gamma
\setminus( \eta'\cup\zeta) \right\vert } \left( 1-\delta \right)
^{\left\vert\eta' \cup\zeta\right\vert
}&=(1-\delta)^{|\zeta|}\sum_{\eta '\subset \gamma\setminus\zeta
}\delta ^{\left\vert (\gamma\setminus \zeta) \setminus \eta'
\right\vert } \left( 1-\delta \right) ^{\left\vert\eta'\right\vert }
\\&=(1-\delta)^{|\zeta|} (\delta+1-\delta)^{|\gamma
\setminus\zeta|}=(1-\delta)^{|\zeta|}.\notag
\end{align}
Combining \eqref{ex1}, \eqref{ex2}, \eqref{ex3}, we get
\begin{align*}
\left( P_{\delta }^\Lambda  F\right) \left( \gamma \right)\notag
=&\bigl(\Xi^\Lambda_\delta\left( \gamma \right)\bigr)^{-1}
\int_{\Gamma_{\Lambda }}\left( z\delta \right) ^{\left\vert \omega
\right\vert
}\prod\limits_{y\in \omega }e^{-E^\phi\left( y,\gamma \right) }\\
&\times
\sum_{\zeta\subset\gamma}\sum_{\beta\subset\omega}G\left( \zeta \cup
\beta \right) (1-\delta)^{|\zeta|} d \lambda \left(  \omega \right).
\end{align*}
Next, Lemma~\ref{Minlos} yields
\begin{align*}
\left( P_{\delta }^\Lambda  F\right) \left( \gamma
\right)=&\bigl(\Xi^\Lambda_\delta\left( \gamma \right)\bigr)^{-1}
\int_{\Gamma_{\Lambda }}\int_{\Gamma_{\Lambda }}\left( z\delta
\right) ^{\left\vert \omega \cup\beta\right\vert
}\prod\limits_{y\in \omega\cup\beta }e^{-E^\phi\left( y,\gamma \right) }\notag\\
&\times \sum_{\zeta\subset\gamma}G\left( \zeta \cup \beta
\right) (1-\delta)^{|\zeta|}d \lambda \left(  \omega \right) d \lambda \left(  \beta \right)\notag\\=&\int_{\Gamma_{\Lambda }}\left(
z\delta \right) ^{\left\vert \beta\right\vert }\prod\limits_{y\in
\beta }e^{-E^\phi\left( y,\gamma \right) }
\sum_{\zeta\subset\gamma}G\left( \zeta \cup \beta \right)
(1-\delta)^{|\zeta|} d \lambda \left(  \beta \right),\notag
\end{align*}
which proves the statement.
\end{proof}

In the next proposition we describe the image of $P_\delta^\Lambda $
under the $K_0$-transform.
\begin{proposition}
Let $\widehat{P}_\delta^\Lambda =K_0^{-1}P_\delta^\Lambda  K_0$. Then
for any $G\in {B_{\mathrm{bs}}}(\Gamma_0)$ the following equality
holds
\begin{align}\label{apprsemigroupLa}
\bigl(\widehat{P}_{\delta }^\Lambda  G\bigr) \left( \eta \right)
=&\sum_{\xi \subset \eta }\left( 1-\delta \right) ^{\left\vert \xi
\right\vert }\int_{\Gamma_{\Lambda }}\left( z\delta \right)
^{\left\vert \omega \right\vert }G\left( \xi \cup \omega \right)
\\&\times \prod\limits_{y\in \xi }e^{-E^\phi\left( y,\omega \right)
}\prod\limits_{y'\in \eta \setminus \xi }\left( e^{-E^\phi\left(
y',\omega \right) }-1\right) d \lambda \left(  \omega \right) , \quad
\eta\in\Gamma_0.\notag
\end{align}
\end{proposition}
\begin{proof}
By \eqref{anotherform} and the definition of $K_0^{-1}$, we have
\begin{align*}
&\quad\bigl(\widehat{P}_{\delta }^\Lambda  G\bigr) \left( \eta \right)\\
&=\sum_{\zeta\subset\eta}(-1)^{|\eta\setminus\zeta|}\sum_{\xi
\subset\zeta }\left( 1-\delta \right) ^{\left\vert \xi \right\vert }
\int_{\Gamma_{\Lambda }}\left( z\delta \right) ^{\left\vert \omega
\right\vert }\prod\limits_{y\in \omega }e^{-E^\phi\left( y,\zeta
\right) }G\left( \xi \cup \omega \right) d \lambda \left(  \omega
\right)\\&=\sum_{\xi\subset\eta} \left( 1-\delta \right)
^{\left\vert \xi \right\vert }\sum_{\zeta\subset\eta\setminus\xi }
(-1)^{|(\eta\setminus\xi)\setminus\zeta|}\int_{\Gamma_{\Lambda
}}\left( z\delta \right) ^{\left\vert \omega \right\vert
}\prod\limits_{y\in \omega }e^{-E^\phi\left( y,\zeta\cup\xi \right)
}G\left( \xi \cup \omega \right) d \lambda \left(  \omega \right).
\end{align*}
By the definition of the relative energy
\[
\prod\limits_{y\in \omega }e^{-E^\phi\left( y,\zeta\cup\xi
\right)}=\prod\limits_{y\in \xi }e^{-E^\phi\left( y,\omega\right)}
\prod\limits_{y'\in \zeta }e^{-E^\phi\left( y',\omega \right)}.
\]
The well-known equality (see, e.g., \cite{FKO2009})
\begin{align*}
\sum_{\zeta\subset\eta\setminus\xi } (-1)^{|(\eta
\setminus\xi)\setminus\zeta|}\prod\limits_{y'\in \zeta
}e^{-E^\phi\left( y',\omega
\right)}&=\Bigl(K_0^{-1}\prod\limits_{y'\in\cdot }e^{-E^\phi\left(
y',\omega \right)}\Bigl)(\eta\setminus\xi)\\&=\prod\limits_{y'\in
\eta \setminus \xi }\left( e^{-E^\phi\left( y',\omega \right)
}-1\right)
\end{align*}
completes the proof.
\end{proof}

\subsection{Construction of the semigroup on ${\mathcal L}_C$}

By analogy with \eqref{apprsemigroupLa}, we consider the following
linear mapping on measurable functions on $\Gamma_0$
\begin{align}\label{apprsemigroup}
\bigl(\widehat{P}_{\delta } G\bigr) \left( \eta \right) := & \sum_{\xi
\subset \eta }\left( 1-\delta \right) ^{\left\vert \xi \right\vert
}\int_{\Gamma_{0}}\left( z\delta \right) ^{\left\vert \omega
\right\vert }G\left( \xi \cup \omega \right)
\\&\times \prod\limits_{y\in \xi }e^{-E^\phi\left( y,\omega \right)
}\prod\limits_{y'\in \eta \setminus \xi }\left( e^{-E^\phi\left(
y',\omega \right) }-1\right) d \lambda \left(  \omega \right) , \quad
\eta\in\Gamma_0.\notag
\end{align}

\begin{proposition}\label{prop_contr}
Let
\begin{equation}
ze^{CC_{\phi }} \leq C. \label{smallparam}
\end{equation}
Then $\widehat{P}_\delta$, given by \eqref{apprsemigroup}, is a well
defined linear operator in ${\mathcal L}_C$, such that
\begin{equation}
\bigl\| \widehat{P}_\delta \bigr\|
\leq 1. \label{contacrion}
\end{equation}
\end{proposition}
\begin{proof}
Since $\phi \geq 0$ we have
\begin{align*}
\left\Vert \widehat{P}_{\delta }G\right\Vert_{C}  \leq &\int_{\Gamma
_{0}}\sum_{\xi \subset \eta }\left( 1-\delta \right) ^{\left\vert
\xi \right\vert }\int_{\Gamma_{0}}\left( z\delta \right)
^{\left\vert \omega \right\vert }\left\vert G\left( \xi \cup \omega
\right) \right\vert \\ & \times \prod\limits_{y\in \xi
}e^{-E^\phi\left( y,\omega \right) }\prod\limits_{y'\in \eta
\setminus \xi }\left\vert e^{-E^\phi\left( y',\omega \right)
}-1\right\vert d \lambda \left(  \omega \right) C ^{\left\vert
\eta \right\vert } d \lambda \left(  \eta \right) \\
= &\int_{\Gamma_{0}}\int_{\Gamma_{0}}\left( 1-\delta \right)
^{\left\vert \xi \right\vert }\int_{\Gamma_{0}}\left( z\delta
\right) ^{\left\vert \omega \right\vert }\left\vert G\left( \xi \cup
\omega \right) \right\vert \\ &\times \prod\limits_{y\in \xi
}e^{-E^\phi\left( y,\omega \right) }\prod\limits_{y'\in \eta
}\left\vert e^{-E^\phi\left( y',\omega \right) }-1\right\vert
d \lambda \left(  \omega \right) C ^{\left\vert \eta \right\vert }C
^{\left\vert \xi \right\vert } d \lambda \left(  \xi
\right) d \lambda \left(  \eta \right) \\
=&\int_{\Gamma_{0}}\int_{\Gamma_{0}}\left( 1-\delta \right)
^{\left\vert \xi \right\vert }\left( z\delta \right) ^{\left\vert
\omega \right\vert }\left\vert G\left( \xi \cup \omega \right)
\right\vert \\ &\times \prod\limits_{y\in \xi }e^{-E^\phi\left(
y,\omega \right) }\exp \left\{ C\int_{{{\mathbb R}^d}}\left(
1-e^{-E^\phi\left( y',\omega \right) }\right) dy'\right\} d \lambda \left(  \omega \right) C ^{\left\vert \xi \right\vert } d \lambda \left(  \xi \right).
\end{align*}

It is easy to see by the induction principle that for $\phi \geq 0$,
$\omega\in\Gamma_0$, $y\notin\omega$
\begin{equation}\label{keypos}
1-e^{-E^\phi\left( y,\omega \right) }=1-\prod\limits_{x\in \omega
}e^{-\phi \left( x-y\right) }\leq \sum_{x\in \omega }\left(
1-e^{-\phi \left( x-y\right) }\right).
\end{equation}
Then
\begin{align*}
\bigl\Vert \widehat{P}_{\delta }G\bigr\Vert_{C} \leq &\int_{\Gamma
_{0}}\int_{\Gamma_{0}}\left( 1-\delta \right) ^{\left\vert \xi
\right\vert }\left( z\delta \right) ^{\left\vert \omega \right\vert
}\left\vert G\left( \xi \cup \omega \right) \right\vert \\
&\times \exp \left\{ C\sum_{x\in \omega }\int_{{{\mathbb
R}^d}}\left( 1-e^{-\phi \left( x-y\right) }\right) dy\right\}
d \lambda \left(  \omega \right) C ^{\left\vert \xi
\right\vert } d \lambda \left(  \xi \right) \\
=&\int_{\Gamma_{0}}\int_{\Gamma_{0}}\left( 1-\delta \right)
^{\left\vert \xi \right\vert }\left( z\delta \right) ^{\left\vert
\omega \right\vert }\left\vert G\left( \xi \cup \omega \right)
\right\vert e^{CC_{\phi }\left\vert \omega \right\vert }C
^{\left\vert \xi \right\vert
} d \lambda \left(  \omega \right) d \lambda \left(  \xi \right) \\
=&\int_{\Gamma_{0}}\left[ \left( 1-\delta \right) C+z\delta
e^{CC_{\phi }} \right] ^{\left\vert \omega \right\vert }\left\vert
G\left( \omega \right) \right\vert d \lambda \left(  \omega \right)
\leq \left\Vert G\right\Vert_{C}.
\end{align*}
For the last inequality we have used that \eqref{smallparam} implies
$\left( 1-\delta \right) C+z\delta e^{CC_{\phi }}\leq C$. Note that,
for $\lambda $-a.a. $\eta\in\Gamma_0$
\begin{equation}\label{correctness}
\bigl(\widehat{P}_\delta G \bigr) (\eta) <\infty,
\end{equation}
and the statement is proved.
\end{proof}

\begin{proposition}\label{prop_ineq}
Let the inequality \eqref{smallparam} be fulfilled and define
\[
\widehat{L}_\delta:=\frac{1}{\delta}(\widehat{P}_\delta-1\!\!1), \quad
\delta\in(0;1),
\]
where $1\!\!1$ is the identity operator in ${\mathcal L}_C$. Then
for any $G\in {\mathcal L}_{2C}$
\begin{equation}\label{ineq}
\bigl\| (\widehat{L}_\delta -\widehat{L})G\bigr\|_C\leq 3\delta \|G\|_{2C}.
\end{equation}
\end{proposition}
\begin{proof}

Let us denote
\begin{align}
\bigl( \widehat{P}_{\delta }^{\left( 0\right) }G\bigr) \left( \eta
\right) =&\sum_{\xi \subset \eta }\left( 1-\delta \right)
^{\left\vert \xi \right\vert }G\left( \xi \right) 0^{\left\vert \eta
\setminus \xi \right\vert }=\left( 1-\delta \right) ^{\left\vert
\eta \right\vert }G\left(
\eta \right) ; \label{P0}\\
\bigl( \widehat{P}_{\delta }^{\left( 1\right) }G\bigr) \left( \eta
\right) =& \, z\delta \sum_{\xi \subset \eta }\left( 1-\delta \right)
^{\left\vert \xi \right\vert }\int_{{{\mathbb R}^d}}G\left( \xi \cup
x\right)\\&\times \prod\limits_{y\in \xi }e^{-\phi \left(
y-x\right) }\prod\limits_{y\in \eta \setminus \xi }\left( e^{-\phi
\left( y-x\right) }-1\right) dx; \label{P1}
\end{align}
and
\begin{equation}
\widehat{P}_{\delta }^{\left( \geq 2\right) }=\widehat{P}_{\delta }-\left(
\widehat{P}_{\delta }^{\left( 0\right) }+\widehat{P}_{\delta }^{\left(
1\right) }\right) \label{P2}.
\end{equation}

Clearly
\begin{align}\label{trin}
&\bigl\| (\widehat{L}_\delta -\widehat{L})G\bigr\|_C=\left\Vert
\frac{1}{\delta }\left( \widehat{P}_{\delta }G-G\right) -\widehat{L} G\right\Vert_{C} \\
&\quad \leq \left\Vert \frac{1}{\delta }\left( \widehat{P}_{\delta }^{\left(
0\right) }G-G\right) -L_{0}G\right\Vert_{C}+\left\Vert
\frac{1}{\delta }\widehat{P}_{\delta }^{\left( 1\right)
}G-L_{1}G\right\Vert_{C}+\frac{1}{\delta } \left\Vert
\widehat{P}_{\delta }^{\left( \geq 2\right) }G\right\Vert_{C}.\notag
\end{align}
Now we estimate each of the terms in \eqref{trin} separately. By
\eqref{L0} and \eqref{P0}, we have
\begin{equation*}
\left\Vert \frac{1}{\delta }\left( \widehat{P}_{\delta }^{\left(
0\right) }G-G\right) -L_{0}G\right\Vert_{C}  =\int_{\Gamma
_{0}}\left\vert \frac{\left( 1-\delta \right) ^{\left\vert \eta
\right\vert }-1}{\delta }+\left\vert \eta \right\vert \right\vert
\left\vert G\left( \eta \right) \right\vert C^{\left\vert \eta
\right\vert }d \lambda \left(  \eta \right).
\end{equation*}
But, for any $|\eta|\geq2$
\begin{align*}
\left\vert \frac{\left( 1-\delta \right) ^{\left\vert \eta
\right\vert }-1}{\delta }
+\left\vert \eta \right\vert \right\vert &=\left\vert\sum_{k=2}^{|\eta|} \binom{|\eta|}{k}(-1)^{k}\delta^{k-1}\right\vert\\
&=\delta\left\vert\sum_{k=2}^{|\eta|} \binom{|\eta|}{k}(-1)^{k}\delta^{k-2}\right\vert\leq \delta \sum_{k=2}^{|\eta|} \binom{|\eta|}{k} <\delta\cdot 2^{|\eta|}.
\end{align*}
Therefore,
\begin{equation}\label{est0}
\left\Vert \frac{1}{\delta }\left( \widehat{P}_{\delta }^{\left(
0\right) }G-G\right) -L_{0}G\right\Vert_{C}\leq \delta \|G\|_{2C}.
\end{equation}

Next, by \eqref{L1} and \eqref{P1}, one can write
\begin{align*}
\left\Vert \frac{1}{\delta }\widehat{P}_{\delta }^{\left( 1\right)
}G-L_{1}G\right\Vert_{C}  = & \, z\int_{\Gamma_{0}}\biggl\vert \sum_{\xi \subset \eta }\left(
\left( 1-\delta \right) ^{\left\vert \xi \right\vert }-1\right)
\int_{{{\mathbb{R}} ^{d}}}G\left( \xi \cup x\right)
\prod\limits_{y\in \xi }e^{-\phi \left( y-x\right) }\\& \times
\prod\limits_{y\in \eta \setminus \xi }\left( e^{-\phi \left(
y-x\right) }-1\right) dx\biggr\vert C^{\left\vert \eta \right\vert
}d \lambda \left(  \eta \right)  \\
\leq & \, z\int_{\Gamma_{0}}\int_{\Gamma_{0}}\left( 1-\left( 1-\delta
\right) ^{\left\vert \xi \right\vert }\right)
\int_{{{\mathbb{R}}^{d}}}\left\vert G\left( \xi \cup x\right)
\right\vert \prod\limits_{y\in \xi }e^{-\phi \left( y-x\right)
}\\& \times \prod\limits_{y\in \eta }\left( 1-e^{-\phi \left(
y-x\right) }\right) dxC^{\left\vert \xi \right\vert }C^{\left\vert
\eta \right\vert } d \lambda \left(  \xi \right) d \lambda \left(
\eta \right) ,
\end{align*}
where we have used Lemma~\ref{Minlos}. Note that for any $\left\vert
\xi \right\vert \geq 1$
\begin{equation*}
1-\left( 1-\delta \right) ^{\left\vert \xi \right\vert }=\delta
\sum_{k=0}^{\left\vert \xi \right\vert -1}\left( 1-\delta \right)
^{k}\leq \delta \left\vert \xi \right\vert
\end{equation*}
Then, by \eqref{smallparam} and \eqref{integrability}, one may
estimate
\begin{align}
\left\Vert \frac{1}{\delta }\widehat{P}_{\delta }^{\left( 1\right)
}G-L_{1}G\right\Vert_{C} \label{spec1} &\leq z\delta \int_{\Gamma_{0}}\left\vert \xi \right\vert
\int_{{{\mathbb{R}}^{d}}}\left\vert G\left( \xi \cup x\right)
\right\vert dxC^{\left\vert \xi \right\vert }e^{CC_{\phi }}d \lambda \left(  \xi \right)   \\
&\leq z \delta \int_{\Gamma_{0}}\left\vert \xi \right\vert \left(
\left\vert \xi \right\vert -1\right) \left\vert G\left( \xi \right)
\right\vert C^{\left\vert \xi \right\vert -1}e^{CC_{\phi }}d \lambda \left(  \xi \right) . \notag
\end{align}
Since $n\left( n-1\right) \leq 2^{n}$, $n\geq 1$ and by
\eqref{smallparam}, the latter expression can be bounded by
\[
\delta \int_{\Gamma_{0}}\left\vert G\left( \xi \right) \right\vert
\left( 2C\right) ^{\left\vert \xi \right\vert }\lambda \left(  d\xi
\right).
\]

Finally, Lemma~\ref{Minlos}, \eqref{keypos} and bound
$e^{-E^\phi(y,\omega)}\leq 1$, imply (we set here $\Gamma_{0}^{\left( \geq
2\right) }:=\bigsqcup_{n\geq2}\Gamma ^{(n)}$)
\begin{align}
\left\Vert \frac{1}{\delta }\widehat{P}_{\delta }^{\left( \geq 2\right)
}G\right\Vert_{C} \leq & \, \frac{1}{\delta }\int_{\Gamma_{0}}\sum_{\xi
\subset \eta }\left( 1-\delta \right) ^{\left\vert \xi \right\vert
}\int_{\Gamma_{0}^{\left( \geq 2\right) }}\left( z\delta \right)
^{\left\vert \omega \right\vert }\left\vert G\left(
\xi \cup \omega \right) \right\vert \label{spec01}\\
& \times \prod\limits_{y\in \xi }e^{-E^\phi\left( y,\omega
\right) }\prod\limits_{y\in \eta \setminus \xi }\left(
1-e^{-E^\phi\left( y,\omega \right) }\right) d \lambda \left(  \omega
\right) C ^{\left\vert \eta \right\vert }d \lambda \left(  \eta
\right) \notag
\\
\leq & \, \delta \int_{\Gamma_{0}}\sum_{\xi \subset \eta }\left(
1-\delta \right) ^{\left\vert \xi \right\vert }\int_{\Gamma
_{0}^{\left( \geq 2\right) }}z^{\left\vert \omega \right\vert
}\left\vert G\left( \xi \cup \omega \right) \right\vert \notag
\\
& \times \prod\limits_{y\in \xi }e^{-E^\phi\left( y,\omega
\right) }\prod\limits_{y\in \eta \setminus \xi }\left(
1-e^{-E^\phi\left( y,\omega \right) }\right) d \lambda \left(  \omega
\right) C
^{\left\vert \eta \right\vert } d \lambda \left(  \eta \right)  \notag\\
\leq &  \, \delta \int_{\Gamma_{0}}\sum_{\xi \subset \eta }\left(
1-\delta \right) ^{\left\vert \xi \right\vert }\int_{\Gamma
_{0}}z^{\left\vert \omega \right\vert }\left\vert G\left( \xi \cup
\omega \right) \right\vert \notag\\
&\times \prod\limits_{y\in \xi }e^{-E^\phi\left( y,\omega
\right) }\prod\limits_{y\in \eta \setminus \xi }\left(
1-e^{-E^\phi\left( y,\omega \right) }\right) d \lambda \left(  \omega
\right) C ^{\left\vert \eta \right\vert }d \lambda \left(  \eta \right)  \notag\\
\leq & \, \delta \int_{\Gamma_{0}}\int_{\Gamma_{0}}\left( 1-\delta
\right) ^{\left\vert \xi \right\vert }z^{\left\vert \omega
\right\vert }\left\vert G\left( \xi \cup
\omega \right) \right\vert \notag\\
& \times \int_{\Gamma_{0}}\prod\limits_{y\in \eta }\left(
1-e^{-E^\phi\left( y,\omega \right) }\right)  C ^{\left\vert \eta
\right\vert }d \lambda \left(  \eta \right) d \lambda \left(  \omega
\right) C ^{\left\vert \xi \right\vert }d\lambda \left( \xi \right)
\notag\\
\leq &  \, \delta \int_{\Gamma_{0}}\int_{\Gamma_{0}}\left( 1-\delta
\right) ^{\left\vert \xi \right\vert }z^{\left\vert \omega
\right\vert }\left\vert G\left( \xi \cup \omega \right) \right\vert
e^{CC_\phi\vert\omega\vert} d\lambda \left( \omega \right) C
^{\left\vert \xi \right\vert }d\lambda \left( \xi \right) \notag\\
\leq & \, \delta \int_{\Gamma_{0}}\left[ \left( 1-\delta \right)
C+ze^{CC_{\phi }}\right] ^{\left\vert \omega \right\vert }\left\vert
G\left(
\omega \right) \right\vert d\lambda \left( \omega \right)  \notag\\
\leq & \,  \delta \int_{\Gamma_{0}}\left[ \left( 2-\delta \right)
C\right] ^{\left\vert \omega \right\vert }\left\vert G\left( \omega
\right) \right\vert d\lambda \left( \omega \right) \leq \delta
\int_{\Gamma_{0}}\left\vert G\left( \omega \right) \right\vert
\left( 2C\right) ^{\left\vert \omega \right\vert }d\lambda \left(
\omega \right) .\notag
\end{align}
Combining inequalities \eqref{est0}--\eqref{spec01} we obtain the
assertion of the proposition.
\end{proof}

We will need the following results in the sequel.

\begin{lemma}[{cf. \cite[Corollary 3.8]{EK1986}}] \label{EK_res}
Let $A$ be a linear operator on a~Banach space $L$ with $D\left(
A\right) $ dense in $L$, and let $|\!|\!|\cdot |\!|\!|$ be a norm on
$D\left( A\right) $ with respect to which $D\left( A\right) $ is a
Banach space. For $n\in \mathbb{N}$ let $T_{n}$ be a linear
$\left\Vert \cdot \right\Vert $-contraction on $L$ such that
$T_{n}:D\left( A\right) \rightarrow D\left( A\right) $, and define
$A_{n}=n\left( T_{n}-1\right) $. Suppose there exist $\omega \geq 0$
and a sequence $\left\{ \varepsilon_{n}\right\} \subset \left(
0;+\infty \right) $ tending to zero such that for $n\in \mathbb{N}$
\begin{equation}\label{approperEK}
\left\Vert \left( A_{n}-A\right) f\right\Vert \leq \varepsilon_{n}
|\!|\!| f |\!|\!|,~f\in D\left( A\right)
\end{equation}
and
\begin{equation}\label{psevdocontr}
\bigl|\!\bigl|\!\bigl| T_{n}\upharpoonright_{D(A)}
\bigr|\!\bigr|\!\bigr| \leq 1+\frac{\omega }{n}.
\end{equation}
Then $A$ is closable and the closure of $A$ generates a strongly
continuous contraction semigroup on $L$.
\end{lemma}

\begin{lemma}[cf. {\cite[Theorem 6.5]{EK1986}}] \label{EK_res-conv}
Let $L, L_n$, $n\in{\mathbb N}$ be Banach spaces, and $p_n:
L\rightarrow L_n$ be bounded linear transformation, such that
$\sup_n \|p_n\|<\infty $. For any $n\in{\mathbb N}$, let $T_n$ be a
linear contraction on $L_n$, let $\varepsilon_n>0$ be such that
$\lim_{n\rightarrow \infty} \varepsilon_n =0$, and put
$A_n=\varepsilon_n^{-1}(T_n - 1\!\!1)$. Let $T(t)$ be a strongly
continuous contraction semigroup on $L$ with generator $A$ and let
$D$ be a core for $A$. Then the following are equivalent:
\begin{enumerate}
\item For each $f\in L$, $T_n^{[t/\varepsilon_n]} p_n f\rightarrow p_n
T(t) f$ in $L_n$ for all $t\geq0$ uniformly on bounded intervals.
Here and below $[\,\cdot\,\,]$ mean the entire part of a real
number.

\item For each $f\in D$, there exists $f_n\in L_n$ for each
$n\in{\mathbb N}$ such that $f_n \rightarrow p_n f$ and $A_n f_n
\rightarrow p_n Af$ in $L_n$.
\end{enumerate}
\end{lemma}

And now we are able to show the existence of the semigroup on
${\mathcal L}_C$.

\begin{theorem}\label{semigroup0}
Let
\begin{equation}\label{verysmallparam}
    z\leq \min\bigl\{Ce^{-CC_{\phi }} ; 2Ce^{-2CC_{\phi }}\bigr\} .
\end{equation}
Then $\bigl(\widehat{L}, {\mathcal L}_{2C}\bigr)$ from
Proposition~\ref{prop-oper} is a closable linear operator in
${\mathcal L}_C$ and its closure $\bigl(\widehat{L}, D(\widehat{L})\bigr)$
generates a strongly continuous contraction semigroup $\widehat{T}_t$ on
${\mathcal L}_C$.
\end{theorem}
\begin{proof} We apply Lemma~\ref{EK_res} for $L={\mathcal L}_C$,
$\bigl(A,D(A)\bigr)=\bigl(\widehat{L}, {\mathcal L}_{2C}\bigr)$,
$|\!|\!| \cdot |\!|\!| :=\|\cdot\|_{2C}$; $T_{n}=\widehat{P}_{\delta}$
and $A_{n}=n\left( T_{n}-1\right) =\frac{1}{\delta } (
\widehat{P}_{\delta }-1\!\!1)=\widehat{L}_\delta$, where
$\delta=\frac{1}{n}$, $n\geq 2$.

Condition $ze^{CC_{\phi }}\leq C$, Proposition~\ref{prop_contr}, and
Proposition~\ref{prop_ineq} provide that $T_n$, $n\geq 2$ are linear
$\|\cdot\|_C$-contractions and \eqref{approperEK} holds with
$\varepsilon_n=\frac{3}{n}=3\delta$. On the other hand, in addition,
Proposition~\ref{prop_contr} applied to the constant $2C$ instead of
$C$ gives \eqref{psevdocontr} for $\omega=0$ under condition
$ze^{2CC_{\phi }}\leq 2C$.
\end{proof}

Moreover, since we proved the existence of the semigroup $\widehat{T}_t$
on ${\mathcal L}_C$ one can apply contractions $\widehat{P}_\delta$
defined above by \eqref{apprsemigroup} to approximate the semigroup
$\widehat{T}_t$.
\begin{corollary}\label{approx0}
Let \eqref{smallparam} holds. Then for any $G\in{\mathcal L}_C$
\[
\bigl(\widehat{P}_{\frac{1}{n}}\bigr)^{[nt]} G \rightarrow \widehat{T}_t G,
\quad n\rightarrow \infty
\]
for all $t\geq0$ uniformly on bounded intervals.
\end{corollary}
\begin{proof}
The statement is a direct consequence of Theorem~\ref{semigroup0},
convergence \eqref{ineq}, and Lemma~\ref{EK_res-conv} (if we set
$L_n=L={\mathcal L}_C$, $p_n=1\!\!1$, $n\in{\mathbb N}$).
\end{proof}

\subsection{Finite-volume approximation of $\widehat{T}_t$}

Note that $\widehat{P}_\delta$ defined by \eqref{apprsemigroup} is a
formal point-wise limit of $\widehat{P}_\delta^\Lambda $ as $\Lambda
\uparrow {{\mathbb R}^d}$. We have shown in \eqref{correctness} that
this definition is correct. Corollary~\ref{approx0} claims
additionally that the linear contractions $\widehat{P}_\delta$
approximate the semigroup $\widehat{T}_t$, when $\delta\downarrow 0$.
One may also show that mappings $\widehat{P}_\delta^\Lambda $ have
a~similar property when $\Lambda \uparrow{{\mathbb R}^d}$,
$\delta\downarrow 0$.

Let us fix a system $\{\Lambda_n\}_{n\geq 2}$, where $\Lambda
_n\in{\mathcal B}_{\mathrm{b}}({{\mathbb R}^d})$, $\Lambda
_n\subset\Lambda_{n+1}$, $\bigcup_{n} \Lambda_n={{\mathbb R}^d}$. We
set
\[
T_n:=\widehat{P}_{\frac{1}{n}}^{\Lambda_n}.
\]
Note that any $T_n$ is a linear mapping on ${B_{\mathrm{bs}}}(\Gamma
_0)$. We consider also the system of Banach spaces of measurable
functions on $\Gamma_0$
\[
{\mathcal L}_{C,n}:=\biggl\{ G:\Gamma(\Lambda_n)\rightarrow{\mathbb
R} \biggm| \|G\|_{C,n}:= \int_{\Gamma({\Lambda_n})} |G(\eta)|
C^{|\eta|} d\lambda ( \eta) <\infty\biggr\}.
\]
Let $p_n:{\mathcal L}_C\rightarrow{\mathcal L}_{C,n}$ be a cut-off
mapping, namely, for any $G\in{\mathcal L}_C$
\[
(p_n G)(\eta) = 1\!\!1_{\Gamma(\Lambda_n)} (\eta) G(\eta).
\]
Then, obviously, $\|p_n G\|_{C,n}\leq \|G\|_C$. Hence,
$p_n:{\mathcal L}_C\rightarrow{\mathcal L}_{C,n}$ is a linear
bounded transformation with $\|p_n\|=1$.

\begin{proposition}\label{approxn}
Let \eqref{smallparam} hold. Then for any $G\in{\mathcal L}_C$
\[
\bigl\| \bigl(T_n\bigr)^{[nt]} p_n G - p_n \widehat{T}_t
G\bigr\|_{C,n}\rightarrow 0, \quad n\rightarrow \infty
\]
for all $t\geq0$ uniformly on bounded intervals.
\end{proposition}
\begin{proof}
The proof of the proposition is completed by showing that all
conditions of Lemma~\ref{EK_res-conv} hold. Using completely the
same arguments as in the proof of Proposition~\ref{prop_contr} one
gets that each $T_n=\widehat{P}_{\frac{1}{n}}^{\Lambda_n}$ is a linear
contraction on ${\mathcal L}_{C,n}$, $n\geq 2$ (note that for any
$n\geq 2$, \eqref{integrability} implies $\int_{\Lambda
_n}\bigl(1-e^{-\phi(x)}\bigr)dx\leq C_\phi<\infty$). Next, we set
$A_n=n (T_n - 1\!\!1_n)$ where $1\!\!1_n$ is a unit operator on
${\mathcal L}_{C,n}$ and let us expand $T_n$ in three parts
analogously to the proof of Proposition~\ref{prop_ineq}:
$T_n=T_n^{(0)}+T_n^{(1)}+T_n^{(\geq2)}$. As a result, $A_n=n
(T_n^{(0)} - 1\!\!1_n) + nT_n^{(1)}+nT_n^{(\geq2)}$. For any $G\in
{\mathcal L}_{2C}$ we set $G_n=p_n G\in{\mathcal
L}_{2C,n}\subset{\mathcal L}_{C,n}$. To finish the proof we have to
verify that for any $G\in {\mathcal L}_{2C}$
\begin{equation}\label{safcond}
\| A_n G_n - p_n \widehat{L}G \|_{C,n} \rightarrow 0, \quad n\rightarrow
\infty.
\end{equation}
For any $G\in {\mathcal L}_{2C}$
\begin{align}\label{innn}
\| A_n G_n - p_n LG \|_{C,n}  \leq & \, \| n(T_n^{(0)} - 1\!\!1_n) G_n -
p_n L_0 G \|_{C,n}\\& + \| nT_n^{(1)} G_n - p_n L_1 G \|_{C,n}
+ \| nT_n^{(\geq2)} G_n \|_{C,n}.\notag
\end{align}
Note, that $p_n L_0 G =L_0 G_n$. Using the same arguments as in the
proof of Proposition~\ref{prop_ineq} we obtain
\begin{equation*}
\| n(T_n^{(0)} - 1\!\!1_n) G_n - p_n L_0 G \|_{C,n} + \|
nT_n^{(\geq2)} G_n \|_{C,n} \leq \frac{2}{n} \|G\|_{2C,n} \leq
\frac{2}{n} \|G\|_{2C}.
\end{equation*}
Next,
\begin{align*}
& \| nT_n^{(1)} G_n - p_n L_1 G \|_{C,n}\\ \leq &\, z\int_{\Gamma
_{\Lambda_n}}\sum_{\xi \subset \eta }\int_{{{\mathbb R}^d}}
\left\vert\left( 1-\frac{1}{n} \right) ^{| \xi | } 1\!\!1_{\Lambda
_n} (x) -1\right\vert |G\left( \xi \cup x\right)| \\ & \times
\prod\limits_{y\in \xi }e^{-\phi \left( y-x\right)
}\prod\limits_{y\in \eta \setminus \xi }\left( 1- e^{-\phi \left(
y-x\right) }\right) dx  C^{\left\vert \eta \right\vert }d\lambda
\left( \eta \right)\\ \leq &\, z\int_{\Gamma(\Lambda_n)}\int_{\Gamma
({\Lambda_n})} \int_{{{\mathbb R}^d}} \left[ 1- \left( 1-\frac{1}{n}
\right) ^{| \xi | } 1\!\!1_{\Lambda_n} (x) \right] |G\left( \xi \cup
x\right)| \\ & \times \prod\limits_{y\in \eta }\left( 1-
e^{-\phi \left( y-x\right) }\right) dx C^{\left\vert \eta \cup \xi
\right\vert }d\lambda \left( \eta \right)d\lambda \left( \xi
\right)\\ \leq & \, C\int_{\Gamma(\Lambda_n)} \int_{{{\mathbb R}^d}}
\left[ 1- \left( 1-\frac{1}{n} \right) ^{| \xi | } 1\!\!1_{\Lambda
_n} (x) \right] |G\left( \xi \cup x\right)| dx C^{\left\vert \xi
\right\vert }d\lambda \left(
\xi \right),\\
\intertext{where we have used \eqref{integrability} and
\eqref{smallparam}. Using the same estimates as for \eqref{spec1} we
may continue} \leq &\, C\int_{\Gamma(\Lambda_n)} \int_{\Lambda_n}
\left[ 1- \left( 1-\frac{1}{n} \right) ^{| \xi | }  \right] |G\left(
\xi \cup x\right)| dx C^{\left\vert \xi \right\vert }d\lambda \left(
\xi \right) \\& + C\int_{\Gamma(\Lambda_n)} \int_{\Lambda
_n^c} |G\left( \xi \cup x\right)| dx
C^{\left\vert \xi \right\vert }d\lambda \left( \xi \right)\\
\leq &\, \frac{1}{n}\|G\|_{2C,n}+ C\int_{\Gamma_{0}} \int_{\Lambda
_n^c} |G\left( \xi \cup x\right)| dx C^{\left\vert \xi \right\vert
}d\lambda \left( \xi \right).
\end{align*}
But by the Lebesgue dominated convergence theorem,
\[
\int_{\Gamma_{0}} \int_{\Lambda_n^c} |G\left( \xi \cup x\right)| dx
C^{\left\vert \xi \right\vert }d\lambda \left( \xi \right)
\rightarrow 0, \quad n\rightarrow \infty.
\]
Indeed, $1\!\!1_{\Lambda_n^c} (x) |G\left( \xi \cup x\right)|
\rightarrow 0 $ point-wisely and may be estimated on $\Gamma
_0\times{{\mathbb R}^d}$ by $|G\left( \xi \cup x\right)|$ which is
integrable:
\[
C \int_{\Gamma_{0}} \int_{{{\mathbb R}^d}} |G\left( \xi \cup
x\right)| dx C^{\left\vert \xi \right\vert }d\lambda \left( \xi
\right) =\int_{\Gamma_0} |\xi| |G(\xi)| C^{\left\vert \xi
\right\vert }d\lambda \left( \xi \right) \leq \|G\|_{2C}<\infty.
\]
Therefore, by \eqref{innn}, the convergence \eqref{safcond} holds
for any $G\in{\mathcal L}_{2C}$, which completes the proof.
\end{proof}

\subsection{Evolution of correlation functions}

Under condition \eqref{verysmallparam}, we proceed now to the same arguments as in Subsection~\ref{subsect-evol-cf}. Namely, one can construct the restriction ${\widehat{T}}^\odot(t)$ of the semigroup of ${\widehat{T}}^\ast(t)$ onto the Banach space $\overline{D(%
{\widehat{L}}^\ast)}$ (recall that the closure is in the norm of ${\mathcal{K}}_C$).
Note that the domain of the dual operator to $(\widehat{L},\L_{2C})$ might be bigger than the domain considered in Subsection~\ref{subsect-evol-cf}. Nevertheless, ${\widehat{T}}^\odot(t)$ will be a $C_0$-semigroup on $\overline{D(%
{\widehat{L}}^\ast)}$ and its generator ${\widehat{L}} ^\odot$ will be a part of $%
{\widehat{L}}^\ast$, namely, \eqref{domLsundual} holds and ${\widehat{L}
}^\ast k ={\widehat{L}}^\odot k$ for any $k\in D({\widehat{L}}^\odot)$.

The next statement is a straightforward consequence of Proposition~\ref{pr3}.
\begin{proposition}
\label{domain-adj} For any $\a\in(0;1)$ the following inclusions
hold ${\mathcal{K}}_{\a C}\subset D({\widehat{L}}^\ast)\subset
\overline{D({\widehat{L}}^\ast)} \subset{\mathcal{K}}_C$.
\end{proposition}

Then, by Proposition~\ref{propLtriangle}, we immediately obtain that, for $k\in{\mathcal{K}}_{\a C}$,
\begin{align}  \label{dual-descent}
({\widehat{L}}^* k)(\eta)=&-\vert \eta \vert k(\eta) \\
&+z\sum_{x\in \eta}e^{-E^\phi (x,\eta\setminus x)} \int_{\Gamma_0}e_\lambda
(e^{-\phi (x - \cdot)}-1,\xi) k((\eta\setminus x)\cup\xi)\,d\lambda (\xi).
\notag
\end{align}

The next statement is an analog of Proposition~\ref{hint}.
\begin{proposition}
\label{invariantspace} Suppose that  (\ref{verysmallparam}) is satisfied.  Furthermore, we additionally assume that
\begin{equation}\label{new_z}
z< C e^{-CC_\phi},\quad\mathrm{if}\quad CC_\phi\leq \ln2.
\end{equation}
Then there exists $\a_0=\a_0(z,\phi,C)\in (0;1)$ such that for any $\a\in (\a_0;1)$ the set
${\mathcal{K}}_{\a C}$ is the ${\widehat{T}}^\ast(t)$-invariant linear subspace of
${\mathcal{K}}_C$.
\end{proposition}

\begin{proof}
Let us consider function $f(x):=x e^{-x}$, $x\geq 0$. It has the following properties:  $f$ is increasing on $[0; 1]$ from $0$ to $e^{-1}$ and it is
asymptotically decreasing on $[1;+\infty)$ from $e^{-1}$ to $0$; $f(x) < f(2x)$ for $x \in (0, \ln 2)$; $x=\ln 2$ is the only  non-zero solution to
$f(x)=f(2x)$.

By assumption \eqref{verysmallparam}, $zC_\phi \leq \min\{CC_\phi e^{-CC_\phi},
2CC_\phi e^{-2CC_\phi}\}$. Therefore, if $CC_\phi e^{-CC_\phi}\neq 2CC_\phi e^{-2CC_\phi}$ then \eqref{verysmallparam} with necessity implies
\begin{equation}\label{less_e-1}
z C_\phi < e^{-1}.
\end{equation}
This inequality remains also true if $CC_\phi=\ln 2$ because of \eqref{new_z}.
Under condition \eqref{less_e-1}, the equation $f(x)=z C_\phi$ has
exactly two roots, say, $0<x_1<1<x_2<+\infty$. Then,
\eqref{new_z} implies $x_1< C C_\phi < 2 C C_\phi \leq x_2$.

If $CC_\phi>1$ then we  set $\a_0:=\max\left\{\frac{1}{2};\frac{1}{C
C_\phi};\frac{1}{C}\right\}<1$. This yields $2\a C C_\phi > C
C_\phi$ and $\a CC_\phi >1>x_{1}$. If $x_{1}<CC_\phi \leq 1$ then we
set $\a_0:=\max\left\{\frac{1}{2};\frac{x_{1}}{C
C_\phi};\frac{1}{C}\right\}<1$ that gives $2\a C C_\phi > C C_\phi$
and $\a CC_\phi >x_{1}$.

As a result,
\begin{equation}\label{ineq-alpha}
x_1<\a CC_\phi  < CC_\phi <2 \a C C_\phi < 2 C C_\phi \leq x_2
\end{equation}
and $1<\a C<C<2\a C<2C$. The last inequality shows that $\L_{2C
}\subset\L_{2\a C}\subset \L_C\subset \L_{\a C}$. Moreover, by
\eqref{ineq-alpha}, we may prove that the operator $(\hL , \L_{2\a
C})$ is closable in $\L_{\a C}$ and its closure is a generator of a
contraction semigroup $\hT_\a (t)$ on $\L_{\a C}$. The proof is
identical to the proofs above.

It is easy to see, that $\hT_\a (t) G= \hT (t) G$ for any
$G\in\L_C$. Indeed, from the construction of the semigroup $\hT
(t)$ and analogous construction for the semigroup
$\hT_\a (t)$, we have that there exists family of mappings
$\hP_\delta$, $\delta>0$ independent of $\a $ and $C$, given by
\eqref{apprsemigroup},
such that $\hP _\delta^{\left[\frac{t}{\delta}\right]}$ for any
$t\geq 0$ strongly converges to $\hT (t)$ and $\hT_\a (t)$ in $\L_C$
and $\L_{\a C}$, correspondingly, as $\delta\goto 0$. Here and below
$[\,\cdot\,]$ means the entire part of a number. Then for any
$G\in\L_{C }\subset\L_{\a C}$ we have that $\hT (t)G\in\L_{C
}\subset\L_{\a C}$ and $\hT_\a (t) G\in\L_{\a C}$ and
\begin{align*}
\| \hT (t)G-\hT_\a (t) G\|_{\a C}&\leq \Bigl\| \hT (t)G-\hP
_\delta^{\left[\frac{t}{\delta}\right]}G\Bigr\|_{\a C} + \Bigl\| \hT_\a (t)
G-\hP_\delta^{\left[\frac{t}{\delta}\right]}G\Bigr\|_{\a C}\\&\leq \Bigl\| \hT
(t)G-\hP_\delta^{\left[\frac{t}{\delta}\right]}G\Bigr\|_{ C} + \Bigl\| \hT_\a
(t) G-\hP_\delta^{\left[\frac{t}{\delta}\right]}G\Bigr\|_{\a C}\goto0,
\end{align*}
as $\delta\goto 0$. Therefore, $\hT (t)G=\hT_\a (t) G$ in $\L_{\a
C}$ (recall that $G\in\L_C$) that yields $ \hT (t)G(\eta)=\hT_\a (t)
G(\eta)$ for $\la$-a.a. $\eta\in\Ga_0$ and, therefore, $\hT
(t)G=\hT_\a (t) G$ in $\L_{C}$.

Note that for any $G\in\L_C\subset\L_{\a C}$ and for any $k\in
\K_{\a C}\subset \K_C$ we have $\hT_\a (t) G\in\L_{\a C}$ and
\begin{equation*}
\lluu   \hT_\a (t) G, k\rruu  =\lluu   G, \hT^\ast_\a (t) k\rruu ,
\end{equation*}
where, by construction, $\hT^\ast_\a (t) k\in\K_{\a C}$. But
$G\in\L_C$, $k\in\K_C$ implies
\begin{equation*}
\lluu   \hT_\a (t) G, k\rruu  =\lluu   \hT (t) G, k\rruu  =\lluu G,
\hT^\ast(t) k\rruu .
\end{equation*}
Hence, $\hT^\ast(t) k = \hT^\ast_\a (t) k\in\K_{\a C}$, $k\in\K_{\a C}$ that proves
the statement.
\end{proof}

\begin{remark}
As a result, \eqref{verysmallparam} implies that for any $k_0\in \overline{D({%
\widehat{L}}^\ast)}$ the Cauchy problem in ${\mathcal{K}}_C$
\begin{equation}
\begin{cases}
\dfrac{\partial}{\partial t} k_t = {\widehat{L}}^\ast k_t \\\label{cau1}
k_t \bigr|_{t=0} = k_0%
\end{cases}%
\end{equation}
has a unique mild solution: $k_t= {\widehat{T}}^\ast (t)k_0= {\widehat{T}}^\odot
(t)k_0\in\overline{D({\widehat{L}}^\ast)}$. Moreover, $k_0\in{\mathcal{K}}_{\a C}
$ implies $k_t\in{\mathcal{K}}_{\a C}$ provided \eqref{new_z} is satisfied.
\end{remark}

\begin{remark}
The Cauchy problem \eqref{cau1} is well-posed in $\mathring{\K}_C=\overline{D(%
{\widehat{L}}^\ast)}$, i.e., for every $k_0\in D({\widehat{L}}^\odot)$ there exists a unique solution $k_{t}\in\mathring{\K}_C$ of  \eqref{cau1}.
\end{remark}

Let \eqref{verysmallparam} and \eqref{new_z} be satisfied and let $\alpha_0$ be chosen as in the proof of Proposition~\ref%
{invariantspace} and fixed. Suppose that  $\alpha\in(\alpha_0;1)$. Then,
Propositions~\ref{domain-adj} and \ref{invariantspace} imply
$\overline{{\mathcal{K}}_{\a C}}\subset\overline{D(\widehat{L}^{\ast
})}$ and the Banach subspace $\overline{{\mathcal{K}}_{\a C}}$ is
${\widehat{T}}^\ast(t)$- and, therefore, ${\widehat{T}}^\odot(t)$-invariant
due to the continuity of these operators.

Let now ${\widehat{T}}^{\odot\a}(t)$ be the restriction of the strongly
continuous semigroup ${\widehat{T}}^\odot(t)$
onto the closed linear
subspace $\overline{{\mathcal{K}}_{\a C}}$.
By general result (see, e.g.,
\cite{EN2000}), ${\widehat{T}} ^{\odot\a}(t)$ is a strongly continuous
semigroups on $\overline{{\mathcal{K}}_{\a C}}$ with generator
${\widehat{L}}^{\odot\a}$ which is the restriction of the operator
${\widehat{L}}^\odot $. Namely,
\begin{equation}
D({\widehat{L}}^{\odot\a})=\Bigl\{k\in \overline{{\mathcal{K}}_{\a C}} \Bigm| {%
\widehat{L}}^\ast k\in\overline{{\mathcal{K}}_{\a C}} \Bigr\},
\label{domAdjtimes}
\end{equation}
and
\begin{equation}
{\widehat{L}}^{\odot\a} k = {\widehat{L}}^\odot k = {\widehat{L}}^\ast k, \qquad k\in D({%
\widehat{L}}^{\odot\a})  \label{restRRren}
\end{equation}

Since ${\widehat{T}}(t)$ is a contraction semigroup on $\L _C$, then, ${\widehat{T}}
^{\prime }(t)$ is also a contraction semigroup on $(\L _C)^{\prime }$; but
isomorphism \eqref{isometry} is isometrical, therefore, ${\widehat{T}}^\ast(t)$
is a contraction semigroup on ${\mathcal{K}}_C$. As a result, its restriction $%
{\widehat{T}}^{\odot\a}(t)$ is a contraction semigroup on $\overline{{\mathcal{K}}_%
{\a C}}$. Note also, that by \eqref{domAdjtimes},
\begin{equation*}
D_{\a C}:=\Bigl\{k\in {\mathcal{K}}_{\a C} \Bigm| {\widehat{L}}^\ast k\in%
\overline{{\mathcal{K}}_{\a C}} \Bigr\}
\end{equation*}
is a core for ${\widehat{L}}^{\odot\a}$ in $\overline{{\mathcal{K}}_{\a C}}$.

By \eqref{apprsemigroup}, for any $k\in {\mathcal{K}}_{{\a C}}$, $G\in B_{\mathrm{bs}}(\Gamma
_{0})$ we have
\begin{align*}
& \int_{\Gamma _{0}}(\widehat{P}_{\delta }G)\left( \eta \right) k\left( \eta
\right) d\lambda \left( \eta \right)  \\
=& \int_{\Gamma _{0}}\sum_{\xi \subset \eta }\left( 1-\delta \right)
^{\left\vert \xi \right\vert }\int_{\Gamma _{0}}\left( z\delta \right)
^{\left\vert \omega \right\vert }G\left( \xi \cup \omega \right)
\prod\limits_{y\in \xi }e^{-E^{\phi }\left( y,\omega \right) } \\
& \times \prod\limits_{y\in \eta \setminus \xi }\left( e^{-E^{\phi }\left(
y,\omega \right) }-1\right) d\lambda \left( \omega \right) k\left( \eta
\right) d\lambda \left( \eta \right)  \\
=& \int_{\Gamma _{0}}\int_{\Gamma _{0}}\left( 1-\delta \right) ^{\left\vert
\xi \right\vert }\int_{\Gamma _{0}}\left( z\delta \right) ^{\left\vert
\omega \right\vert }G\left( \xi \cup \omega \right) \prod\limits_{y\in \xi
}e^{-E^{\phi }\left( y,\omega \right) } \\
& \times \prod\limits_{y\in \eta }\left( e^{-E^{\phi }\left( y,\omega
\right) }-1\right) d\lambda \left( \omega \right) k\left( \eta \cup \xi
\right) d\lambda \left( \xi \right) d\lambda \left( \eta \right)  \\
=& \int_{\Gamma _{0}}\int_{\Gamma _{0}}\sum_{\omega \subset \xi }\left(
1-\delta \right) ^{\left\vert \xi \setminus \omega \right\vert }\left(
z\delta \right) ^{\left\vert \omega \right\vert }G\left( \xi \right)
\prod\limits_{y\in \xi \setminus \omega }e^{-E^{\phi }\left( y,\omega
\right) } \\
& \times \prod\limits_{y\in \eta }\left( e^{-E^{\phi }\left( y,\omega
\right) }-1\right) k\left( \eta \cup \xi \setminus \omega \right) d\lambda
\left( \xi \right) d\lambda \left( \eta \right) ,
\end{align*}%
therefore,%
\begin{align}\label{apprfordual}
(\widehat{P}_{\delta }^{\ast }k)\left( \eta \right)  =&\sum_{\omega \subset \eta
}\left( 1-\delta \right) ^{\left\vert \eta \setminus \omega \right\vert
}\left( z\delta \right) ^{\left\vert \omega \right\vert }\prod\limits_{y\in
\eta \setminus \omega }e^{-E^{\phi }\left( y,\omega \right) } \\
& \times \int_{\Gamma _{0}}\prod\limits_{y\in \xi }\left( e^{-E^{\phi
}\left( y,\omega \right) }-1\right) k\left( \xi \cup \eta \setminus \omega
\right) d\lambda \left( \xi \right) .\notag
\end{align}

\begin{proposition}\label{imp_prop}
Suppose that \eqref{verysmallparam} and \eqref{new_z} are fulfilled. Then, for any $k\in D_{\a C}$ and $\alpha\in(\alpha_{0},\,1)$, where $\alpha_0$ is chosen as in the proof of Proposition~\ref{invariantspace},
\begin{equation}  \label{apprrenest}
\lim_{\delta\rightarrow 0}\biggl\Vert \frac{1}{\delta}( {\widehat{P}}
^\ast_{\delta}-1\!\!1 )k - {\widehat{L}}^{\odot\a} k\biggr\Vert_{{\mathcal{K}}_C}
=0.
\end{equation}
\end{proposition}

\begin{proof}
Let us recall \eqref{dual-descent} and define
\begin{align*}
(\hP_{\delta }^{\ast, (0) }k)\left( \eta \right) =&\,(1-\delta)^\n k(\eta);\\
(\hP_{\delta }^{\ast, (1) }k)\left( \eta \right) =&\,z\delta
\sum_{x\in \eta
}\left( 1-\delta \right)^{\left\vert \eta \right\vert -1} e_\la\left( e^{-\phi \left( x -\cdot\right) },\eta
\setminus x \right) \notag\\&\times\int_{\Gamma_{0}}e_{\lambda
}\left( e^{-\phi \left( x -\cdot\right)
}-1,\xi \right) k\left( \xi \cup \eta \setminus x
\right) d\lambda \left( \xi \right);
\end{align*}
and $\hP_{\delta }^{\ast, (\geq 2) } = \hP_{\delta }^{\ast } -
\hP_{\delta }^{\ast, (0) } -  \hP_{\delta }^{\ast, (1)
}$.

We will use the following elementary inequality, for any
$n\in\N\cup\{0\}$, $\delta\in(0;1)$
\begin{align*}
0 \leq n-   \frac{1-(1-\delta)^n}{\delta}\leq\delta \frac{n(n-1)}{2}.
\end{align*}
Then, for any $k\in\K_\aC$ and $\lambda$-a.a. $\eta\in\Gamma_{0}$, $\eta\neq\emptyset$
\begin{align}
&C^{-\en}\biggl\vert\frac{1}{\delta}( \hP^{\ast,(0)}_{\delta,\eps}-\1 )k(\eta) + |\eta|k(\eta)\biggr\vert\notag\\
\leq& \,\Vert k\Vert_{\K_\aC} \a^\en \biggl\vert \en-
\frac{1-(1-\delta)^\en}{\delta}\biggr\vert\leq \frac{\delta}{2} \Vert
k\Vert_{\K_\aC} \a^\en \en (\en-1)\label{eq1}
\end{align}
and the function $\a^x x(x-1)$ is bounded for $x\geq 1$, $\a\in(0;1)$.
Next, for any $k\in\K_\aC$ and $\lambda$-a.a. $\eta\in\Gamma_{0}$, $\eta\neq\emptyset$
\begin{align}
&C^{-\en}\biggl\vert\frac{1}{\delta} \hP ^{\ast,(1)}_{\delta}k(\eta)
-z\sum_{x\in\eta}\int_{\Ga_0} e_{\la  }\left( e^{- \phi \left(
x-\cdot \right) },\eta\setminus x \right)
\notag\\&\qquad\qquad\times e_{\la }\left( e^{-\phi \left( x-\cdot
\right) }-1,\xi   \right)
k\left( \xi \cup \eta\setminus x\right) d\la(\xi)\biggr\vert\notag\\
\leq\, & \Vert k\Vert_{\K_\aC}
\frac{z}{\aC}\a^\en\sum_{x\in\eta}\bigl(1-\left( 1-\delta
\right)^{\left\vert \eta \right\vert -1} \bigr)
\int_{\Gamma_{0}}e_{\lambda }\left( \a C \bigl(e^{-\phi \left( x-\cdot \right) }-1\bigr),\xi \right)  d\lambda \left( \xi \right)\notag\\
 \leq\, & \Vert k\Vert_{\K_\aC} \frac{z}{\aC} \a^\en\sum_{x\in\eta}\bigl(1-\left( 1-\delta \right)^{\left\vert \eta
\right\vert -1} \bigr) \exp{\{\a C C_\phi\}}\notag\\
\leq\,& \Vert k\Vert_{\K_\aC}\frac{z}{\aC} \a^\en \delta \en (\en-1)
\exp{\{\a C C_\phi\}}.\label{eq2}
\end{align}
which is small in $\delta$ uniformly by $\en$. Now, using inequality
\[
1-e^{-E^\phi\left( y,\omega \right) }=1-\prod\limits_{x\in \omega
}e^{-\phi \left( x-y\right) }\leq \sum_{x\in \omega }\left(
1-e^{-\phi \left( x-y\right) }\right),
\]
we obtain
\begin{align}
&\frac{1}{\delta }C^{-\left\vert \eta \right\vert } \sum_{\substack{
\omega \subset \eta  \\ \left\vert \omega \right\vert \geq 2}}
\left( 1-\delta \right)^{\left\vert \eta \setminus \omega
\right\vert }\left( z\delta \right)^{\left\vert \omega \right\vert
}e_\la\left( e^{- E^{\phi }\left( \cdot ,\omega \right) },\eta
\setminus \omega \right)\notag\\&\qquad\times\int_{\Gamma_{0} }e_\la\left(
\Bigl\vert e^{- E^{\phi }\left( \cdot ,\omega \right) }-1
\Bigr\vert,\xi \right) |k( \xi
\cup \eta \setminus \omega ) | d\lambda \left( \xi \right) \notag \\
=\,&\Vert k\Vert_{\K_\aC}\alpha^{\left\vert \eta \right\vert }\frac{1}{\delta }\sum_{\substack{ %
\omega \subset \eta  \\ \left\vert \omega \right\vert \geq 2}}\left(
1-\delta \right)^{\left\vert \eta \setminus \omega \right\vert
}\left( \frac{z\delta }{\alpha C}\exp \left\{ \alpha C
C_\phi\right\} \right)
^{\left\vert \omega \right\vert }; \notag\\
\intertext{recall that $\a>\a_0$, therefore, $z\exp\{\aC
C_\phi\}\leq \aC$, and one may continue} \leq\, &\Vert k\Vert_{\K_\aC}
\alpha ^{\left\vert \eta \right\vert }\frac{1}{\delta }\sum
_{\substack{ \omega \subset \eta  \\ \left\vert \omega \right\vert \geq 2}}%
\left( 1-\delta \right)^{\left\vert \eta \setminus \omega
\right\vert }\delta^{\left\vert \omega \right\vert }\notag\\=\,&\Vert
k\Vert_{\K_\aC}\delta \alpha ^{\left\vert \eta \right\vert
}\sum_{k=2}^{\left\vert \eta \right\vert }\frac{\left\vert \eta
\right\vert !}{k!\left( \left\vert \eta \right\vert -k\right)
!}\left(
1-\delta \right)^{\left\vert \eta \right\vert -k}\delta^{k-2} \notag\\
=\, &\Vert k\Vert_{\K_\aC}\delta \alpha^{\left\vert \eta \right\vert
}\sum_{k=0}^{\left\vert \eta \right\vert -2}\frac{\left\vert \eta
\right\vert !}{\left( k+2\right) !\left( \left\vert \eta \right\vert
-k-2\right) !}\left( 1-\delta \right)
^{\left\vert \eta \right\vert -k-2}\delta^{k} \notag\\
=\, &\Vert k\Vert_{\K_\aC}\delta \alpha^{\left\vert \eta \right\vert
}\left\vert \eta \right\vert \left( \left\vert \eta \right\vert
-1\right) \sum_{k=0}^{\left\vert \eta \right\vert -2}\frac{\left(
\left\vert \eta \right\vert -2\right) !}{\left( k+2\right) !\left(
\left\vert \eta \right\vert -k-2\right) !}\left( 1-\delta
\right)^{\left\vert \eta \right\vert -2-k}\delta^{k} \notag\\
\leq\,  &\Vert k\Vert_{\K_\aC}\delta \alpha^{\left\vert \eta
\right\vert }\left\vert \eta \right\vert \left( \left\vert \eta
\right\vert -1\right) \sum_{k=0}^{\left\vert \eta \right\vert
-2}\frac{\left( \left\vert \eta
\right\vert -2\right) !}{k!\left( \left\vert \eta \right\vert -k-2\right) !}%
\left( 1-\delta \right)^{\left\vert \eta \right\vert -2-k}\delta
^{k}\notag\\=\,  &\Vert k\Vert_{\K_\aC}\delta \alpha^{\left\vert \eta
\right\vert }\left\vert \eta \right\vert \left( \left\vert \eta
\right\vert -1\right).\label{eq3}
\end{align}
Combining inequalities \eqref{eq1}--\eqref{eq3} we obtain
\eqref{apprrenest}.
\end{proof}

As a result, we obtain an approximation for the semigroup.

\begin{theorem}
Let $\alpha_0$ be chosen as in the proof of the Proposition~\ref%
{invariantspace} and be fixed. Let $\alpha\in(\alpha_0;1)$ and
$k\in\overline{{\mathcal{K} }_{\a C}}$ be given. Then
\begin{equation*}
({\widehat{P}} ^\ast_{\delta})^{[{t}/{\delta}]}k\rightarrow{\widehat{T}%
} ^{\odot\a}(t)k, \quad \delta\rightarrow 0
\end{equation*}
in the space $\overline{{\mathcal{K}}_{\a C}}$ with norm
$\|\cdot\|_{{\mathcal{K}}_C}$ for all $t\geq 0$ uniformly on bounded
intervals.
\end{theorem}
\begin{proof}
We may apply Proposition~\ref{imp_prop} to use
Lemma~\ref{EK_res-conv} in the case $L_n=L=\overline{\L_\aC}$,
$p_n=\1$, $f_n=f=k$, $\eps_n=\delta\rightarrow0$, $n\in\N$.
\end{proof}

\subsection{Positive definiteness}

We consider a small modification of the notion of positive definite functions considered in Proposition~\ref{exuniqmeas}. Namely, we denote by
$L_{\mathrm{ls}}^0(\Ga _0)$ the set of all measurable functions
on $\Ga_0$ which have a local support, i.e. $G\in
L_{\mathrm{ls}}^0(\Ga _0)$ if there exists $\La \in \B_b({\R}^{d})$
such that $G\upharpoonright_{\Ga _0\setminus \Ga (\La) }=0$. We will say that a measurable function $k:\Gamma_0\rightarrow{\mathbb{R}}$
is a positive defined function if, for any $G\in
L_{\mathrm{ls}}^0(\Ga _0)$ such that $KG\geq 0$ and $G\in\mathcal{L}_{C}$ for some $C>1$ the
inequality \eqref{Lenpos} holds.

For a given $C>1$, we set $\mathcal{L}_{C}^{\mathrm{ls}}= L_{\mathrm{ls}}^0(\Ga _0) \cap \mathcal{L}_{C}$. Since $\Bbs\subset\mathcal{L}_{C}^{\mathrm{ls}}$, for any $C>1$, Proposition~\ref{exuniqmeas} (see also the second part of Remark~\ref{remLen}) implies that if $k$ is a positive definite function as above then there exists a unique measure $\mu\in{%
\mathcal{M}}^1_{\mathrm{fm}}(\Gamma)$ such that $k=k_\mu$ be its
correlation function in the sense of \eqref{eqmeans}. Our aim is to show that the evolution
$k\mapsto \widehat{T}^{\odot }(t)k$ preserves this property of the positive definiteness.

\begin{theorem}
Let \eqref{verysmallparam} holds and $k\in \overline{D(\widehat{L}^{\ast })}%
\subset \mathcal{K}_{C}$ be a positive definite function. Then $k_{t}:=\widehat{T}^{\odot }(t)k\in \overline{D(\widehat{L}^{\ast })}%
\subset \mathcal{K}_{C}$ will be a positive definite function for any $t\geq0$.
\end{theorem}

\begin{proof}
Let $C>0$ be arbitrary and fixed.
For any $G\in \mathcal{L}_{C}^{\mathrm{ls}}$ we have
\begin{equation}\label{sdual}
\int_{\Gamma _{0}}G\left( \eta \right) k_{t}\left( \eta \right)
d\lambda \left( \eta \right) =\int_{\Gamma _{0}}(\widehat{T}(t)G)\left(
\eta \right) k\left( \eta \right) d\lambda \left( \eta \right) .
\end{equation}
By Proposition~\ref{approxn}, under condition
\eqref{verysmallparam}, we obtain that
\[
\lim_{n\rightarrow 0}\int_{\Gamma (\Lambda _{n})}\left\vert T_{n}^{\left[ nt%
\right] }\1_{\Gamma (\Lambda _{n})}G\left( \eta \right) -\1_{\Gamma
(\Lambda _{n})}(\eta)(\widehat{T}(t)G)\left( \eta \right) \right\vert
C^{\left\vert \eta \right\vert }d\lambda \left( \eta \right) =0,
\]%
where for $n\geq 2$%
\[
T_{n}=\widehat{P}_{\frac{1}{n}}^{\Lambda _{n}}
\]%
and $\La_n \nearrow\X.$ Note that, by the dominated convergence
theorem,
\begin{align*}
\int_{\Gamma _{0}}(\widehat{T}(t)G)\left( \eta \right) k\left( \eta
\right) d\lambda \left( \eta \right)  =&\lim_{n\rightarrow \infty
}\int_{\Gamma _{0}}\1_{\Gamma (\Lambda _{n})}\left( \eta \right)
(\widehat{T}(t)G)\left( \eta
\right) k\left( \eta \right) d\lambda \left( \eta \right)  \\
=&\lim_{n\rightarrow \infty }\int_{\Gamma (\Lambda _{n})}(\widehat{T}(t)G)\left( \eta \right) k\left( \eta \right) d\lambda \left( \eta
\right) .
\end{align*}%
Next,%
\begin{align*}
&\left\vert \int_{\Gamma (\Lambda _{n})}(\widehat{T}(t)G)\left( \eta
\right) k\left( \eta \right) d\lambda \left( \eta \right)
-\int_{\Gamma (\Lambda _{n})}T_{n}^{\left[ nt\right] }\1_{\Gamma
(\Lambda _{n})}G\left( \eta
\right) k\left( \eta \right) d\lambda \left( \eta \right) \right\vert  \\
\leq &\int_{\Gamma (\Lambda _{n})}\left\vert T_{n}^{\left[
nt\right]
}\1_{\Gamma (\Lambda _{n})}G\left( \eta \right) -\1_{\Gamma (\Lambda _{n})}(\eta)(%
\widehat{T}(t)G)\left( \eta \right) \right\vert k\left( \eta \right)
d\lambda
\left( \eta \right)  \\
\leq  & \, \|k\|_{\K_C}\int_{\Gamma (\Lambda _{n})}\left\vert T_{n}^{\left[
nt\right]
}\1_{\Gamma (\Lambda _{n})}G\left( \eta \right) -\1_{\Gamma (\Lambda _{n})}(\eta)(%
\widehat{T}(t)G)\left( \eta \right) \right\vert C^{\left\vert \eta
\right\vert }d\lambda \left( \eta \right) \rightarrow
0,~~n\rightarrow \infty .
\end{align*}%
Therefore,%
\begin{equation}\label{eq-dop}
\int_{\Gamma _{0}}(\widehat{T}(t)G)\left( \eta \right) k\left( \eta
\right) d\lambda \left( \eta \right) =\lim_{n\rightarrow \infty
}\int_{\Gamma (\Lambda _{n})}T_{n}^{\left[ nt\right] }\1_{\Gamma
(\Lambda _{n})}G\left( \eta \right) k\left( \eta \right) d\lambda
\left( \eta \right) .
\end{equation}
Our aim is to show that for any $G\in \mathcal{L}_{C}^{\mathrm{ls}}$ the inequality $KG\geq0$ implies
\[
\int_{\Gamma _{0}}G\left( \eta \right) k_{t}\left( \eta \right)
d\lambda \left( \eta \right) \geq 0.
\]
By \eqref{sdual} and \eqref{eq-dop}, it is enough to show that for any $m\in
\mathbb{N}$ and for any $G\in \mathcal{L}_{C}^{\mathrm{ls}}$ such that $KG\geq0$
the following inequality holds
\begin{equation}\label{eq-dop2}
\int_{\Gamma _{0}}\1_{\Gamma (\Lambda _{n})}T_{n}^{m}\1_{\Gamma
(\Lambda _{n})}G\left( \eta \right) k\left( \eta \right) d\lambda
\left( \eta \right) \geq 0, \quad m\in\N_0.
\end{equation}
The inequality \eqref{eq-dop2} is fulfilled  if only
\begin{equation}\label{dop12}
K\1_{\Gamma (\Lambda _{n})}T_{n}^{m}G_{n}\geq 0,
\end{equation}
where $G_{n}:=\1_{\Gamma (\Lambda _{n})}G$.
Note that
\begin{align}
\bigl( K\1_{\Gamma (\Lambda _{n})}T_{n}^{m}G_{n}\bigr) \left(
\gamma \right)  =&\sum_{\eta \Subset \gamma }\1_{\Gamma (\Lambda _{n})}\left( \eta
\right) \left( T_{n}^{m}G_{n}\right) \left( \eta \right)  \label{dop23}\\
=&\sum_{\eta \subset \gamma _{\Lambda _{n}}}\left(
T_{n}^{m}G_{n}\right) \left( \eta \right) =\left(
KT_{n}^{m}G_{n}\right) \left( \gamma _{\Lambda _{n}}\right) \notag
\end{align}
for any $m\in\N_0$. In particular,
\begin{equation}\label{dop123}
\left( KG_{n}\right) \left( \gamma \right) =\left( K\1_{\Gamma
(\Lambda _{n})}G\right) \left( \gamma \right) =\left( KG\right)
\left( \gamma _{\Lambda _{n}}\right) \geq 0.
\end{equation}

Let us now consider any $\tilde{G}\in \mathcal{L}_{C}^{\mathrm{ls}}$ (stress that $\tilde{G}$ is
not necessary equal to $0$ outside of $\Ga(\Lambda _{n})$) and
suppose that $\bigl( K \tilde{G}\bigr) \left( \gamma \right) \geq 0$
for any $\gamma \in \Gamma (\Lambda _{n})$. Then
\begin{align}\label{dop345}
&\bigl( KT_{n}\tilde{G}\bigr) \left( \gamma _{\Lambda _{n}}\right) = \bigl( K\widehat{P}_{\frac{1}{n}}^{\Lambda _{n}}\tilde{G}\bigr)
\left(
\gamma _{\Lambda _{n}}\right) =\bigl( P_{\frac{1}{n}}^{\Lambda _{n}}K\tilde{G}\bigr) \left(
\gamma
_{\Lambda _{n}}\right)  \\
=& \Bigl( \Xi _{\frac{1}{n}}^{\Lambda _{n}}\left( \gamma _{\Lambda
_{n}}\right) \Bigr) ^{-1}\sum_{\eta \subset \gamma _{\Lambda
_{n}}}\biggl(
\frac{1}{n}\biggr) ^{\left\vert \eta \right\vert }\biggl( 1-\frac{1}{n}%
\biggr) ^{\left\vert \gamma \setminus \eta \right\vert } \notag\\
&\times \int_{\Gamma (\Lambda _{n})}\biggl( \frac{z}{n}\biggr)
^{\left\vert \omega \right\vert }\prod_{y\in \omega }e^{-E^{\phi
}\left( y,\gamma \right) }\bigl( K\tilde{G}\bigr) \bigl( \left(
\gamma _{\Lambda _{n}}\setminus \eta
\right) \cup \omega \bigr) d\lambda \left( \omega \right)
\geq 0.\notag
\end{align}
By \eqref{dop123}, setting $\tilde{G}=G_n\in\mathcal{L}_{C}^{\mathrm{ls}}$ we obtain, because of \eqref{dop345}, $KT_nG_n \geq0$.
Next, setting $\tilde{G}=T_nG_n\in\mathcal{L}_{C}^{\mathrm{ls}}$ we obtain, by \eqref{dop345}, $KT_n^2 G_n \geq0$.
Then, using an induction mechanism, we obtain that%
\[
\left( KT_{n}^{m}G_{n}\right) \left( \gamma _{\Lambda _{n}}\right)
\geq 0, \quad m\in\N_0,
\]
that, by \eqref{dop12} and \eqref{dop23}, yields \eqref{eq-dop2}.
This completes the proof. \end{proof}

\subsection{Ergodicity}
Let $k\in\overline{{\mathcal{K}}_{\a C}}$ be such that $k(\emptyset)=0$
then, by \eqref{apprfordual}, $(\widehat{P}_{\delta }^{\ast }k)\left( \emptyset \right)=0$. Class of all such functions we denote by
$\K_\a^0$.
\begin{proposition}\label{propergod}
Assume that there exists $\nu\in(0;1)$ such that
\begin{equation}  \label{nu-verysmallparam}
z\leq \min\Bigl\{\nu Ce^{-CC_{\phi }} ; \, 2Ce^{-2CC_{\phi }}\Bigr\}.
\end{equation}
Let, additionally, $\a\in(\a_0;1)$,
where $\alpha_0$ is chosen as in the proof of the Proposition~\ref{invariantspace}.
Then for any $\delta \in(0;1)$ the following estimate holds
\begin{equation}\label{supercontraction}
\Bigl\| \widehat{P}_{\delta }^{\ast }\! \upharpoonright_{\K_\a^{0}} \Bigr\|\leq 1-(1-\nu)\delta.
\end{equation}
\end{proposition}
\begin{proof}
It is easily seen that for any $k\in\K_\a^0$ the following
inequality holds
\[
\left\vert k\left( \eta \right) \right\vert \leq \\1_{\left\vert
\eta \right\vert >0}\left\Vert k\right\Vert
_{\mathcal{K}_{C}}C^{\left\vert \eta \right\vert },
\quad\lambda\mathrm{-a.a.}\;\;\eta\in\Ga_0.
\]%
Then, using \eqref{apprfordual}, we have
\begin{align*}
&C^{-\left\vert \eta \right\vert }\left\vert (\widehat{P}_{\delta
}^{\ast
}k)\left( \eta \right) \right\vert  \\
\leq &C^{-\left\vert \eta \right\vert }\sum_{\omega \subset \eta
}\left( 1-\delta \right) ^{\left\vert \eta \setminus \omega
\right\vert }\left( z\delta \right) ^{\left\vert \omega \right\vert
}\int_{\Gamma _{0}}\prod\limits_{y\in \xi }\left( 1-e^{-E^{\phi
}\left( y,\omega \right) }\right) \left\vert k\left( \xi \cup \eta
\setminus \omega \right)
\right\vert d\lambda \left( \xi \right)  \\
\leq &\left\Vert k\right\Vert _{\mathcal{K}_{C}}\sum_{\omega \subset
\eta }\left( 1-\delta \right) ^{\left\vert \eta \setminus \omega
\right\vert }\left( \frac{z\delta }{C}\right) ^{\left\vert \omega
\right\vert }\int_{\Gamma _{0}}\prod\limits_{y\in \xi }\left(
1-e^{-E^{\phi }\left( y,\omega \right) }\right) C^{\left\vert \xi
\right\vert }\1_{\left\vert \xi \right\vert +\left\vert \eta
\setminus \omega \right\vert >0}d\lambda \left(
\xi \right)  \\
=&\left\Vert k\right\Vert _{\mathcal{K}_{C}}\sum_{\omega \subsetneq
\eta }\left( 1-\delta \right) ^{\left\vert \eta \setminus \omega
\right\vert }\left( \frac{z\delta }{C}\right) ^{\left\vert \omega
\right\vert }\int_{\Gamma _{0}}\prod\limits_{y\in \xi }\left(
1-e^{-E^{\phi }\left( y,\omega \right) }\right) C^{\left\vert \xi
\right\vert }d\lambda \left( \xi
\right)  \\
&+\left\Vert k\right\Vert _{\mathcal{K}_{C}}\left( \frac{z\delta }{C}%
\right) ^{\left\vert \eta \right\vert }\int_{\Gamma
_{0}}\prod\limits_{y\in \xi }\left( 1-e^{-E^{\phi }\left( y,\omega
\right) }\right) C^{\left\vert \xi \right\vert }\1_{\left\vert \xi
\right\vert >0}d\lambda \left( \xi
\right)  \\
=&\left\Vert k\right\Vert _{\mathcal{K}_{C}}\sum_{\omega \subsetneq
\eta }\left( 1-\delta \right) ^{\left\vert \eta \setminus \omega
\right\vert }\left( \frac{z\delta }{C}\right) ^{\left\vert \omega
\right\vert }\int_{\Gamma _{0}}\prod\limits_{y\in \xi }\left(
1-e^{-E^{\phi }\left( y,\omega \right) }\right) C^{\left\vert \xi
\right\vert }d\lambda \left( \xi
\right)  \\
&+\left\Vert k\right\Vert _{\mathcal{K}_{C}}\left( \frac{z\delta }{C}%
\right) ^{\left\vert \eta \right\vert }\int_{\Gamma
_{0}}\prod\limits_{y\in \xi }\left( 1-e^{-E^{\phi }\left( y,\omega
\right) }\right) C^{\left\vert
\xi \right\vert }d\lambda \left( \xi \right) -\left\Vert k\right\Vert _{%
\mathcal{K}_{C}}\left( \frac{z\delta }{C}\right) ^{\left\vert \eta
\right\vert } \\
=&\left\Vert k\right\Vert _{\mathcal{K}_{C}}\sum_{\omega \subset
\eta }\left( 1-\delta \right) ^{\left\vert \eta \setminus \omega
\right\vert }\left( \frac{z\delta }{C}\right) ^{\left\vert \omega
\right\vert }\int_{\Gamma _{0}}\prod\limits_{y\in \xi }\left(
1-e^{-E^{\phi }\left( y,\omega \right) }\right) C^{\left\vert \xi
\right\vert }d\lambda \left( \xi
\right) \\&-\left\Vert k\right\Vert _{\mathcal{K}_{C}}\left( \frac{z\delta }{C}%
\right) ^{\left\vert \eta \right\vert } \\
=&\left\Vert k\right\Vert _{\mathcal{K}_{C}}\sum_{\omega \subset
\eta }\left( 1-\delta \right) ^{\left\vert \eta \setminus \omega
\right\vert }\left( \frac{z\delta }{C}\right) ^{\left\vert \omega
\right\vert }\exp \left\{ C\int_{\mathbb{R}^{d}}\left( 1-e^{-E^{\phi
}\left( y,\omega \right)
}\right) dy\right\} \\&-\left\Vert k\right\Vert _{\mathcal{K}_{C}}\left( \frac{%
z\delta }{C}\right) ^{\left\vert \eta \right\vert } \\
\leq &\left\Vert k\right\Vert _{\mathcal{K}_{C}}\sum_{\omega
\subset \eta }\left( 1-\delta \right) ^{\left\vert \eta \setminus
\omega \right\vert }\left( \frac{z\delta }{C}\right) ^{\left\vert
\omega \right\vert }\exp \left\{ CC_{\beta }\left\vert \omega
\right\vert \right\} -\left\Vert k\right\Vert
_{\mathcal{K}_{C}}\left( \frac{z\delta }{C}\right) ^{\left\vert
\eta \right\vert } \\
\leq &\left\Vert k\right\Vert _{\mathcal{K}_{C}}\sum_{\omega
\subset \eta }\left( 1-\delta \right) ^{\left\vert \eta \setminus
\omega \right\vert }\left( \nu  \delta \right) ^{\left\vert \omega
\right\vert }-\left\Vert k\right\Vert _{\mathcal{K}_{C}}\left(
\frac{z\delta }{C}\right) ^{\left\vert
\eta \right\vert } \\
=&\left\Vert k\right\Vert _{\mathcal{K}_{C}}\left( \left( 1-\left(
1-\nu  \right) \delta \right) ^{\left\vert \eta \right\vert }-\left(
\frac{z\delta
}{C}\right) ^{\left\vert \eta \right\vert }\right)  \\
=&\left\Vert k\right\Vert _{\mathcal{K}_{C}}\left( 1-\left( 1-\nu
\right) \delta -\frac{z\delta }{C}\right) \sum_{j=0}^{\left\vert
\eta \right\vert -1}\left( 1-\left( 1-\nu  \right) \delta \right)
^{\left\vert
\eta \right\vert -1-\left\vert j\right\vert }\left( \frac{z\delta }{C}%
\right) ^{j} \\
\leq &\left\Vert k\right\Vert _{\mathcal{K}_{C}}\left( 1-\left(
1-\nu  \right) \delta -\frac{z\delta }{C}\right)
\sum_{j=0}^{\left\vert \eta
\right\vert -1}\left( \frac{z\delta }{C}\right) ^{j} \\
=&\left\Vert k\right\Vert _{\mathcal{K}_{C}}\left( 1-\left( 1-\nu
\right) \delta -\frac{z\delta }{C}\right) \frac{1-\left( \frac{z\delta }{C}%
\right) ^{\left\vert \eta \right\vert }}{1-\frac{z\delta }{C}} \\
\leq &\left\Vert k\right\Vert _{\mathcal{K}_{C}}\left( 1-\left(
1-\nu
\right) \delta -\frac{z\delta }{C}\right) \frac{1}{1-\frac{z\delta }{C}} \\
= &\left\Vert k\right\Vert _{\mathcal{K}_{C}}\left(
1-\frac{\left( 1-\nu  \right) \delta }{1-\frac{z\delta }{C}}\right)
\leq\left\Vert k\right\Vert _{\mathcal{K}_{C}}\bigl( 1-\left( 1-\nu
\right) \delta \bigr),
\end{align*}
where we have used that, clearly, $z<\nu C<C$. The statement is
proved.
\end{proof}

\begin{remark}
Condition \eqref{nu-verysmallparam} is equivalent to \eqref{verysmallparam} and \eqref{new_z}.
\end{remark}

As it was mentioned in Example~\ref{ex-Gl}, under condition (cf. \eqref{less_e-1})
\begin{equation}\label{LAHT}
z C_\phi < (2e)^{-1},
\end{equation}
there exists (see, e.g., \cite{FKL2007} for details) a Gibbs measure $\mu$ on $\bigl(\Ga,
\B(\Ga)\bigr)$ corresponding to the potential $\phi\geq 0$ and
activity parameter $z$. We denote the corresponding correlation
function by $k_\mu$. The measure $\mu$ is reversible (symmetrizing) for the operator defined by \eqref{genGa} (see, e.g., \cite{FKL2007,KL2005}). Therefore, for any $F\in K\Bbs$
\begin{equation}\label{invGibbs}
\int_\Ga LF(\ga)d\mu(\ga)=0.
\end{equation}

\begin{theorem}\label{thmergod}
Let \eqref{LAHT} and \eqref{nu-verysmallparam} hold and let $\a\in(\a_0;1)$,
where $\alpha_0$ is chosen as in the proof of Proposition~\ref{invariantspace}. Let
$k_0\in\Ka$, $k_t={\widehat{T}} ^{\odot\a}(t)k_0$. Then for any $t\geq0$
\begin{equation}\label{ergodineq}
\|k_t -k_\mu\|_{\K_C}\leq e^{-(1-\nu)t} \|k_0 -k_\mu\|_{\K_C}.
\end{equation}
\end{theorem}

\begin{proof}
First of all, let us note that for any $\a\in(\a_0;1)$ the
inequality \eqref{ineq-alpha} implies $z\leq\a C \exp\{-\aC
C_\phi\}$. Hence $k_\mu(\eta)\leq (\aC)^{|\eta|}$, $\eta\in\Ga_0$.
Therefore, $k_\mu\in\K_\aC\subset\Ka\cap D(\widehat{L}^\ast)$. By
\eqref{invGibbs}, for any $G\in \Bbs$ we have
$\langle\!\langle\widehat{L}G, k_\mu\rangle\!\rangle=0$. It means that
$\widehat{L}^\ast k_\mu = 0$. Therefore, ${\widehat{L}}^{\odot\a}k_\mu=0$.
As a result,
 ${\widehat{T}} ^{\odot\a}(t)k_\mu=k_\mu$.
Let $r_0=k_0-k_\mu\in\Ka$. Then $r_0\in\K_a^0$ and
\begin{align*}
&\|k_t -k_\mu\|_{\K_C}= \bigl\| {\widehat{T}} ^{\odot\a}(t) r_0\bigr\|_{\K_C}
\\&\leq \Bigl\| \bigl(\widehat{P}_{\delta
}^{\ast
}\bigr)^{\left[\frac{t}{\delta}\right]} r_0\Bigr\|_{\K_C}+\Bigl\| {\widehat{T}} ^{\odot\a}(t) r_0- \bigl(\widehat{P}_{\delta
}^{\ast
}\bigr)^{\left[\frac{t}{\delta}\right]} r_0\Bigr\|_{\K_C}\\
&\leq \Bigl\| \widehat{P}_{\delta
}^{\ast
}\upharpoonright_{\K_\a^{0}}\Bigr\|^{\left[\frac{t}{\delta}\right]} \cdot \|r_0\|_{\K_C}+\Bigl\| {\widehat{T}} ^{\odot\a}(t) r_0- \bigl(\widehat{P}_{\delta
}^{\ast
}\bigr)^{\left[\frac{t}{\delta}\right]} r_0\Bigr\|_{\K_C}\\
&\leq \bigl( 1 -(1-\nu)\delta \bigr)^{\frac{t}{\delta}-1} \|r_0\|_{\K_C}+\Bigl\| {\widehat{T}} ^{\odot\a}(t) r_0- \bigl(\widehat{P}_{\delta
}^{\ast
}\bigr)^{\left[\frac{t}{\delta}\right]} r_0\Bigr\|_{\K_C},
\end{align*}
since $0<1-(1-\nu)\delta<1$ and
$\frac{t}{\delta}<\bigl[\frac{t}{\delta}\bigr]+1$. Taking the limit
as $\delta\downarrow 0$ in the right hand side of this inequality we
obtain \eqref{ergodineq}.
\end{proof}

\end{document}